\documentclass[10pt,reqno]{amsart}

\usepackage[utf8x]{inputenc}
\usepackage[english]{babel}
\usepackage{amsmath,amsfonts,amssymb,amsthm,shuffle}
\usepackage[T1]{fontenc}
\usepackage[math]{anttor}

\usepackage[top=3.5cm,bottom=3.5cm,left=3.6cm,right=3.6cm]{geometry}

\usepackage{xcolor}

\definecolor{ColBlack}{RGB}{0,0,0} 
\definecolor{ColWhite}{RGB}{255,255,255} 
\definecolor{ColAA}{HTML}{520db1} 
\definecolor{ColAB}{HTML}{1a34c0} 
\definecolor{ColAC}{HTML}{3851db} 
\definecolor{ColBA}{HTML}{a80b3a} 
\definecolor{ColBB}{HTML}{a80b27} 
\definecolor{ColBC}{HTML}{b10d0d} 

\usepackage[hyperindex=true,frenchlinks=true,colorlinks=true,citecolor=ColBB,
linkcolor=ColBA,urlcolor=ColBC,linktocpage,pagebackref=true]{hyperref}

\usepackage{tikz}
\usetikzlibrary{shapes}
\usetikzlibrary{fit}
\usetikzlibrary{decorations.pathmorphing}
\usetikzlibrary{arrows.meta}

\usepackage{mathtools}
\usepackage{dsfont}
\usepackage{wasysym}
\usepackage{stmaryrd}
\usepackage{cite}
\usepackage{subfig}
\usepackage{multirow}
\usepackage{enumitem}
\usepackage{multicol}
\usepackage{cancel}

\usepackage[color=ColAA!40]{todonotes}

\allowdisplaybreaks

\numberwithin{equation}{subsection}

\makeatletter
\def\l@section{\@tocline{1}{3pt}{1pc}{5pc}{}}
\def\l@subsection{\@tocline{2}{2pt}{2pc}{5pc}{}}
\makeatother

\newtheorem{Theorem}{Theorem}[subsection]
\newtheorem{Proposition}[Theorem]{Proposition}
\newtheorem{Lemma}[Theorem]{Lemma}

\renewcommand{\leq}{\leqslant}
\renewcommand{\geq}{\geqslant}

\newcommand{\ColAA}[1]{\textcolor{ColAA}{#1}}
\newcommand{\ColAB}[1]{\textcolor{ColAB}{#1}}

\newcommand{\Hide}[1]{\ColAA{\tt HIDEN}}
\newcommand{\Def}[1]{\ColAB{\em #1}}
\newcommand{\Par}[1]{\left(#1\right)}
\newcommand{\Bra}[1]{\left\{#1\right\}}

\newcommand{\OEIS}[1]{\href{http://oeis.org/#1}{{\bf #1}}}

\tikzstyle{Centering}=[{baseline={([yshift=-0.5ex]current bounding box.center)}}]

\tikzstyle{MarkAA}=[draw=ColAA!80,fill=ColAA!8]
\tikzstyle{MarkAB}=[draw=ColAB!80,fill=ColAB!8]
\tikzstyle{MarkAC}=[draw=ColAC!80,fill=ColAC!8]
\tikzstyle{MarkBA}=[draw=ColBA!80,fill=ColBA!8]
\tikzstyle{MarkBB}=[draw=ColBB!80,fill=ColBB!8]
\tikzstyle{MarkBC}=[draw=ColBC!80,fill=ColBC!8]

\tikzstyle{Node}=[circle,MarkAA,inner sep=1pt,minimum size=2mm,thick,font=\scriptsize]
\tikzstyle{BigNode}=[rectangle,MarkBA,inner sep=1pt,minimum size=4mm,thick,font=\scriptsize]
\tikzstyle{Edge}=[draw=ColBB!80,cap=round,thick,rounded corners=2.5pt]
\tikzstyle{Leaf}=[rectangle,MarkBC,inner sep=0pt,minimum size=1mm,thick]
\tikzstyle{NodeST}=[font=\scriptsize]

\tikzstyle{NodeGraph}=[circle,MarkAB,inner sep=1pt,minimum size=1.5mm,thick]
\tikzstyle{NodeLabeldGraph}=[font=\scriptsize,node distance=4mm]
\tikzstyle{NodeLabelGraph}=[font=\tiny,node distance=3mm]
\tikzstyle{BigNodeLabeldGraph}=[font=\small,node distance=7mm]
\tikzstyle{MarkedNodeGraph}=[NodeGraph,rectangle,draw=ColAB!90,fill=ColAB!50,
minimum size=1.5mm]
\tikzstyle{Marked2NodeGraph}=[NodeGraph,regular polygon,regular polygon sides=6,
draw=ColBA!90,fill=ColBA!50,minimum size=1.75mm]
\tikzstyle{EdgeGraph}=[ColBB!70,cap=round,thick]
\tikzstyle{MarkedEdgeGraph}=[EdgeGraph,ColAA!90,very thick]
\tikzstyle{Marked2EdgeGraph}=[EdgeGraph,ColBA!90,very thick]
\tikzstyle{EdgeLabel}=[midway,inner sep=1pt,fill=ColWhite!0,font=\tiny]
\tikzstyle{MapGraph}=[ColAA!100,draw,dashed,-{>[scale=1.0,length=6,width=5]}]
\tikzstyle{FaceXY}=[fill=ColAA,opacity=.1]
\tikzstyle{FaceXZ}=[fill=ColBA,opacity=.2]
\tikzstyle{FaceYZ}=[fill=ColBC,opacity=.2]
\tikzstyle{Face}=[FaceYZ]

\tikzstyle{PathNode}=[NodeGraph]
\tikzstyle{PathStep}=[EdgeGraph]
\tikzstyle{PathDiag}=[EdgeGraph,ColBA!30,dotted]

\tikzstyle{Map}=[ColBlack!100,draw,-{>[scale=1.5,length=4,width=5]}]
\tikzstyle{Injection}=[ColBlack!100,draw,{>[scale=1.5,length=4,width=5]}-{>[scale=1.5,
length=4,width=5]}]
\tikzstyle{MapEmbedding}=[Injection]
\tikzstyle{MapIsomorphism}=[Map,double]

\tikzstyle{LineGrid}=[very thin,dashed,draw=ColAC!40]
\tikzstyle{Grid}=[LineGrid]

\newcommand{\Z}{\mathbb{Z}}

\newcommand{\R}{\mathbb{R}}

\newcommand{\SetIntervalPoset}{\mathsf{IP}}
\newcommand{\cc}{\mathsf{CC}(n)}
\newcommand{\ccc}{\mathsf{CC}}
\newcommand{\dit}{\mathsf{TID}(n)}
\newcommand{\SetIntervalBinaryTrees}{\mathsf{int(T_2(n))}}
\newcommand{\ccs}{\mathsf{SCC}(n)}
\newcommand{\DiffIndexes}{\mathrm{D}}
\newcommand{\TamariOrder}{\mathsf{ta}}
\newcommand{\TamariIntervalOrder}{\mathsf{int(ta)}}
\newcommand{\SetBinaryTrees}{\mathsf{T_2}}

\newcommand{\Interval}{\mathsf{int}}
\newcommand{\IntervalPoset}{\pi}

\newcommand{\Leq}{\preccurlyeq}

\newcommand{\LatticeL}{\mathcal{L}}
\newcommand{\PosetP}{\mathcal{P}}

\newcommand{\Angle}[1]{\left\langle#1\right\rangle}
\DeclareMathOperator{\Covered}{\lessdot}
\newcommand{\DiffIndices}{\mathrm{D}}

\newcommand{\TreeT}{\mathfrak{t}}
\newcommand{\TreeS}{\mathfrak{s}}

\newcommand{\TreeA}{\mathfrak{a}}
\newcommand{\TreeB}{\mathfrak{b}}
\newcommand{\TreeC}{\mathfrak{c}}
\newcommand{\ForestF}{\mathfrak{f}}

\newcommand{\OutputWings}{\mathcal{O}}
\newcommand{\InputWings}{\mathcal{I}}

\newcommand{\CubicReal}{\mathfrak{C}}
\newcommand{\Volume}{\mathrm{vol}}

\newcommand{\SynVolume}{\mathrm{sv}}

\newcommand{\Inp}{\mathrm{in}}
\newcommand{\Outp}{\mathrm{out}}

\title[Geometric realizations of Tamari interval lattices via cubic coordinates]
    {Geometric realizations of Tamari interval lattices\\ via cubic coordinates}
\keywords{Tamari lattices; Tamari intervals; 
interval-posets; posets; geometric realizations; cubical complexes.}
\date{\today}
\author{Camille Combe}
\address{\scriptsize Institut de Recherche Mathématique Avancée
UMR 7501, Université de Strasbourg et CNRS,
7 rue René Descartes
67000 Strasbourg, France.}
\email{cami.comb@gmail.com}

\begin{document}

\begin{abstract}
We introduce cubic coordinates, which are integer words encoding intervals in the Tamari lattices.
Cubic coordinates are in bijection with interval-posets, themselves known to be in bijection with Tamari intervals.
We show that in each degree the set of cubic coordinates forms a lattice, isomorphic to the lattice of Tamari intervals. Geometric realizations are naturally obtained by placing cubic coordinates in space, highlighting some of their properties. We consider the cellular structure of these realizations. Finally, we show that the poset of cubic coordinates is shellable.
\end{abstract}

\maketitle

\tableofcontents

\section*{Introduction}
The Tamari lattices are partial orders having extremely rich combinatorial and algebraic properties. These partial orders are defined on the set of binary trees and rely on the rotation operation~\cite{Tam62}.
We are interested in the intervals of these lattices, meaning the pairs of comparable binary trees. Tamari intervals of size $n$ also form a lattice. The number of these objects is given by a formula that was proved by Chapoton~\cite{Cha06}:
\begin{equation}
\frac{2(4n+1)!}{(n+1)!(3n+2)!}.
\end{equation}

Strongly linked with associahedra, Tamari lattices have been recently generalized in many ways~\cite{BPR12,PRV17}. In this process, the number of intervals of these generalized lattices have also been enumerated through beautiful formulas~\cite{BMFPR12,FPR17}.
Many bijections between Tamari intervals of size $n$ and other combinatorial objects are known. For instance, a bijection with $3$-connected planar triangulations is presented by Bernardi and Bonichon in~\cite{BB09} (see also \cite{Fan18}). It has been proved by Châtel and Pons that Tamari intervals are in bijection with interval-posets of the same size~\cite{CP15}.
\smallbreak

We provide in this paper a new bijection with Tamari intervals, which is inspired by interval-posets. 
More precisely, we first build two words of size $n$ from the Tamari diagrams~\cite{Pal86} of a binary tree. If they satisfy a certain property of compatibility, we build a Tamari interval diagram from these two words. 
We show that Tamari interval diagrams and interval-posets are in bijection.
Then we propose a new encoding of Tamari intervals, by building $(n-1)$-tuples of numbers from Tamari interval diagrams. We call these tuples cubic coordinates. 
This new encoding has two obvious virtues: it is very 
compact and it 
gives a way of comparing in a simple manner two Tamari intervals, through a fast algorithm. 
On the other hand, some properties of Tamari intervals translate nicely in the setting of cubic coordinates. For instance, synchronized Tamari intervals~\cite{FPR17} become cubic coordinates with no zero entry.
Besides, cubic coordinates provide naturally a geometric realization of the lattice of Tamari intervals, by seeing them as space coordinates. Indeed, all cubic coordinates of size $n$ can be placed in the space $\mathbb{R}^{n-1}$. By drawing their cover relations, we obtain a directed graph. This gives us a realization of cubic coordinate lattices, which we call cubic realization. This realization leads us to many questions, in particular about the cells it contains.
We characterize these cells in a combinatorial way, and we deduce a formula to compute the volume of the cubic realization in the geometrical sense.
Another direction, more topological, involves the shellability of partial order. We show, drawing inspiration from the work of Björner and Wachs~\cite{BW96, BW97}, that the cubic coordinates poset is EL-shellable, and as a consequence its associated complex is shellable.
\smallbreak

This article is organized in three sections.
\smallbreak

The first section is dedicated to reminders about some definitions, such as binary trees, Tamari intervals and interval-posets, and sets out the conventions used. Because of its key role in this work, the bijection between Tamari intervals and interval-posets is also recalled in this section. 
\smallbreak

In the second section, we define Tamari interval diagrams and show that they are in bijection, size by size, with interval-posets.
We then define cubic coordinates and show that they are in bijection, size by size, with Tamari interval diagrams.
Using this two bijections, and after having endowed the set of cubic coordinates with a partial order, we show that there is a poset isomorphism between the poset of cubic coordinates and the poset of Tamari intervals. 
\smallbreak

As pointed out above, the poset of cubic coordinates can then be realized geometrically. This cubic realization and the cells that compose it are the object of the third section.
For each cell, we then associate a synchronized cubic coordinate, which is a cubic coordinate without letter $0$.
By relying upon this particular cubic coordinate, we give a formula to compute the volume of the cubic realization. 
Finally, we extend the result of Björner and Wachs on the Tamari posets to the Tamari interval posets, by showing that the cubic coordinate posets are EL-shellable.
\smallbreak

This article is a complete version of~\cite{Com19}. All the proofs are given and several new results are presented, such as the EL-shellability of cubic coordinate posets.
\smallbreak

\subsubsection*{General notations and conventions}
Throughout this article, for all words $u$, we denote by $u_i$ the $i$-th letter of $u$. For any integers $i$ and $j$, $[i, j]$ denotes the set $\{i, i + 1, \dots, j\}$. For any integer $i$, $[i]$ denotes the set $[1, i]$.
All posets considered in this article are finite.
\smallbreak

\section{Preliminaries}\label{Sec:préliminaires}

In this first section we provide some basic notions of combinatorics and the conventions used afterwards. For this, we recall the definitions of lattices, binary trees, Tamari intervals, and interval-posets. Also, we recall the bijection given in~\cite{CP15}.

\subsection{Posets and lattices}
A \Def{partially ordered set}, commonly called \Def{poset}, is a pair $(\PosetP,\Leq_{\PosetP})$. When the context is clear, we simply denote this pair by~$\PosetP$.
\smallbreak

When two elements $x$ and $y$ of $\PosetP$ satisfy $x \Leq_{\PosetP} y$, then we say that $x$ and $y$ are \Def{comparable}. Otherwise, they are \Def{incomparable}. 
\smallbreak

Let $x , y\in \PosetP$ such that $x \Leq_{\PosetP} y$ and $x \neq y$. The element $y$ \Def{covers} $x$, denoted by $x \Covered_{\PosetP} y$, for the partial order $\Leq_{\PosetP}$ if, for all $z \in \PosetP$ such that $x \Leq_{\PosetP} z \Leq_{\PosetP} y$, either $z=x$ or $z=y$. The binary relation $\Covered_{\PosetP}$ is called the \Def{covering relation} of the poset $\PosetP$. By a slight abuse of notation, the set of elements $(x,y)$ such that $x \Covered_{\PosetP} y$ is also denoted by $\Covered_{\PosetP}$.
\smallbreak

A \Def{maximal element} of $\PosetP$ is an element $x$ such that if there is $y \in \PosetP$ such that $x \Leq_\PosetP y$, then $y = x$. Likewise, a \Def{minimal element} of $\PosetP$ is an element $y$ such that if there is $x \in \PosetP$ such that $x \Leq_\PosetP y$, then $x = y$.
A poset $\PosetP$ is \Def{bounded} if it has a unique maximal element and a unique minimal element for $\Leq_\PosetP$. 
\smallbreak

Since a partial order is transitive, one can realize posets or lattices by knowing only covering relations. The natural way to realize posets is to draw their \Def{Hasse diagrams}, by drawing a edge between all $x$ and $y$ in $\PosetP$ such that $(x,y) \in \Covered_{\PosetP}$. For any $(x,y) \in \Covered_{\PosetP}$, we choose the convention to represent $x$ at the top and $y$ at the bottom in the Hasse diagrams. We will keep this convention for all realizations.
\smallbreak

Let $x,y\in \PosetP$, the \Def{join} between $x$ and $y$, denoted by $\vee_{\PosetP} (x,y)$ (or $x \vee_\PosetP y$), is defined by
\begin{equation}
\vee_{\PosetP} (x,y) := \mathrm{min}_{\Leq_{\PosetP}}\{z\in \PosetP~:~x\Leq_{\PosetP}z \mbox{ and } y\Leq_{\PosetP}z \}.
\end{equation}
The \Def{meet} between $x$ and $y$, denoted by $\wedge_{\PosetP} (x,y)$ (or $x \wedge_\PosetP y$), is defined by
\begin{equation}
\wedge_{\PosetP} (x,y) := \mathrm{max}_{\Leq_{\PosetP}}\{z\in \PosetP~:~z\Leq_{\PosetP}x \mbox{ and } z\Leq_{\PosetP}y \}.
\end{equation}
A poset $\PosetP$ is a \Def{join-semilattice} if for all $x,y\in \PosetP$, $\vee_{\PosetP} (x,y)$ exists. Likewise, a poset $\PosetP$ is a \Def{meet-semilattice} if for all $x,y\in \PosetP$, $\wedge_{\PosetP} (x,y)$ exists.
A poset $(\LatticeL,\Leq_{\LatticeL})$ is a \Def{lattice} if $\LatticeL$ is a join-semilattice and a meet-semilattice.
\smallbreak

Let $\PosetP$ be a poset and $u^{(1)}, u^{(2)} \in \PosetP$ such that $u^{(1)} \Leq_{\PosetP} u^{(2)}$. 
An \Def{interval} $(u^{(1)}, u^{(2)})$ is the set of all elements between $u^{(1)}$ and $u^{(2)}$. 
The set of intervals of $\PosetP$ is denoted by $\Interval(\PosetP)$.
The \Def{poset of intervals} of a poset $\PosetP$ is the poset on the set $\Interval(\PosetP)$ endowed with the partial order $\Leq_{\Interval(\PosetP)}$ defined, for all $(u^{(1)},u^{(2)}),(v^{(1)}, v^{(2)})\in \Interval(\PosetP)$, by
\begin{equation}\label{posetintervalle}
(u^{(1)}, u^{(2)}) \Leq_{\Interval(\PosetP)} (v^{(1)},v^{(2)}) \mbox{ if } u^{(1)}\Leq_{\PosetP}v^{(1)} \mbox{ and } u^{(2)}\Leq_{\PosetP} v^{(2)}.
\end{equation}
In the same way, for $(u^{(1)},u^{(2)}), (v^{(1)},v^{(2)})\in \Interval(\LatticeL)$ such that $(u^{(1)},u^{(2)})\Leq_{\Interval(\LatticeL)}(v^{(1)},v^{(2)})$, a covering relation for the partial order $\Leq_{\Interval(\LatticeL)}$ is defined.
\smallbreak

The property of being a lattice is preserved under this construction.

\begin{Proposition}\label{treillis}
If $(\LatticeL,\Leq_{\LatticeL})$ is a lattice, then $(\Interval(\LatticeL),\Leq_{\Interval(\LatticeL)})$ is a lattice.
\end{Proposition}

\begin{proof}
Let $(u^{(1)},u^{(2)}), (v^{(1)},v^{(2)})\in \Interval(\LatticeL)$. 
First, we have to show that $\vee_{\LatticeL}(u^{(1)},v^{(1)})\Leq_{\LatticeL} \vee_{\LatticeL}(u^{(2)},v^{(2)})$.
By the definition of the join, one has $u^{(2)} \Leq_{\LatticeL} \vee_{\LatticeL}(u^{(2)},v^{(2)})$ and $v^{(2)}\Leq_{\LatticeL} \vee_{\LatticeL}(u^{(2)},v^{(2)})$. Furthermore, since $u^{(1)}\Leq_{\LatticeL} u^{(2)}$ and $v^{(1)} \Leq_{\LatticeL} v^{(2)}$, one has $u^{(1)}\Leq_{\LatticeL} \vee_{\LatticeL}(u^{(2)},v^{(2)})$ and $v^{(1)} \Leq_{\LatticeL} \vee_{\LatticeL}(u^{(2)},v^{(2)})$. In addition, $\vee_{\LatticeL}(u^{(1)},v^{(1)})$ is the minimal element of $\LatticeL$ satisfying $u^{(1)}\Leq_{\LatticeL} \vee_{\LatticeL}(u^{(1)},v^{(1)})$ and $v^{(1)}\Leq_{\LatticeL} \vee_{\LatticeL}(u^{(1)},v^{(1)})$. Thus, $\vee_{\LatticeL}(u^{(1)},v^{(1)})\Leq_{\LatticeL} \vee_{\LatticeL}(u^{(2)},v^{(2)})$.
\smallbreak

From~\eqref{posetintervalle}, one has
\begin{equation}
\begin{split}
\begin{array}{l}
\vee_{\Interval(\LatticeL)} \left((u^{(1)},u^{(2)}),(v^{(1)},v^{(2)})\right) \\
= \mathrm{min}_{\Leq_{\Interval(\LatticeL)}} \{(w^{(1)},w^{(2)})\in\Interval(\LatticeL)~:~(u^{(1)},u^{(2)})\Leq_{\Interval(\LatticeL)}(w^{(1)},w^{(2)}), (v^{(1)},v^{(2)})\Leq_{\Interval(\LatticeL)}(w^{(1)},w^{(2)})\} \\
= \mathrm{min}_{\Leq_{\Interval(\LatticeL)}} \{(w^{(1)},w^{(2)})\in\Interval(\LatticeL)~:~u^{(1)} \Leq_{\LatticeL} w^{(1)}, u^{(2)} \Leq_{\LatticeL} w^{(2)}, v^{(1)}\Leq_{\LatticeL} w^{(1)}, v^{(2)}\Leq_{\LatticeL} w^{(2)}\} \\
= \left(\vee_{\LatticeL}(u^{(1)},v^{(1)}),\vee_{\LatticeL}(u^{(2)},v^{(2)})\right).
\end{array}
\end{split}
\end{equation}

The case of the meet $\wedge_{\Interval(\LatticeL)} \left((u^{(1)},u^{(2)}),(v^{(1)},v^{(2)})\right) = \left(\wedge_{\LatticeL}(u^{(1)},v^{(1)}),\wedge_{\LatticeL}(u^{(2)},v^{(2)})\right)$ is symmetrical.
\end{proof}

\subsection{Rooted trees and binary trees}\label{Subsec:binarytrees}
A \Def{rooted tree}, or simply a \Def{tree} in our context, is defined 
recursively as a node together with
a (possibly empty) sequence of rooted trees.
We shall use the standard terminology about trees like \Def{root},
\Def{edge}, \Def{child}, \Def{descendant}, \Def{subtree}, {\em etc.}
The \Def{size} of a tree is its number of nodes.
The nodes of the trees considered in this work are labeled by positive integers.
We draw trees with the root at the top, where a node is depicted by $
\begin{tikzpicture}[scale=.2,Centering]
        \node[Node](1)at(0,0){};
\end{tikzpicture}$
with its label inside the circle.
A \Def{forest} is a sequence of trees. From a forest $\ForestF$ of $n$ trees, it is always possible to build a tree $\TreeT$ by taking the root of each element of $\ForestF$ and by linking all these roots to an artificial node, such that this artificial node become the root of $\TreeT$. The size of the obtained tree is one plus the sum of all sizes of trees in $\ForestF$.
\smallbreak

A \Def{binary tree} (or \Def{$2$-tree}) $\TreeT$ is either a leaf or a node attached through two edges to two binary trees, 
which are called respectively the \Def{left subtree} and the \Def{right subtree} of $\TreeT$. Recall that the \Def{size} of a binary tree is its number of nodes.
We denote by $\SetBinaryTrees(n)$ the set of binary trees of size $n$. The set of binary trees is enumerated by Catalan numbers.
We draw binary trees with the root at the top and the leaves at the bottom, where a node is depicted by $
\begin{tikzpicture}[scale=.2,Centering]
        \node[Node](1)at(0,0){};
\end{tikzpicture}$ and a leaf is depicted by $
\begin{tikzpicture}[scale=.3,Centering]
        \node[Leaf](0)at(0,0){};
\end{tikzpicture}$ (see for instance Figure~\ref{arbrebinaire}).
\smallbreak

Let $\TreeT\in\SetBinaryTrees(n)$. Each node of $\TreeT$ is numbered recursively, starting with the left subtree, then the root, and ending with the right subtree. An example is given in Figure~\ref{arbrebinaire}. This numbering then establishes a total order on the nodes of a binary tree called the \Def{infix order}. Afterwards, this numbering is used to refer to the nodes. The sequence of nodes numbered from $1$ to $n$ forms the \Def{infix traversal}.
\smallbreak

When the size $n$ of $\TreeT$ satisfies $n \geq 1$, the \Def{canopy} of $\TreeT$ is the word of size $n-1$ on the alphabet $\{0, 1\}$ built by assigning to each leaf of $\TreeT$ a letter as follows.
Any leaf oriented to the left (resp.\ right) is labeled by $0$ (resp.\ $1$).
The canopy of $\TreeT$ is the word obtained by reading from left to right the labels thus established, forgetting the first and the last one (since there are always respectively $0$ and $1$).
For instance, the binary tree in Figure~\ref{arbrebinaire} has for canopy the word $0110100$. 
There is a link between infix order of a binary tree and its canopy. For a node of index $i$ for the infix order in a tree $\TreeT$, the right subtree of $i$ is a leaf oriented to the right if and only if the $i$-th letter of the canopy of $\TreeT$ is $1$. The left subtree of $i$ is a leaf oriented to the left if and only if the $(i-1)$-th letter of the canopy of $\TreeT$ is $0$. The two direct implications can be proved by induction on the set of binary trees, for instance, see Lemma 4.3. of \cite{Gir12}. The converses are simply given by the definition of the canopy.
\smallbreak

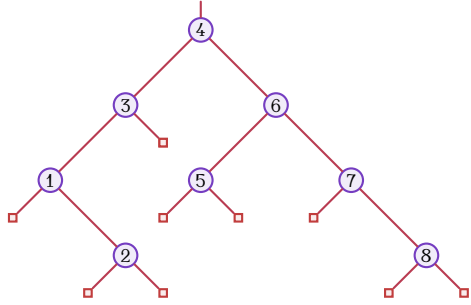
\begin{figure}[ht]
\centering
    \scalebox{1}{
    \begin{tikzpicture}[scale=.4,Centering]
        \node[Node](1)at(2,8){$1$};
        \node[Node](2)at(4,6){$2$};
        \node[Node](3)at(4,10){$3$};
        \node[Node](4)at(6,12){$4$};
        \node[Node](5)at(6,8){$5$};
        \node[Node](6)at(8,10){$6$};
        \node[Node](7)at(10,8){$7$};
        \node[Node](8)at(12,6){$8$};
        \node[Leaf](55)at(5,9){};
        \node[Leaf](11)at(1,7){};
        \node[Leaf](33)at(5,9){};
        \node[Leaf](22)at(3,5){};
        \node[Leaf](222)at(5,5){};
        \node[Leaf](55)at(5,7){};
        \node[Leaf](555)at(7,7){};
        \node[Leaf](77)at(9,7){};
        \node[Leaf](88)at(11,5){};
        \node[Leaf](888)at(13,5){};
        \draw[PathStep](4)--(3);
        \draw[PathStep](6)--(4);
        \draw[PathStep](3)--(1);
        \draw[PathStep](1)--(2);
        \draw[PathStep](5)--(55);
        \draw[PathStep](5)--(555);
        \draw[PathStep](1)--(11);
        \draw[PathStep](3)--(33);
        \draw[PathStep](2)--(22);
        \draw[PathStep](2)--(222);
        \draw[PathStep](6)--(5);
        \draw[PathStep](6)--(7);
        \draw[PathStep](7)--(8);
        \draw[PathStep](7)--(77);
        \draw[PathStep](8)--(88);
        \draw[PathStep](8)--(888);
        \node(r)at(6,14){};
        \draw[Edge](r)--(4);
    \end{tikzpicture}}
\caption{A binary tree of size $8$ and the numbering of its nodes in the infix order.}
\label{arbrebinaire}
\end{figure}

A fundamental operation in binary trees is the \Def{right rotation}\cite{Tam62}. Let $k$ and $l$ be the indices for the infix order of two nodes of a binary tree $\TreeT$, such that the node $k$ is the left child of the node $l$. Right rotation locally changes the tree $\TreeT$ so that $l$ becomes the right child of $k$ (see Figure~\ref{OpérationRot}). Equivalently, this means that the local configuration  $((\TreeA,\TreeB),\TreeC)$ becomes $(\TreeA,(\TreeB,\TreeC))$, where $\TreeA, \TreeB$ and $\TreeC$ are the subtrees shown in Figure~\ref{OpérationRot}.
\smallbreak

\begin{figure}[h!]
\scalebox{1}{
    \begin{tikzpicture}[scale=.4,Centering]
        \node[Node](1)at(5,10){$l$};
        \node[BigNode](2)at(1,6){$\TreeA$};
        \node[BigNode](3)at(5,6){$\TreeB$};
        \node[Node](4)at(3,8){$k$};
        \node[BigNode](5)at(7,8){$\TreeC$};
        \node[Node](6)at(13,10){$k$};
        \node[BigNode](7)at(11,8){$\TreeA$};
        \node[Node](8)at(15,8){$l$};
        \node[BigNode](9)at(17,6){$\TreeC$};
        \node[BigNode](10)at(13,6){$\TreeB$};

        \draw[PathStep](4)--(1);
        \draw[PathStep](2)--(4);
        \draw[PathStep](3)--(4);
        \draw[PathStep](5)--(1);
        \draw[PathStep](6)--(7);
        \draw[PathStep](6)--(8);
        \draw[PathStep](8)--(9);
        \draw[PathStep](8)--(10);
        \node(r)at(5,11.5){};
        \draw[Edge, ColAB!90,dotted](r)--(1);
        \node(rr)at(13,11.5){};
        \draw[Edge, ColAB!90,dotted](rr)--(6);
        \draw[Edge, ColAB!90](8,9)edge[->]node[above]{r. rot.}(10,9);
    \end{tikzpicture}}
\caption{Right rotation of edge $(k,l)$ in $\TreeT$ (on the left), where $\TreeA$, $\TreeB$, and $\TreeC$ are any subtrees.}
\label{OpérationRot}
\end{figure}
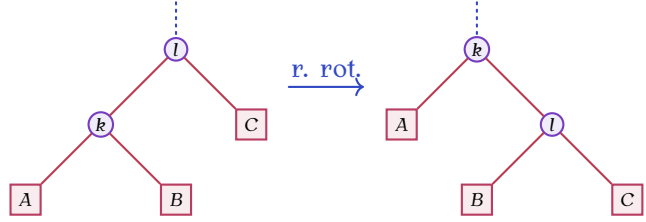

\subsection{Tamari intervals and interval-posets}\label{sss:TamariIntervals-IntervalPosets}
For any $n \geq 0$, let $\TreeS,\TreeT\in\SetBinaryTrees(n)$. We set $\TreeS\Leq_{\TamariOrder} \TreeT$ if either $\TreeT = \TreeS$ or $\TreeT$ is obtained by successively applying one or more right rotations in $\TreeS$.
The set $\SetBinaryTrees(n)$ endows with $\Leq_{\TamariOrder}$ is the \Def{Tamari lattice} of order $n$~\cite{HT72}. 
Moreover, $\TreeS$ is covered by $\TreeT$, denoted by $\TreeS \Covered_\TamariOrder \TreeT$, if $\TreeT$ is obtained from $\TreeS$ by performing one right rotation. 
\smallbreak

In the literature, the Tamari lattice is closely related to the \Def{associahedron}, or the \Def{Stasheff polytope} after the work of Stasheff. More precisely, the Hasse diagram of the Tamari lattice is the $1$-skeleton of the associahedron.
\smallbreak

Let $\TreeS,\TreeT\in\SetBinaryTrees(n)$. A \Def{Tamari interval} of \Def{size} $n$ is an interval $(\TreeS,\TreeT)$ for the Tamari order $\Leq_{\TamariOrder}$. The set of Tamari intervals of size $n$ is denoted by $\SetIntervalBinaryTrees$. 
\smallbreak

The Tamari interval lattice is the set $\SetIntervalBinaryTrees$ endowed with the partial order $\Leq_\TamariIntervalOrder$. Let $n \geq 0$ and $(\TreeS,\TreeT), (\TreeS',\TreeT')\in\SetIntervalBinaryTrees$, following~\eqref{posetintervalle}, we have that $(\TreeS,\TreeT) \Leq_\TamariIntervalOrder (\TreeS',\TreeT')$ if $\TreeS \Leq_{\TamariOrder} \TreeS'$ and $\TreeT\Leq_{\TamariOrder} \TreeT'$. According to Proposition~\ref{treillis}, the poset so defined is a lattice. Moreover, it follows from the definition of $\Leq_\TamariIntervalOrder$ that $(\TreeS',\TreeT')$ covers $(\TreeS,\TreeT)$ if
\begin{itemize}
\item either $\TreeS'$ is obtained by a single right rotation of an edge in $\TreeS$ and $\TreeT'=\TreeT$,
\item or $\TreeT'$ is obtained by a single right rotation of an edge in $\TreeT$ and $\TreeS'=\TreeS$.
\end{itemize}
\smallbreak

It is known from~\cite{Cha06} that Tamari intervals of size $n$ are enumerated by 
\begin{equation}
\frac{2(4n+1)!}{(n+1)!(3n+2)!}.
\end{equation}
The first numbers are
\begin{equation}
1, 1, 3, 13, 68, 399, 2530, 16965.
\end{equation}
This sequence is Sequence~\OEIS{A000260} of \cite{Slo}.
\smallbreak

Interval-posets are posets introduced by Châtel and Pons in~\cite{CP15} in order to study the Tamari lattice.
Indeed, there is a poset isomorphism between the Tamari interval lattices and the set of interval-posets endowed with a certain partial order. 
\smallbreak

Let $n \geq 0$ and $\{\IntervalPoset_{1},\dots, \IntervalPoset_{n}\}$ be a set of $n$ symbols numbered from $1$ to $n$.
An \Def{interval-poset} $\IntervalPoset$ is a partial order $\lhd$ on the set $\{\IntervalPoset_{1},\dots, \IntervalPoset_{n}\}$ such that 
\begin{enumerate}[label=(\roman*)]
\item\label{pip1} if $i < k$ and $\IntervalPoset_k\lhd \IntervalPoset_i$, then for all $\IntervalPoset_j$ such that $i<j<k$, one has $\IntervalPoset_j\lhd \IntervalPoset_i$,
\item\label{pip2} if $i < k$ and $\IntervalPoset_i\lhd \IntervalPoset_k$, then for all $\IntervalPoset_j$ such that $i<j<k$, one has $\IntervalPoset_j\lhd \IntervalPoset_k$.
\end{enumerate} 
The \Def{size} of an interval-poset is the cardinality of its underlying set.
The set of interval-posets of size $n$ is denoted by $\SetIntervalPoset(n)$, and the elements of interval-poset are called \Def{vertices}. 
\smallbreak

The two conditions~\ref{pip1} and~\ref{pip2} of interval-posets are referred to as \Def{interval-poset properties}. For any $i<j$, the relations $\IntervalPoset_j\lhd \IntervalPoset_i$ are known as \Def{decreasing relations} and the relations $\IntervalPoset_i\lhd \IntervalPoset_j$ are known as \Def{increasing relations}. 
\smallbreak

As it is shown in Figure~\ref{RIP2}, the Hasse diagram of interval-posets can be drawn as directed graph where two vertices $\IntervalPoset_i$ and $\IntervalPoset_j$ are related by an arrow from $\IntervalPoset_i$ to $\IntervalPoset_j$ (resp.\ $\IntervalPoset_j$ to $\IntervalPoset_i$) if $\IntervalPoset_i\lhd \IntervalPoset_j$ (resp.\ $\IntervalPoset_j\lhd \IntervalPoset_i$) where $i<j$. 
\smallbreak
 
Let $n \geq 0$ and $(\TreeS,\TreeT)\in\SetIntervalBinaryTrees$ and $\IntervalPoset\in\SetIntervalPoset(n)$.
We will recall a bijection $\rho$ relating on the one hand the restriction of $\IntervalPoset$ to its decreasing relations with the binary tree $\TreeS$, and on the other hand the restriction of $\IntervalPoset$ to its increasing relations with the binary tree $\TreeT$.
\smallbreak

Thus, from the restriction of $\IntervalPoset$ to its decreasing (resp.\ increasing) relations we build a forest referred to as the \Def{decreasing} (resp.\ \Def{increasing}) \Def{forest}, such that if $\IntervalPoset_j \lhd \IntervalPoset_i$ with $i<j$ (resp.\ $j<i$), then the node $j$ is a descendant of the node $i$. Otherwise, if $\IntervalPoset_j \ntriangleleft \IntervalPoset_i$ with $i < j$ (resp.\ $j < i$) the node $j$ is placed to the right (resp.\ left) of the node $i$. 
\smallbreak

Note that we obtain a decreasing (resp.\ increasing) forest formed by trees labelled from the roots to the leaves in increasing (resp.\ decreasing) order. Moreover, the prefix (resp.\ suffix) traversal of the decreasing (resp.\ increasing) forest gives the sequence of labels $1, \dots, n$.
Let us add a virtual root node (without label) on the top of both decreasing and increasing forests to form two trees. We denote by $\TreeS'$ and $\TreeT'$ the trees respectively obtained from the decreasing and the increasing forests. 
\smallbreak

Let $\rho$ be the map sending $\IntervalPoset$ to the pair $(\TreeS, \TreeT)$ of binary trees defined such that the tree $\TreeS$ (resp.\ $\TreeT$) is the unique binary tree obtained by reading $\TreeS'$ (resp.\ $\TreeT'$) in the following way. For all label $i, j$ in $\TreeS'$ (resp.\ $\TreeT'$), if a node $j$ is a descendant of a node $i$ in $\TreeS'$ (resp.\ $\TreeT'$), then $j$ becomes a right (resp.\ left) descendant of the node $i$ in $\TreeS$ (resp.\ $\TreeT$). If a node $i$ is a left (resp.\ right) brother of a node $j$ in $\TreeS'$ (resp.\ $\TreeT'$), then $i$ becomes a left (resp.\ right) descendant of the node $j$ in $\TreeS$ (resp.\ $\TreeT$).
\smallbreak

Figure~\ref{ITex} gives an example of construction by the bijection $\rho$ of a Tamari interval from an interval-poset of size $5$.
\smallbreak

\begin{figure}[h!]
\subfloat[][Interval-poset of size $5$.]{
\begin{minipage}[c]{.5\textwidth}
\centering
\scalebox{.8}
{\begin{tikzpicture}
\draw (-10.2,2.7) node[anchor=north west] {$\IntervalPoset_1$};
\draw (-9.2,2.7) node[anchor=north west] {$\IntervalPoset_2$};
\draw (-8.2,2.7) node[anchor=north west] {$\IntervalPoset_3$};
\draw (-7.2,2.7) node[anchor=north west] {$\IntervalPoset_4$};
\draw (-6.2,2.7) node[anchor=north west] {$\IntervalPoset_5$};
\draw [shift={(-8,3)},line width=1pt,color=ColBB!80]  plot[domain=0:3.14,variable=\t]({1*1*cos(\t r)+0*1*sin(\t r)},{0*1*cos(\t r)+1*1*sin(\t r)});
\draw [shift={(-7.5,3)},line width=1pt,color=ColBB!80]  plot[domain=0.04:3.2,variable=\t]({1*0.5*cos(\t r)+0*0.5*sin(\t r)},{0*0.5*cos(\t r)+1*0.5*sin(\t r)});
\draw [shift={(-6.5,2)},line width=1pt,color=ColBB!80]  plot[domain=-3.2:0.04,variable=\t]({1*0.5*cos(\t r)+0*0.5*sin(\t r)},{0*0.5*cos(\t r)+1*0.5*sin(\t r)});
\draw [shift={(-9.46,2)},line width=1pt,color=ColBB!80]  plot[domain=-3.2:0.04,variable=\t]({1*0.5*cos(\t r)+0*0.5*sin(\t r)},{0*0.5*cos(\t r)+1*0.5*sin(\t r)});
\draw [shift={(-9,2)},line width=1pt,color=ColBB!80]  plot[domain=-3.14:0,variable=\t]({1*1*cos(\t r)+0*1*sin(\t r)},{0*1*cos(\t r)+1*1*sin(\t r)});
\begin{scriptsize}
\draw [fill=ColBB!20,shift={(-7,3)},rotate=180] (0,0) ++(0 pt,3pt) -- ++(2.6pt,-4.5pt)--++(-5.2pt,0 pt) -- ++(2.6pt,4.5pt);\draw [fill=ColBB!20,shift={(-7,2)}] (0,0) ++(0 pt,3pt) -- ++(2.6pt,-4.5pt)--++(-5.2pt,0 pt) -- ++(2.6pt,4.5pt);
\draw [fill=ColBB!20,shift={(-10,2)}] (0,0) ++(0 pt,3pt) -- ++(2.6pt,-4.5pt)--++(-5.2pt,0 pt) -- ++(2.6pt,4.5pt);
\end{scriptsize}
\end{tikzpicture}}
\end{minipage}}
\subfloat[][Decreasing (on the left) and increasing (on the right) forests]{
\begin{minipage}[c]{.5\textwidth}
\centering
\scalebox{1}
{\begin{tikzpicture}[scale=.5,Centering]
        \node[Node](1)at(3,10){$1$};
        \node[Node](2)at(2,9){$2$};
        \node[Node](3)at(4,9){$3$};
        \node[Node](4)at(6,10){$4$};
        \node[Node](5)at(6,9){$5$};
        \node[Node](6)at(8,10){$1$};
        \node[Node](7)at(9,9){$2$};
        \node[Node](8)at(11,9){$3$};
        \node[Node](9)at(10,10){$4$};
        \node[Node](10)at(12,10){$5$};
        \draw[PathStep](1)--(3);
        \draw[PathStep](1)--(2);
        \draw[PathStep](4)--(5);
        \draw[PathStep](9)--(7);
        \draw[PathStep](9)--(8);
        \node[inner sep=0 mm](r)at(4.5,11.5){};
        \draw[Edge,ColAB!90,dotted](r)--(4);
        \draw[Edge,ColAB!90,dotted](r)--(1);
        \node[inner sep=0 mm](rr)at(10,11.5){};
        \draw[Edge,ColAB!90,dotted](rr)--(9);
        \draw[Edge,ColAB!90,dotted](rr)--(10);
        \draw[Edge,ColAB!90,dotted](rr)--(6);
    \end{tikzpicture}}
\end{minipage}}
\vspace{0.5cm}

\subfloat[][Left binary tree.]{
    \begin{minipage}{.45\textwidth}
    \centering
    \scalebox{1}{
    \begin{tikzpicture}[scale=.5,Centering]
        \node[Node](1)at(3,8){$1$};
        \node[Node](2)at(3,4){$2$};
        \node[Node](3)at(5,6){$3$};
        \node[Node](4)at(5,10){$4$};
        \node[Node](5)at(7,8){$5$};
        \node[Leaf](55)at(6,7){};
        \node[Leaf](555)at(8,7){};
        \node[Leaf](11)at(2,7){};
        \node[Leaf](33)at(6,5){};
        \node[Leaf](22)at(2,3){};
        \node[Leaf](222)at(4,3){};
        \draw[PathStep](4)--(1);
        \draw[PathStep](1)--(3);
        \draw[PathStep](3)--(2);
        \draw[PathStep](4)--(5);
        \draw[PathStep](5)--(55);
        \draw[PathStep](5)--(555);
        \draw[PathStep](1)--(11);
        \draw[PathStep](3)--(33);
        \draw[PathStep](2)--(22);
        \draw[PathStep](2)--(222);
        \node(r)at(5,11){};
        \draw[Edge](r)--(4);
    \end{tikzpicture}}
\end{minipage}
    \label{subfig:BT left}}
    \qquad
    \subfloat[][Right binary tree.]{
    \begin{minipage}{.45\textwidth}
    \centering
    \scalebox{1}{
    \begin{tikzpicture}[scale=.5,Centering]
        \node[Node](1)at(5,10){$1$};
        \node[Node](2)at(5,6){$2$};
        \node[Node](3)at(7,4){$3$};
        \node[Node](4)at(7,8){$4$};
        \node[Node](5)at(9,6){$5$};
        \node[Leaf](55)at(10,5){};
        \node[Leaf](555)at(8,5){};
        \node[Leaf](11)at(4,9){};
        \node[Leaf](33)at(6,3){};
        \node[Leaf](22)at(4,5){};
        \node[Leaf](333)at(8,3){};
        \draw[PathStep](4)--(1);
        \draw[PathStep](2)--(4);
        \draw[PathStep](3)--(2);
        \draw[PathStep](4)--(5);
        \draw[PathStep](5)--(55);
        \draw[PathStep](5)--(555);
        \draw[PathStep](1)--(11);
        \draw[PathStep](3)--(33);
        \draw[PathStep](2)--(22);
        \draw[PathStep](3)--(333);
        \node(r)at(5,11){};
        \draw[Edge](r)--(1);
    \end{tikzpicture}}
\end{minipage}}
    \label{subfig:BT right}
\caption{Construction of a Tamari interval from an interval-poset by $\rho$.}
\label{ITex}
\end{figure}
\smallbreak

In this section, we shall draw interval-posets as follows.
For any $i<j$, if $\IntervalPoset_j \lhd \IntervalPoset_i$ and there is no vertex $\IntervalPoset_k$ such that $\IntervalPoset_k\lhd \IntervalPoset_i$ and $j<k$, then we draw an arrow with source $\IntervalPoset_j$ and target $\IntervalPoset_i$ from below as shown in the example in Figure~\ref{IP8ex}. Symmetrically, if $\IntervalPoset_j\lhd \IntervalPoset_k$ and $j<k$ and if there is no $\IntervalPoset_i$ such that $\IntervalPoset_i\lhd \IntervalPoset_k$ and $i<j$, then we draw an arrow with source $\IntervalPoset_j$ and target $\IntervalPoset_k$ from above. We refer to this directed graph with two types of arrows as the \Def{minimalist representation} of $\IntervalPoset$.
\smallbreak

The closure for the interval-poset properties is given by adding the decreasing relations $\IntervalPoset_j\lhd \IntervalPoset_i$ for any relation $\IntervalPoset_k\lhd \IntervalPoset_i$ and by adding the increasing relations $\IntervalPoset_j\lhd \IntervalPoset_k$ for any relation $\IntervalPoset_i\lhd \IntervalPoset_k$, for any $i<j<k$.
By taking the reflexive closure and the closure for the interval-poset properties, an interval-poset is obtained from the minimalist representation. The interest of the minimalist representation is justified later, in particular with Theorem~\ref{bijDIT}. It is important to represent the decreasing relations and the increasing relations independently.
\smallbreak

\begin{figure}[h!]
\subfloat[][Minimalist representation.]{
\begin{minipage}[c]{.5\textwidth}
\centering
\scalebox{.8}
{\begin{tikzpicture}
\draw (-8.27,5.8) node[anchor=north west] {$\IntervalPoset_1$};
\draw (-7.27,5.8) node[anchor=north west] {$\IntervalPoset_2$};
\draw (-6.27,5.8) node[anchor=north west] {$\IntervalPoset_3$};
\draw (-5.27,5.8) node[anchor=north west] {$\IntervalPoset_4$};
\draw (-4.27,5.8) node[anchor=north west] {$\IntervalPoset_5$};
\draw (-3.27,5.8) node[anchor=north west] {$\IntervalPoset_6$};
\draw (-2.27,5.8) node[anchor=north west] {$\IntervalPoset_7$};
\draw (-1.27,5.8) node[anchor=north west] {$\IntervalPoset_8$};
\draw [shift={(-6,5.36027397260274)},line width=1pt,color=ColBB!80]  plot[domain=0.3095786777173675:2.832013975872426,variable=\t]({1*2.0998212757588393*cos(\t r)+0*2.0998212757588393*sin(\t r)},{0*2.0998212757588393*cos(\t r)+1*2.0998212757588393*sin(\t r)});
\draw [shift={(-6.5,5.914285714285715)},line width=1pt,color=ColBB!80]  plot[domain=0.16977827396833786:2.9718143796214553,variable=\t]({1*0.5072937401304201*cos(\t r)+0*0.5072937401304201*sin(\t r)},{0*0.5072937401304201*cos(\t r)+1*0.5072937401304201*sin(\t r)});
\draw [shift={(-1.5,5.958695652173913)},line width=1pt,color=ColBB!80]  plot[domain=0.08242154867336304:3.0591711049164303,variable=\t]({1*0.5017031484347476*cos(\t r)+0*0.5017031484347476*sin(\t r)},{0*0.5017031484347476*cos(\t r)+1*0.5017031484347476*sin(\t r)});
\draw [shift={(-7,5)},line width=1pt,color=ColBB!80]  plot[domain=-3.141592653589793:0,variable=\t]({1*1*cos(\t r)+0*1*sin(\t r)},{0*1*cos(\t r)+1*1*sin(\t r)});
\draw [shift={(-3,5)},line width=1pt,color=ColBB!80]  plot[domain=-3.141592653589793:0,variable=\t]({1*1*cos(\t r)+0*1*sin(\t r)},{0*1*cos(\t r)+1*1*sin(\t r)});
\begin{scriptsize}
\draw [fill=ColBB!20,shift={(-4,6)},rotate=180] (0,0) ++(0 pt,3pt) -- ++(2.598076211353316pt,-4.5pt)--++(-5.196152422706632pt,0 pt) -- ++(2.598076211353316pt,4.5pt);
\draw [fill=ColBB!20,shift={(-6,6)},rotate=180] (0,0) ++(0 pt,3pt) -- ++(2.598076211353316pt,-4.5pt)--++(-5.196152422706632pt,0 pt) -- ++(2.598076211353316pt,4.5pt);
\draw [fill=ColBB!20,shift={(-1,6)},rotate=180] (0,0) ++(0 pt,3pt) -- ++(2.598076211353316pt,-4.5pt)--++(-5.196152422706632pt,0 pt) -- ++(2.598076211353316pt,4.5pt);
\draw [fill=ColBB!20,shift={(-8,5)}] (0,0) ++(0 pt,3pt) -- ++(2.598076211353316pt,-4.5pt)--++(-5.196152422706632pt,0 pt) -- ++(2.598076211353316pt,4.5pt);
\draw [fill=ColBB!20,shift={(-4,5)}] (0,0) ++(0 pt,3pt) -- ++(2.598076211353316pt,-4.5pt)--++(-5.196152422706632pt,0 pt) -- ++(2.598076211353316pt,4.5pt);
\end{scriptsize}\end{tikzpicture}}
\end{minipage}
\label{RIP1}}
\subfloat[][Hasse diagram.]{
\begin{minipage}[c]{.5\textwidth}
\centering
\scalebox{.8}
{\begin{tikzpicture}
\draw (-8.27,5.8) node[anchor=north west] {$\IntervalPoset_1$};
\draw (-7.27,5.8) node[anchor=north west] {$\IntervalPoset_2$};
\draw (-6.27,5.8) node[anchor=north west] {$\IntervalPoset_3$};
\draw (-5.27,5.8) node[anchor=north west] {$\IntervalPoset_4$};
\draw (-4.27,5.8) node[anchor=north west] {$\IntervalPoset_5$};
\draw (-3.27,5.8) node[anchor=north west] {$\IntervalPoset_6$};
\draw (-2.27,5.8) node[anchor=north west] {$\IntervalPoset_7$};
\draw (-1.27,5.8) node[anchor=north west] {$\IntervalPoset_8$};
\draw [shift={(-6,5.36027397260274)},line width=1pt,color=ColBB!80]  plot[domain=0.3095786777173675:2.832013975872426,variable=\t]({1*2.0998212757588393*cos(\t r)+0*2.0998212757588393*sin(\t r)},{0*2.0998212757588393*cos(\t r)+1*2.0998212757588393*sin(\t r)});
\draw [shift={(-6.5,5.914285714285715)},line width=1pt,color=ColBB!80]  plot[domain=0.16977827396833786:2.9718143796214553,variable=\t]({1*0.5072937401304201*cos(\t r)+0*0.5072937401304201*sin(\t r)},{0*0.5072937401304201*cos(\t r)+1*0.5072937401304201*sin(\t r)});
\draw [shift={(-1.5,5.958695652173913)},line width=1pt,color=ColBB!80]  plot[domain=0.08242154867336304:3.0591711049164303,variable=\t]({1*0.5017031484347476*cos(\t r)+0*0.5017031484347476*sin(\t r)},{0*0.5017031484347476*cos(\t r)+1*0.5017031484347476*sin(\t r)});
\draw [shift={(-7,5)},line width=1pt,color=ColBB!80]  plot[domain=-3.141592653589793:0,variable=\t]({1*1*cos(\t r)+0*1*sin(\t r)},{0*1*cos(\t r)+1*1*sin(\t r)});
\draw [shift={(-3,5)},line width=1pt,color=ColBB!80]  plot[domain=-3.141592653589793:0,variable=\t]({1*1*cos(\t r)+0*1*sin(\t r)},{0*1*cos(\t r)+1*1*sin(\t r)});
\draw [shift={(-4.5,5.94)},line width=1pt,color=ColBB!80]  plot[domain=0.11942892601833573:3.0221637275714572,variable=\t]({1*0.5035871324805659*cos(\t r)+0*0.5035871324805659*sin(\t r)},{0*0.5035871324805659*cos(\t r)+1*0.5035871324805659*sin(\t r)});
\draw [shift={(-3.5,5.1105)},line width=1pt,color=ColBB!80]  plot[domain=3.359096592577117:6.065681368192262,variable=\t]({1*0.5120646931785084*cos(\t r)+0*0.5120646931785084*sin(\t r)},{0*0.5120646931785084*cos(\t r)+1*0.5120646931785084*sin(\t r)});
\begin{scriptsize}
\draw [fill=ColBB!20,shift={(-4,6)},rotate=180] (0,0) ++(0 pt,3pt) -- ++(2.598076211353316pt,-4.5pt)--++(-5.196152422706632pt,0 pt) -- ++(2.598076211353316pt,4.5pt);
\draw [fill=ColBB!20,shift={(-6,6)},rotate=180] (0,0) ++(0 pt,3pt) -- ++(2.598076211353316pt,-4.5pt)--++(-5.196152422706632pt,0 pt) -- ++(2.598076211353316pt,4.5pt);
\draw [fill=ColBB!20,shift={(-1,6)},rotate=180] (0,0) ++(0 pt,3pt) -- ++(2.598076211353316pt,-4.5pt)--++(-5.196152422706632pt,0 pt) -- ++(2.598076211353316pt,4.5pt);
\draw [fill=ColBB!20,shift={(-8,5)}] (0,0) ++(0 pt,3pt) -- ++(2.598076211353316pt,-4.5pt)--++(-5.196152422706632pt,0 pt) -- ++(2.598076211353316pt,4.5pt);
\draw [fill=ColBB!20,shift={(-4,5)}] (0,0) ++(0 pt,3pt) -- ++(2.598076211353316pt,-4.5pt)--++(-5.196152422706632pt,0 pt) -- ++(2.598076211353316pt,4.5pt);
\end{scriptsize}\end{tikzpicture}}
\end{minipage}
\label{RIP2}}

\subfloat[][Diagram with all apparent (except reflexive) relations.]{
\begin{minipage}[c]{.5\textwidth}
\centering
\scalebox{.8}
{\begin{tikzpicture}
\draw (-4.25,1.7) node[anchor=north west] {$\IntervalPoset_1$};
\draw (-3.25,1.7) node[anchor=north west] {$\IntervalPoset_2$};
\draw (-2.25,1.7) node[anchor=north west] {$\IntervalPoset_3$};
\draw (-1.25,1.7) node[anchor=north west] {$\IntervalPoset_4$};
\draw (-0.25,1.7) node[anchor=north west] {$\IntervalPoset_5$};
\draw (0.75,1.7) node[anchor=north west] {$\IntervalPoset_6$};
\draw (1.75,1.7) node[anchor=north west] {$\IntervalPoset_7$};
\draw (2.75,1.7) node[anchor=north west] {$\IntervalPoset_8$};
\draw [shift={(-2,2)},line width=1pt,color=ColBB!80]  plot[domain=0:3.141592653589793,variable=\t]({1*2*cos(\t r)+0*2*sin(\t r)},{0*2*cos(\t r)+1*2*sin(\t r)});
\draw [shift={(-0.5,2.0660526315789474)},line width=1pt,color=ColBB!80]  plot[domain=-0.1313447195593218:3.2729373731491145,variable=\t]({1*0.5043440791151454*cos(\t r)+0*0.5043440791151454*sin(\t r)},{0*0.5043440791151454*cos(\t r)+1*0.5043440791151454*sin(\t r)});
\draw [shift={(-2.5,2.068859649122807)},line width=1pt,color=ColBB!80]  plot[domain=-0.1368583851894547:3.278451038779247,variable=\t]({1*0.5047193787416098*cos(\t r)+0*0.5047193787416098*sin(\t r)},{0*0.5047193787416098*cos(\t r)+1*0.5047193787416098*sin(\t r)});
\draw [shift={(-1,2)},line width=1pt,color=ColBB!80]  plot[domain=0:3.141592653589793,variable=\t]({1*1*cos(\t r)+0*1*sin(\t r)},{0*1*cos(\t r)+1*1*sin(\t r)});
\draw [shift={(-1.5,1.9525862068965516)},line width=1pt,color=ColBB!80]  plot[domain=0.03159867435861777:3.1099939792311755,variable=\t]({1*1.5007491688408348*cos(\t r)+0*1.5007491688408348*sin(\t r)},{0*1.5007491688408348*cos(\t r)+1*1.5007491688408348*sin(\t r)});
\draw [shift={(2.5,2.066052631578948)},line width=1pt,color=ColBB!80]  plot[domain=-0.1313447195593227:3.2729373731491154,variable=\t]({1*0.5043440791151456*cos(\t r)+0*0.5043440791151456*sin(\t r)},{0*0.5043440791151456*cos(\t r)+1*0.5043440791151456*sin(\t r)});
\draw [shift={(0.5,0.9519090909090907)},line width=1pt,color=ColBB!80]  plot[domain=-3.237479516405944:0.09588686281615068,variable=\t]({1*0.5023074113898682*cos(\t r)+0*0.5023074113898682*sin(\t r)},{0*0.5023074113898682*cos(\t r)+1*0.5023074113898682*sin(\t r)});
\draw [shift={(1,1)},line width=1pt,color=ColBB!80]  plot[domain=-3.141592653589793:0,variable=\t]({1*1*cos(\t r)+0*1*sin(\t r)},{0*1*cos(\t r)+1*1*sin(\t r)});
\draw [shift={(-3,1)},line width=1pt,color=ColBB!80]  plot[domain=-3.141592653589793:0,variable=\t]({1*1*cos(\t r)+0*1*sin(\t r)},{0*1*cos(\t r)+1*1*sin(\t r)});
\draw [shift={(-3.5,1.051)},line width=1pt,color=ColBB!80]  plot[domain=3.243241109473272:6.181536851296107,variable=\t]({1*0.5025942697643891*cos(\t r)+0*0.5025942697643891*sin(\t r)},{0*0.5025942697643891*cos(\t r)+1*0.5025942697643891*sin(\t r)});
\begin{scriptsize}
\draw [fill=ColBB!20,shift={(0,2)},rotate=180] (0,0) ++(0 pt,3pt) -- ++(2.598076211353316pt,-4.5pt)--++(-5.196152422706632pt,0 pt) -- ++(2.598076211353316pt,4.5pt);
\draw [fill=ColBB!20,shift={(-2,2)},rotate=180] (0,0) ++(0 pt,3pt) -- ++(2.598076211353316pt,-4.5pt)--++(-5.196152422706632pt,0 pt) -- ++(2.598076211353316pt,4.5pt);
\draw [fill=ColBB!20,shift={(3,2)},rotate=180] (0,0) ++(0 pt,3pt) -- ++(2.598076211353316pt,-4.5pt)--++(-5.196152422706632pt,0 pt) -- ++(2.598076211353316pt,4.5pt);
\draw [fill=ColBB!20,shift={(0,1)}] (0,0) ++(0 pt,3pt) -- ++(2.598076211353316pt,-4.5pt)--++(-5.196152422706632pt,0 pt) -- ++(2.598076211353316pt,4.5pt);
\draw [fill=ColBB!20,shift={(-4,1)}] (0,0) ++(0 pt,3pt) -- ++(2.598076211353316pt,-4.5pt)--++(-5.196152422706632pt,0 pt) -- ++(2.598076211353316pt,4.5pt);
\end{scriptsize}\end{tikzpicture}}
\end{minipage}}
\caption{Different representations of an interval-poset of size $8$.}
\label{IP8ex}
\end{figure}
\smallbreak

Let $n \geq 0$ and $\IntervalPoset, \IntervalPoset'\in\SetIntervalPoset(n)$ and $(\TreeS,\TreeT) := \rho(\IntervalPoset)$, $(\TreeS',\TreeT') := \rho(\IntervalPoset')$.
Let $(\star)$ (resp.\ $(\diamond)$) the following condition: $\IntervalPoset'$ is obtained from $\IntervalPoset$ by adding (resp.\ removing) only decreasing (resp.\ increasing) relations of target a vertex $\IntervalPoset_k$, such that if only one of these decreasing (resp.\ increasing) relations is removed (resp.\ added), then either $\IntervalPoset$ is obtained or the object obtained is not an interval-poset.
\smallbreak

For the sequel, we need to recall that $(\TreeS',\TreeT')$ covers $(\TreeS,\TreeT)$ if and only if $\IntervalPoset$ and $\IntervalPoset'$ satisfy either $(\star)$ or $(\diamond)$.
\smallbreak

\begin{Lemma}\label{rotation/ip}
The interval-posets $\IntervalPoset$ and $\IntervalPoset'$ satisfy $(\star)$ (resp.\ $(\diamond)$) for the vertex $\IntervalPoset_i$ (resp.\ $\IntervalPoset_j$) if and only if $\TreeS'$ (resp.\ $\TreeT'$) is obtained by a unique right rotation of the edge $(i,j)$ in $\TreeS$ (resp.\ $\TreeT$) and $\TreeT' = \TreeT$ (resp.\ $\TreeS' = \TreeS$).
\end{Lemma}
\begin{proof}
Suppose $\IntervalPoset$ and $\IntervalPoset'$ satisfy $(\star)$ for the vertex $\IntervalPoset_i$. Therefore, $\IntervalPoset'$ has more decreasing relations of target $\IntervalPoset'_i$ than the vertex $\IntervalPoset_i$ in $\IntervalPoset$. Suppose that the vertices $\IntervalPoset_j$ and $\IntervalPoset_i$ are not related in $\IntervalPoset$, and that $\IntervalPoset'_j$ and $\IntervalPoset'_i$ are related in $\IntervalPoset'$, with $i<j$. Then, by the  interval-poset property~\ref{pip1}, for any $\IntervalPoset'_k$ such that $i<k<j$, $\IntervalPoset'_{k}\lhd \IntervalPoset'_i$. Moreover, if we remove only one of these decreasing relations, we obtain either $\IntervalPoset$ or an object that is no longer an interval-poset. This means that the number of descending relations added in $\IntervalPoset'$ is minimal, or equivalently, that the vertex $\IntervalPoset_j$ is closest to the vertex $\IntervalPoset_i$ such that $\IntervalPoset_j$ and $\IntervalPoset_i$ are not related in $\IntervalPoset$ and $i<j$.
This case is depicted in Figure~\ref{IPD}. 
\begin{figure}[ht!]
\scalebox{0.9}{\begin{tikzpicture}
\draw (-5.1,3.3) node[anchor=north west] {$\overbrace{\dots \IntervalPoset_{i-1}}^{\TreeA} \quad\enspace \IntervalPoset_{i}\qquad\quad {\overbrace{\IntervalPoset_{i+1} \dots \IntervalPoset_{j-1}}^{\TreeB}}\quad\quad \IntervalPoset_{j}\quad\qquad \overbrace{\IntervalPoset_{j+1} \dots}^\TreeC$};
\draw [shift={(-1.5,2.276935483870968)},line width=1pt,color=ColBB!80]  plot[domain=3.324160505389721:6.100617455379658,variable=\t]({1*1.5253502096983658*cos(\t r)+0*1.5253502096983658*sin(\t r)},{0*1.5253502096983658*cos(\t r)+1*1.5253502096983658*sin(\t r)});
\draw [shift={(-1,2.2111111111111112)},line width=1pt,dotted,color=ColBB!80]  plot[domain=3.2467587768116766:6.178019183957702,variable=\t]({1*2.011111111111111*cos(\t r)+0*2.011111111111111*sin(\t r)},{0*2.011111111111111*cos(\t r)+1*2.011111111111111*sin(\t r)});
\draw [shift={(2.1489473684210525,2.1310526315789473)},line width=1pt,color=ColBB!80]  plot[domain=3.2551650253318054:6.187003801461131,variable=\t]({1*1.1563973571594799*cos(\t r)+0*1.1563973571594799*sin(\t r)},{0*1.1563973571594799*cos(\t r)+1*1.1563973571594799*sin(\t r)});
\begin{scriptsize}
\draw [fill=ColBB!20,shift={(-3,2)}] (0,0) ++(0 pt,3pt) -- ++(2.598076211353316pt,-4.5pt)--++(-5.196152422706632pt,0 pt) -- ++(2.598076211353316pt,4.5pt);
\draw [fill=ColBB!20,shift={(1,2)}] (0,0) ++(0 pt,3pt) -- ++(2.598076211353316pt,-4.5pt)--++(-5.196152422706632pt,0 pt) -- ++(2.598076211353316pt,4.5pt);
\end{scriptsize}
\end{tikzpicture}}
\caption{Interval-poset of the decreasing forest before (without dotted line) and after (with dotted line) the right rotation of the edge $(i,j)$, where $\TreeA$, $\TreeB$ and $\TreeC$ may be empty.}
\label{IPD}
\end{figure}
By the bijection $\rho$, add these decreasing relations of target $\IntervalPoset_i$ in $\IntervalPoset$ leads to the decreasing forest induced by $\TreeS'$ represented by Figure~\ref{OpéRotD}.
\begin{figure}[ht]
\subfloat[][\small Binary trees $\TreeS$ and $\TreeS'$ (resp.\ $\TreeT$ and $\TreeT'$).]{
\begin{minipage}[c]{.55\textwidth}
\centering
\scalebox{1}{
    \begin{tikzpicture}[scale=.5,Centering]
        \node[Node](1)at(5,10){$j$};
        \node[BigNode](2)at(1,6){$A$};
        \node[BigNode](3)at(5,6){$B$};
        \node[Node](4)at(3,8){$i$};
        \node[BigNode](5)at(7,8){$C$};
        \node[Node](6)at(13,10){$i$};
        \node[BigNode](7)at(11,8){$A$};
        \node[Node](8)at(15,8){$j$};
        \node[BigNode](9)at(17,6){$C$};
        \node[BigNode](10)at(13,6){$B$};

        \draw[PathStep](4)--(1);
        \draw[PathStep](2)--(4);
        \draw[PathStep](3)--(4);
        \draw[PathStep](5)--(1);
        \draw[PathStep](6)--(7);
        \draw[PathStep](6)--(8);
        \draw[PathStep](8)--(9);
        \draw[PathStep](8)--(10);
        \node(r)at(5,11.5){};
        \draw[Edge, ColAB!90,dotted](r)--(1);
        \node(rr)at(13,11.5){};
        \draw[Edge, ColAB!90,dotted](rr)--(6);
        \draw[Edge, ColAB!90](8,9)edge[->]node[above]{r. rot.}(10,9);
    \end{tikzpicture}}
\end{minipage}
\label{OpéRotO}}
\vspace{0.5cm}

\subfloat[][\small Decreasing forests induced by $\TreeS$ and $\TreeS'$.]{
\begin{minipage}[c]{.5\textwidth}
\centering
\scalebox{1}{
    \begin{tikzpicture}[scale=.5,Centering]
        \node[BigNode](1)at(3,10){$\TreeA$};
        \node[Node](2)at(5,10){$i$};
        \node[BigNode](3)at(5,8.5){$\TreeB$};
        \node[Node](4)at(7,10){$j$};
        \node[BigNode](5)at(7,8.5){$\TreeC$};
        \node[BigNode](6)at(10,10){$\TreeA$};
        \node[BigNode](7)at(13,7.5){$\TreeC$};
        \node[BigNode](8)at(11,9){$\TreeB$};
        \node[Node](9)at(12,10){$i$};
        \node[Node](10)at(13,9){$j$};
        \draw[PathStep](2)--(3);
        \draw[PathStep](4)--(5);
        \draw[PathStep](10)--(9);
        \draw[PathStep](9)--(8);
        \draw[PathStep](7)--(10);
        \node[inner sep=0 mm](r)at(5,11.5){};
        \draw[Edge,ColAB!90,dotted](r)--(4);
        \draw[Edge,ColAB!90,dotted](r)--(1);
        \draw[Edge,ColAB!90,dotted](r)--(2);
        \node[inner sep=0 mm](rr)at(11,11.5){};
        \draw[Edge,ColAB!90,dotted](rr)--(9);
        \draw[Edge,ColAB!90,dotted](rr)--(6);
    \end{tikzpicture}}
\end{minipage}
\label{OpéRotD}}
\subfloat[][\small Increasing forests induced by $\TreeT$ and $\TreeT'$.]{
\begin{minipage}[c]{.5\textwidth}
\centering
    \scalebox{1}{
    \begin{tikzpicture}[scale=.5,Centering]
        \node[Node](1)at(3,10){$j$};
        \node[Node](2)at(2,9){$i$};
        \node[BigNode](3)at(4,9){$\TreeB$};
        \node[BigNode](4)at(5,10){$\TreeC$};
        \node[BigNode](5)at(2,7.5){$\TreeA$};
        \node[Node](6)at(8,10){$i$};
        \node[BigNode](7)at(8,8.5){$\TreeA$};
        \node[BigNode](8)at(10,8.5){$\TreeB$};
        \node[Node](9)at(10,10){$j$};
        \node[BigNode](10)at(12,10){$\TreeC$};
        \draw[PathStep](1)--(3);
        \draw[PathStep](1)--(2);
        \draw[PathStep](2)--(5);
        \draw[PathStep](6)--(7);
        \draw[PathStep](9)--(8);
        \node[inner sep=0 mm](r)at(4,11.5){};
        \draw[Edge,ColAB!90,dotted](r)--(4);
        \draw[Edge,ColAB!90,dotted](r)--(1);
        \node[inner sep=0 mm](rr)at(10,11.5){};
        \draw[Edge,ColAB!90,dotted](rr)--(9);
        \draw[Edge,ColAB!90,dotted](rr)--(10);
        \draw[Edge,ColAB!90,dotted](rr)--(6);
    \end{tikzpicture}}
\end{minipage}
\label{OpéRotC}}
\caption{Right rotation of the edge $(i,j)$ in the binary tree $\TreeS$ (resp.\ $\TreeT$), where $\TreeA$, $\TreeB$ and $\TreeC$ are subtrees.}
\label{OpéRot}
\end{figure}
A unique right rotation is then made between the trees $\TreeS$ and $\TreeS'$ (see Figure~\ref{OpéRotO}).
Furthermore, since the increasing relations are unchanged between $\IntervalPoset$ and $\IntervalPoset'$, the increasing forests induced by $\TreeT$ and $\TreeT'$ are the same, and thus $\TreeT' = \TreeT$.
\smallbreak

Reciprocally, suppose that $\TreeS'$ is obtained by a unique right rotation of the edge $(i,j)$ in $\TreeS$ and that $\TreeT'=\TreeT$. The case is depicted by Figure~\ref{OpéRotO}, and the two decreasing forests induced by $\TreeS$ and $\TreeS'$ are depicted by Figure~\ref{OpéRotD}. By the bijection $\rho$, we then obtain the interval-poset whose restriction to decreasing relations is shown by Figure~\ref{IPD}. Since $\TreeT'=\TreeT$, the increasing relations of the interval-posets associated with $(\TreeS,\TreeT)$ and $(\TreeS',\TreeT')$ are the same. Finally, $\IntervalPoset'$ is obtained by adding only decreasing relations of target $\IntervalPoset_i$ in $\IntervalPoset$. Furthermore, if only one of these relations is removed, then either $\IntervalPoset$ is obtained, or the object obtained is not an interval-poset. This means that $\IntervalPoset$ and $\IntervalPoset'$ satisfy $(\star)$.
\smallbreak

Symmetrically, we show that $\IntervalPoset$ and $\IntervalPoset'$ satisfy $(\diamond)$ for $\IntervalPoset_j$ if and only if $\TreeT'$ is obtained by a unique right rotation of the edge $(i,j)$ in $\TreeT$ and $\TreeS'=\TreeS$. Figure~\ref{OpéRotC} and Figure~\ref{IPC} depicts this case.
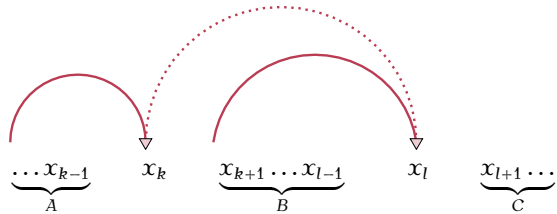
\begin{figure}[ht!]
\scalebox{0.9}{\begin{tikzpicture}
\draw (-5.1,1.8) node[anchor=north west] {$\underbrace{\dots \IntervalPoset_{i-1}}_{\TreeA} \quad\quad \IntervalPoset_{i}\quad\quad \underbrace{\IntervalPoset_{i+1} \dots \IntervalPoset_{j-1}}_\TreeB\quad\quad\enspace \IntervalPoset_{j}\quad\quad \underbrace{\IntervalPoset_{j+1} \dots}_\TreeC$};
\draw [shift={(-4,2)},line width=1pt,color=ColBB!80]  plot[domain=0:3.141592653589793,variable=\t]({1*1*cos(\t r)+0*1*sin(\t r)},{0*1*cos(\t r)+1*1*sin(\t r)});
\draw [shift={(-1,2)},line width=1pt,dotted,color=ColBB!80]  plot[domain=0:3.141592653589793,variable=\t]({1*2*cos(\t r)+0*2*sin(\t r)},{0*2*cos(\t r)+1*2*sin(\t r)});
\draw [shift={(-0.5,1.7775968992248063)},line width=1pt,color=ColBB!80]  plot[domain=0.14719634926715644:2.994396304322637,variable=\t]({1*1.5163980807276238*cos(\t r)+0*1.5163980807276238*sin(\t r)},{0*1.5163980807276238*cos(\t r)+1*1.5163980807276238*sin(\t r)});
\begin{scriptsize}
\draw [fill=ColBB!20,shift={(-3,2)},rotate=180] (0,0) ++(0 pt,3pt) -- ++(2.598076211353316pt,-4.5pt)--++(-5.196152422706632pt,0 pt) -- ++(2.598076211353316pt,4.5pt);
\draw [fill=ColBB!20,shift={(1,2)},rotate=180] (0,0) ++(0 pt,3pt) -- ++(2.598076211353316pt,-4.5pt)--++(-5.196152422706632pt,0 pt) -- ++(2.598076211353316pt,4.5pt);
\end{scriptsize}\end{tikzpicture}}
\caption{Interval-poset of the increasing forest before (with dotted lines) and after (without dotted lines) the right rotation of the edge $(i,j)$, where $\TreeA$, $\TreeB$ and $\TreeC$ may be empty.}
\label{IPC}
\end{figure}
\end{proof}

\section{Cubic coordinates and Tamari intervals}\label{sec:coordonnées-cubiques}

The aim of this section is to build the poset of the cubic coordinates, then to establish the poset isomorphism between this poset and the poset of the Tamari intervals. To achieve this goal, we first define the Tamari interval diagrams based on the interval-posets. The cubic coordinates are then obtained from the Tamari interval diagrams.
\smallbreak

\subsection{Tamari interval diagrams}

Let us give the definition of a Tamari diagram, as formulated in~\cite{BW97}.
For any $n \geq 0$, a \Def{Tamari diagram} is a word $u$ of length $n$ on the alphabet $\mathbb{N}$ which satisfies the two following conditions:
\begin{enumerate}[label=(\roman*)]
\item \label{condDT1+} $0\leq u_i \leq n - i$ for all $i\in [n]$,
\item \label{condDT2+} $u_{i+j} \leq u_i - j$ for all $i\in [n]$ and $j\in [0, u_i]$.
\end{enumerate}
The \Def{size} of a Tamari diagram is its number of letters.
For instance, the sets of Tamari diagrams of size $2$, $3$ and $4$ are
\begin{equation}
\begin{split}
&\qquad \qquad \qquad \{00, 10\}, \qquad \qquad \qquad \{000, 100, 010, 200, 210\},\\ 
&\{0000, 0010, 0100, 0200, 0210, 1000, 1010, 
2000, 2100, 3000, 3010, 3100, 3200, 3210\}.
\end{split}
\end{equation}
\smallbreak

In the literature, Tamari diagrams are also known as bracket vectors, objects inspired by the right bracketing introduced in~\cite{HT72} by Huang and Tamari. Furthermore, Tamari diagrams are known to be enumerated by Catalan numbers 
\begin{equation}
\mathrm{cat}(n) := \frac{1}{n + 1} \binom{2n}{n}.
\end{equation}
\smallbreak

A dual version of Tamari diagrams can be defined by considering the opposite of Conditions~\ref{condDT1+} and \ref{condDT2+}. 
For any $n \geq 0$, a \Def{dual Tamari diagram} is a word $v$ of length $n$ on the alphabet $\mathbb{N}$ which satisfies the two following conditions:
\begin{enumerate}[label=(\roman*)]
\item \label{condDTD1+} $0\leq v_i \leq i - 1$ for all $i\in [n]$,
\item \label{condDTD2+} $v_{i-j} \leq v_i - j$ for all $i\in [n]$ and $j\in [0, v_i]$.
\end{enumerate}
The \Def{size} of a dual Tamari diagram is its number of letters.
In other words, $v=v_1\dots v_n$ is a dual Tamari diagram if and only if $v_n\dots v_1$ is a Tamari diagram.
\smallbreak

Note that the first condition of a Tamari diagram $u$ and of a dual Tamari diagram $v$ of size $n$ implies that $u_n = 0$ and $v_1 = 0$.
\smallbreak

A graphical representation of a Tamari diagram $u$ of size $n$ by needles and diagonals provides a simple way to check Condition~\ref{condDT2+} of a Tamari diagram.
For each position $i \in [n]$, we draw a needle from the point $(i - 1, 0)$ to the point $\Par{i - 1, u_i}$ in the
Cartesian plane. Condition~\ref{condDT2+} says that one can draw lines of slope $-1$ passing
through the $x$-axis and the top of each needle without crossing any other needle. For instance, the Tamari diagram $9021043100$ is drawn by Figure~\ref{dtdtd}. One can observe that none of its diagonals, drawn as dotted lines, crosses a needle.
\smallbreak

Likewise, a graphical representation can be given for the dual Tamari diagram $v$ of size $n$. One draws $v$ in the same way as Tamari diagram, and Condition~\ref{condDTD2+} says that one can draw lines of slope $1$ passing through the $x$-axis and the top of each needle without crossing any other needle. Figure~\ref{dtdtd} also depicts the dual Tamari diagram $0010040002$. 
\smallbreak

\begin{figure}[ht]
    \centering
    \begin{minipage}{.45\textwidth}
    \centering
    \scalebox{1}{
    \begin{tikzpicture}[scale=.4,Centering]
        \draw[Grid](0,0)grid(9,9);
        \draw[PathStep, draw=ColAB!70](0,0)--(9,0);
        \node[PathNode](1)at(0,9){};
        \node[PathNode](2)at(1,0){};
        \node[PathNode](3)at(2,2){};
        \node[PathNode](4)at(3,1){};
        \node[PathNode](5)at(4,0){};
        \node[PathNode](6)at(5,4){};
        \node[PathNode](7)at(6,3){};
        \node[PathNode](8)at(7,1){};
        \node[PathNode](9)at(8,0){};
        \node[PathNode](10)at(9,0){};
        \draw[PathStep](0,0)--(1);
        \draw[PathStep](2,0)--(3);
        \draw[PathStep](3,0)--(4);
        \draw[PathStep](5,0)--(6);
        \draw[PathStep](6,0)--(7);
        \draw[PathStep](7,0)--(8);
        \draw[PathDiag](1)--(6);
        \draw[PathDiag](6)--(7);
        \draw[PathDiag](7)--(10);
        \draw[PathDiag](3)--(4);
        \draw[PathDiag](4)--(5);
        \draw[PathDiag](8)--(9);
    \end{tikzpicture}}
\end{minipage}
    \qquad
    \begin{minipage}{.45\textwidth}
    \centering
    \scalebox{1}{
    \begin{tikzpicture}[scale=.4,Centering]
        \draw[Grid](0,0)grid(9,9);
        \draw[PathStep, draw=ColAB!70](0,0)--(9,0);
        \node[PathNode](1)at(0,0){};
        \node[PathNode](2)at(1,0){};
        \node[PathNode](3)at(2,1){};
        \node[PathNode](4)at(3,0){};
        \node[PathNode](5)at(4,0){};
        \node[PathNode](6)at(5,4){};
        \node[PathNode](7)at(6,0){};
        \node[PathNode](8)at(7,0){};
        \node[PathNode](9)at(8,0){};
        \node[PathNode](10)at(9,2){};
        \draw[PathStep](2,0)--(3);
        \draw[PathStep](5,0)--(6);
        \draw[PathStep](9,0)--(10);
        \draw[PathDiag](3)--(6);
        \draw[PathDiag](2)--(3);
        \draw[PathDiag](8)--(10);
    \end{tikzpicture}}
\end{minipage}
    \caption{A Tamari diagram $9021043100$ (on the left) and a dual Tamari diagram $0010040002$ (on the right) of size $10$.}
\label{dtdtd}
\end{figure}
\smallbreak

For any $n\geq 0$, the set of Tamari diagrams of size $n$ is in bijection with $\SetBinaryTrees(n)$. Indeed, one builds from a Tamari diagram $u$ of size $n$ a binary tree $\TreeS$ recursively as follows. If $n = 0$, $\TreeS$ is defined as the leaf. Otherwise, let $i$ be the smallest position in $u$ such that $u_{i}$ is the maximum allowed value, namely $n - i$. Then $\TreeS_1 := u_1\dots u_{i-1}$ and $\TreeS_2 := u_{i + 1}\dots u_{n}$ are also Tamari diagrams. One forms $\TreeS$ by grafting the binary trees obtained recursively by this process applied on $\TreeS_1$ and on $\TreeS_2$ to a new node.
Reciprocally, for each node of index $i$ of the tree $\TreeS$, labeled with an infix traversal, the value of the $i$-th letter of the corresponding Tamari diagram is given by the number of nodes in the right subtree of the node $i$. The complete demonstration is given in~\cite{Pal86}.
\smallbreak

In the case of dual Tamari diagrams, the construction of the binary tree $\TreeT$ is also recursive, except that it is the maximum position $i$ in the dual Tamari diagram whose value is the highest allowed on that section of the word that should be chosen first. Similarly for the reciprocal, the procedure is identical, except that the value of the $i$-th letter in the dual Tamari diagram is given by the number of nodes in the left subtree of the node $i$ in the tree $\TreeT$. 
\smallbreak

For instance, in Figure~\ref{arbrebinaire}, the Tamari diagram is $10040210$ and the dual Tamari diagram is $00230100$. Figure~\ref{fig:bijectionTD-BT} depicts the corresponding binary tree of the Tamari diagram $1003010$.
\smallbreak

\begin{figure}
\center
\scalebox{1}{
       \begin{tikzpicture}[Centering,xscale=2,yscale=2]
        \node[](a)at(0,0){
        \begin{tikzpicture}[scale=.5,Centering]
        \node[Node](1)at(-2,-2){$1$};
        \node[Node](2)at(-1,-3){$2$};
        \node[Node](3)at(-1,-1){$3$};
        \node[Node](4)at(0,0){$4$};
        \node[Node](5)at(0,-2){$5$};
        \node[Node](6)at(1,-1){$6$};
        \node[Node](7)at(2,-2){$7$};
        \node[Leaf](11)at(-2.5,-2.5){};
        \node[Leaf](22)at(-1.5,-3.5){};
        \node[Leaf](33)at(-0.5,-1.5){};
        \node[Leaf](222)at(-0.5,-3.5){};
        \node[Leaf](55)at(0.5,-2.5){};
        \node[Leaf](555)at(-0.5,-2.5){};
        \node[Leaf](77)at(2.5,-2.5){};
        \node[Leaf](777)at(1.5,-2.5){};
        \draw[PathStep](1)--(2);
        \draw[PathStep](3)--(1);
        \draw[PathStep](4)--(3);
        \draw[PathStep](4)--(6);
        \draw[PathStep](6)--(5);
        \draw[PathStep](6)--(7);
        \draw[PathStep](1)--(11);
        \draw[PathStep](2)--(22);
        \draw[PathStep](2)--(222);
        \draw[PathStep](3)--(33);
        \draw[PathStep](5)--(55);
        \draw[PathStep](5)--(555);
        \draw[PathStep](7)--(77);
        \draw[PathStep](7)--(777);
        \node(r)at(0,1){};
        \draw[Edge](r)--(4);
    \end{tikzpicture}
        };
       \node[](b)at(4,0){
       \begin{tikzpicture}[scale=.4,Centering]
        \draw[Grid](0,0)grid(6,3);
        \draw[PathStep, draw=ColAB!70](0,0)--(6,0);
        \node[PathNode](1)at(0,1){};
        \node[PathNode](2)at(1,0){};
        \node[PathNode](3)at(2,0){};
        \node[PathNode](4)at(3,3){};
        \node[PathNode](5)at(4,0){};
        \node[PathNode](6)at(5,1){};
        \node[PathNode](7)at(6,0){};
        \draw[PathStep](0,0)--(1);
        \draw[PathStep](3,0)--(4);
        \draw[PathStep](5,0)--(6);
        \draw[PathDiag](1)--(2);
        \draw[PathDiag](4)--(6);
        \draw[PathDiag](6)--(7);
    \end{tikzpicture}
        };
    \draw[EdgeGraph, ->, shorten >= 4mm](a)--(b);
    \end{tikzpicture}}
\caption{A binary tree and the associated Tamari diagram of the same size.}
    \label{fig:bijectionTD-BT}
\end{figure}
\smallbreak

Let $n \geq 0$ and $u$ be a Tamari diagram, and $v$ be a dual Tamari diagram, both of size $n$. The diagrams $u$ and $v$ are \Def{compatible} if there are no $i,j$ with $1\leq i < j\leq n$ such that $u_i \geq j-i$ and $v_j \geq j-i$.
If $u$ and $v$ are compatible, then the pair $(u,v)$ is called \Def{Tamari interval diagram}.
The set of Tamari interval diagrams of size $n$ is denoted by $\dit$. 
\smallbreak

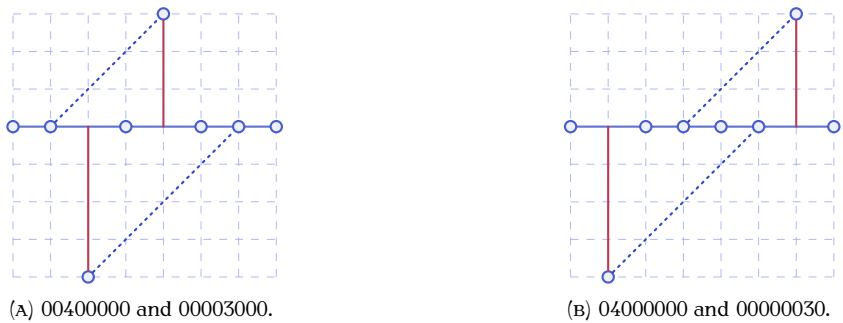
\begin{figure}[ht]
    \centering
    \begin{minipage}{.45\textwidth}
    \centering
    \scalebox{1}{
    \begin{tikzpicture}[scale=.5,Centering]
        \draw[Grid](0,0)grid(7,7);
        \draw[PathStep, draw=ColAB!70](0,4)--(7,4);
        \node[PathNode](1)at(0,4){};
        \node[PathNode](2)at(1,4){};
        \node[PathNode](3)at(2,0){};
        \node[PathNode](33)at(2,4){};
        \node[PathNode](4)at(3,4){};
        \node[PathNode](5)at(4,7){};
        \node[PathNode](55)at(4,4){};
        \node[PathNode](6)at(5,4){};
        \node[PathNode](7)at(6,4){};
        \node[PathNode](8)at(7,4){};
        \draw[PathStep](33)--(3);
        \draw[PathStep](55)--(5);
        \draw[PathDiag](2)--(5);
        \draw[PathDiag](3)--(7);
    \end{tikzpicture}}
\end{minipage}
    \label{subfig:tamari-int-diag-wrong}
    \qquad
    \begin{minipage}{.45\textwidth}
    \centering
    \scalebox{1}{
    \begin{tikzpicture}[scale=.5,Centering]
        \draw[Grid](0,0)grid(7,7);
        \draw[PathStep, draw=ColAB!70](0,4)--(7,4);
        \node[PathNode](1)at(0,4){};
        \node[PathNode](2)at(1,0){};
        \node[PathNode](22)at(1,4){};
        \node[PathNode](3)at(2,4){};
        \node[PathNode](4)at(3,4){};
        \node[PathNode](5)at(4,4){};
        \node[PathNode](6)at(5,4){};
        \node[PathNode](7)at(6,7){};
        \node[PathNode](77)at(6,4){};
        \node[PathNode](8)at(7,4){};
        \draw[PathStep](22)--(2);
        \draw[PathStep](77)--(7);
        \draw[PathDiag](2)--(6);
        \draw[PathDiag](4)--(7);
    \end{tikzpicture}}
\end{minipage}
    \label{subfig:tamari-int-diag-good}
    \caption{Two incompatible diagrams (on the left) and two compatible diagrams (on the right).}
    \label{fig:example-tamari-interval-diag}
\end{figure}
\smallbreak

In other words, a Tamari diagram $u$ of size $n$ and a dual Tamari diagram $v$ of size $n$ are compatible if for any needle of position $i$ and height $v_i\neq 0$ in $v$ (resp.\ $u_i\neq 0$ in $u$), there is no needle of position $j$ and height greater than or equal to $i-j$ in $u$ (resp.\ $j-i$ in $v$) with $i-v_i\leq j\leq i-1$ (resp.\ $i+1 \leq j\leq i + u_i$) and $i\in[n]$. 
\smallbreak

For example, the two diagrams in Figure~\ref{dtdtd} are compatible. Figure~\ref{fig:example-tamari-interval-diag} gives two other examples of two incompatible diagrams $00400000$ and $00003000$, and two compatible diagrams $04000000$ and $00000030$. Hereinafter, if $u$ and $v$ are compatible, we can also say that $u$ and $v$ satisfy the compatibility condition.
\smallbreak

\begin{figure}[ht]
    \centering
    \begin{minipage}{.45\textwidth}
    \centering
    \scalebox{1}{
\begin{tikzpicture}[scale=.5,Centering]
        \draw[Grid](0,0)grid(9,13);
        \draw[PathStep, draw=ColAB!70](0,9)--(9,9);
        \node[PathNode](1)at(0,0){};
        \node[PathNode](2)at(1,9){};
        \node[PathNode](33)at(2,7){};
        \node[PathNode](3)at(2,10){};
        \node[PathNode](4)at(3,8){};
        \node[PathNode](5)at(4,9){};
        \node[PathNode](66)at(5,5){};
        \node[PathNode](6)at(5,13){};
        \node[PathNode](7)at(6,6){};
        \node[PathNode](8)at(7,8){};
        \node[PathNode](9)at(8,9){};
        \node[PathNode](10)at(9,11){};
        \node[PathNode](11)at(0,9){};
        \node[PathNode](44)at(3,9){};
        \node[PathNode](77)at(6,9){};
        \node[PathNode](88)at(7,9){};
        \node[PathNode](1010)at(9,9){};
        \draw[PathStep](2,9)--(3);
        \draw[PathStep](5,9)--(6);
        \draw[PathStep](1010)--(10);
        \draw[PathStep](11)--(1);
        \draw[PathStep](2,9)--(33);
        \draw[PathStep](5,9)--(66);
        \draw[PathStep](77)--(7);
        \draw[PathStep](88)--(8);
        \draw[PathStep](44)--(4);
        \draw[PathDiag](2)--(3);
        \draw[PathDiag](3)--(6);
        \draw[PathDiag](7,9)--(10);
        \draw[PathDiag](1)--(66);
        \draw[PathDiag](66)--(7);
        \draw[PathDiag](7)--(9,9);
        \draw[PathDiag](33)--(4);
        \draw[PathDiag](4)--(5);
        \draw[PathDiag](8)--(9);
    \end{tikzpicture}}
\end{minipage}
    \label{subfig:tamari-interval-diag-bijection}
    \qquad
    \begin{minipage}{.45\textwidth}
    \centering
    \scalebox{.7}{
\begin{tikzpicture}
\draw (-8.27,5.8) node[anchor=north west] {$x_1$};
\draw (-7.27,5.8) node[anchor=north west] {$x_2$};
\draw (-6.27,5.8) node[anchor=north west] {$x_3$};
\draw (-5.27,5.8) node[anchor=north west] {$x_4$};
\draw (-4.27,5.8) node[anchor=north west] {$x_5$};
\draw (-3.27,5.8) node[anchor=north west] {$x_6$};
\draw (-2.27,5.8) node[anchor=north west] {$x_7$};
\draw (-1.27,5.8) node[anchor=north west] {$x_8$};
\draw (-0.27,5.8) node[anchor=north west] {$x_9$};
\draw (0.73,5.8) node[anchor=north west] {$x_{10}$};
\draw [shift={(-6.5,6.052181818181818)},line width=1pt,color=ColBB!80]  plot[domain=-0.10398719188036498:3.245579845470158,variable=\t]({1*0.5027155678400654*cos(\t r)+0*0.5027155678400654*sin(\t r)},{0*0.5027155678400654*cos(\t r)+1*0.5027155678400654*sin(\t r)});
\draw [shift={(-5,6)},line width=1pt,color=ColBB!80]  plot[domain=0:3.141592653589793,variable=\t]({1*2*cos(\t r)+0*2*sin(\t r)},{0*2*cos(\t r)+1*2*sin(\t r)});
\draw [shift={(0,6)},line width=1pt,color=ColBB!80]  plot[domain=0:3.141592653589793,variable=\t]({1*1*cos(\t r)+0*1*sin(\t r)},{0*1*cos(\t r)+1*1*sin(\t r)});
\draw [shift={(-3.5,5.565358361774745)},line width=1pt,color=ColBB!80]  plot[domain=3.2665730184936757:6.158204942275703,variable=\t]({1*4.535375406427633*cos(\t r)+0*4.535375406427633*sin(\t r)},{0*4.535375406427633*cos(\t r)+1*4.535375406427633*sin(\t r)});
\draw [shift={(-1,5.290127388535032)},line width=1pt,color=ColBB!80]  plot[domain=3.285651457406168:6.139126503363212,variable=\t]({1*2.0209339181621346*cos(\t r)+0*2.0209339181621346*sin(\t r)},{0*2.0209339181621346*cos(\t r)+1*2.0209339181621346*sin(\t r)});
\draw [shift={(-0.5,5.246883116883118)},line width=1pt,color=ColBB!80]  plot[domain=3.3047188902568587:6.120059070512521,variable=\t]({1*1.5201813291189716*cos(\t r)+0*1.5201813291189716*sin(\t r)},{0*1.5201813291189716*cos(\t r)+1*1.5201813291189716*sin(\t r)});
\draw [shift={(-0.5,4.970754716981132)},line width=1pt,color=ColBB!80]  plot[domain=-3.200016654616556:0.05842400102676268,variable=\t]({1*0.5008545563123628*cos(\t r)+0*0.5008545563123628*sin(\t r)},{0*0.5008545563123628*cos(\t r)+1*0.5008545563123628*sin(\t r)});
\draw [shift={(-5,5)},line width=1pt,color=ColBB!80]  plot[domain=-3.141592653589793:0,variable=\t]({1*1*cos(\t r)+0*1*sin(\t r)},{0*1*cos(\t r)+1*1*sin(\t r)});
\draw [shift={(-4.5,5.028297872340427)},line width=1pt,color=ColBB!80]  plot[domain=3.1981280872692808:6.226649873500098,variable=\t]({1*0.5008001293719831*cos(\t r)+0*0.5008001293719831*sin(\t r)},{0*0.5008001293719831*cos(\t r)+1*0.5008001293719831*sin(\t r)});
\begin{scriptsize}
\draw [fill=ColBB!20,shift={(-6,6)},rotate=180] (0,0) ++(0 pt,3pt) -- ++(2.598076211353316pt,-4.5pt)--++(-5.196152422706632pt,0 pt) -- ++(2.598076211353316pt,4.5pt);
\draw [fill=ColBB!20,shift={(-3,6)},rotate=180] (0,0) ++(0 pt,3pt) -- ++(2.598076211353316pt,-4.5pt)--++(-5.196152422706632pt,0 pt) -- ++(2.598076211353316pt,4.5pt);
\draw [fill=ColBB!20,shift={(1,6)},rotate=180] (0,0) ++(0 pt,3pt) -- ++(2.598076211353316pt,-4.5pt)--++(-5.196152422706632pt,0 pt) -- ++(2.598076211353316pt,4.5pt);
\draw [fill=ColBB!20,shift={(-8,5)}] (0,0) ++(0 pt,3pt) -- ++(2.598076211353316pt,-4.5pt)--++(-5.196152422706632pt,0 pt) -- ++(2.598076211353316pt,4.5pt);
\draw [fill=ColBB!20,shift={(-3,5)}] (0,0) ++(0 pt,3pt) -- ++(2.598076211353316pt,-4.5pt)--++(-5.196152422706632pt,0 pt) -- ++(2.598076211353316pt,4.5pt);
\draw [fill=ColBB!20,shift={(-2,5)}] (0,0) ++(0 pt,3pt) -- ++(2.598076211353316pt,-4.5pt)--++(-5.196152422706632pt,0 pt) -- ++(2.598076211353316pt,4.5pt);
\draw [fill=ColBB!20,shift={(-1,5)}] (0,0) ++(0 pt,3pt) -- ++(2.598076211353316pt,-4.5pt)--++(-5.196152422706632pt,0 pt) -- ++(2.598076211353316pt,4.5pt);
\draw [fill=ColBB!20,shift={(-6,5)}] (0,0) ++(0 pt,3pt) -- ++(2.598076211353316pt,-4.5pt)--++(-5.196152422706632pt,0 pt) -- ++(2.598076211353316pt,4.5pt);
\draw [fill=ColBB!20,shift={(-5,5)}] (0,0) ++(0 pt,3pt) -- ++(2.598076211353316pt,-4.5pt)--++(-5.196152422706632pt,0 pt) -- ++(2.598076211353316pt,4.5pt);
\end{scriptsize}
\end{tikzpicture}}
\end{minipage}
    \label{subfig:interval-poset-bijection}
    \caption{A Tamari interval diagram of size $10$ (on the left) and its associated interval-poset (on the right).}
\label{dit}
\end{figure}
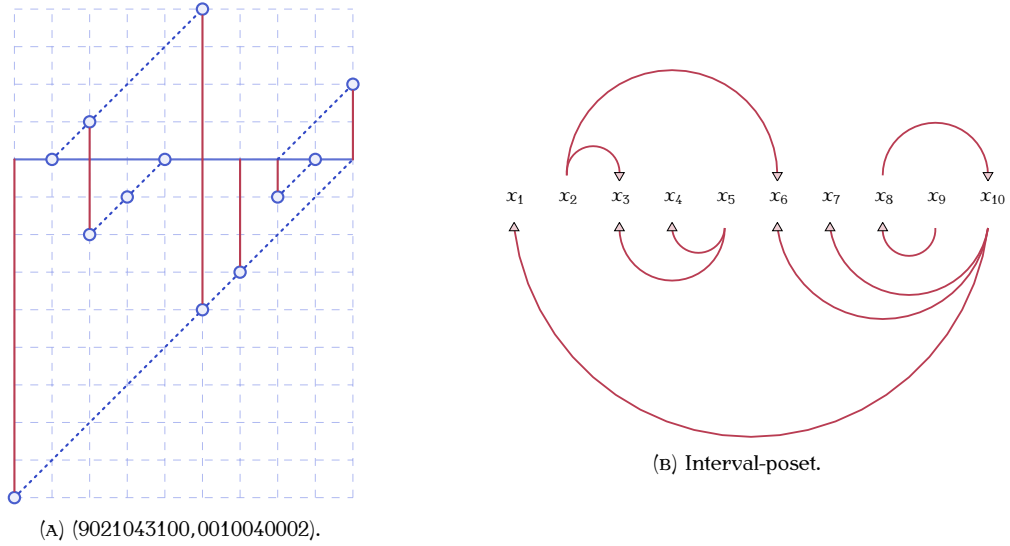
\smallbreak

As for Tamari diagrams and dual Tamari diagrams, a graphical representation of the Tamari interval diagram is also possible, as shown in Figure~\ref{fig:example-tamari-interval-diag}.
Figure~\ref{dit} gives the representation of the Tamari interval diagram $(9021043100, 0010040002)$ formed by the two diagrams seen in Figure~\ref{dtdtd} which are compatible, where we have simply considered the symmetry relative to the abscissa axis of the Tamari diagram, and placed it under the dual Tamari diagram. Thus, the Tamari diagram $u$ is drawn below and the dual Tamari diagram $v$ is drawn above. With such a representation, it is then easy to verify that $u$ and $v$ are compatible.
Indeed, any needle of $u$ that is below the diagonal linking the top of the needle in position $j$ in $v$ to the abscissa point $j-v_j$, has a diagonal that intersects the $x$-axis strictly before the position $j$.
Symmetrically, any needles of $v$ that is above a diagonal linking the top of the needle in position $i$ in $u$ to the abscissa point $i + u_i$, has a diagonal that intersects the $x$-axis strictly after the position $i$.
\smallbreak

One consequence of the compatibility condition is that each needle of non-zero height in the dual Tamari diagram $v$ is always preceded by a needle of $u$ of zero height. Symmetrically, each non-zero height needle in the Tamari diagram $u$ is always followed by a needle of $v$ of zero height. In other words, for any $i\in [n]$, $u_i$ and $v_{i+1}$ can both be zero, but cannot both be non-zero.
\smallbreak

\subsection{Link with interval-posets}\label{subsecIP}
Let us show that there is a bijection between the set of Tamari interval diagrams and the set of interval-posets of the same size.
\smallbreak

Let $n\geq 0$ and $\chi$ be the map sending a Tamari interval diagram $(u,v)$ of size $n$ to the relation
\begin{equation}
\begin{split} 
    \left(\{\IntervalPoset_{1},\dots, \IntervalPoset_{n}\}, 
    \lhd
    \right)
\end{split}
\end{equation}
where $\IntervalPoset_{i+l}\lhd \IntervalPoset_{i}$ for all $i\in [n]$ and $0\leq l\leq u_{i}$,
and $\IntervalPoset_{i-k}\lhd \IntervalPoset_{i}$ for all $i\in [n]$ and $0\leq k\leq v_{i}$.
\smallbreak

\begin{Proposition}
For any $n \geq 0$, the map $\chi$ has values in $\SetIntervalPoset(n)$.
\end{Proposition}
\smallbreak

\begin{proof}
Let $(u,v) \in\dit$ and $\IntervalPoset := \chi(u,v)$. First, we show that $\lhd$ is a partial order, then that interval-poset properties are satisfied.
\smallbreak

\begin{enumerate}[label=(\arabic*)]
\item By the definition of $\chi$ one has $\IntervalPoset_{i+ l}\lhd \IntervalPoset_{i}$ and $\IntervalPoset_{i-k}\lhd \IntervalPoset_{i}$ with $0\leq l\leq u_{i}$ and $0\leq k\leq v_{i}$ for all $\IntervalPoset_{i}\in \IntervalPoset$. Specifically, $\IntervalPoset_{i}\lhd \IntervalPoset_{i}$. This shows that $\IntervalPoset$ is reflexive. 

\item Let $\IntervalPoset_{i}$, $\IntervalPoset_{j}$ and $\IntervalPoset_{k}$ be vertices of $\IntervalPoset$ with $i < j < k$.
\begin{enumerate}
\item Suppose that $\IntervalPoset_{j}\lhd \IntervalPoset_{i}$ and that $\IntervalPoset_{k}
\lhd \IntervalPoset_{j}$. Then $\IntervalPoset_{j}\lhd 
\IntervalPoset_{i}$ implies that there is an integer $0\leq i'\leq 
u_{i}$ such that $j= i + i'$. Therefore, by Condition~\ref{condDT2+} of a Tamari diagram, $u_{j} = u_{i+i'} \leq u_{i} - i'$.
Likewise, $\IntervalPoset_{k}\lhd \IntervalPoset_{j}$ 
implies that there is an integer $0\leq j'\leq v_{j}$  
such that $k = j+j'$. Still by the same condition, one has $u_{k} = u_{j+j'}\leq u_{j} - j'$.
By using these two inequalities, we obtain that $u_{i}\geq u_{k} + i' +j'$. Since $i'+j' = k-i$, then we have $u_{i}\geq k-i$, which implies by the definition of $\chi$ that $\IntervalPoset_{k}\lhd \IntervalPoset_{i}$ in $\IntervalPoset$.

\item Suppose that $\IntervalPoset_{j}\lhd \IntervalPoset_{i}$ and that $\IntervalPoset_{i}
\lhd \IntervalPoset_{k}$. Therefore, $\IntervalPoset_{j}\lhd \IntervalPoset_{k}$ because $\IntervalPoset_{i}
\lhd \IntervalPoset_{k}$ implies that each vertex between $\IntervalPoset_i$ and $\IntervalPoset_k$ is in relation with $\IntervalPoset_k$.

\item Suppose that $\IntervalPoset_{i}\lhd \IntervalPoset_{j}$ and that $\IntervalPoset_{j}\lhd \IntervalPoset_{k}$. Then $\IntervalPoset_{i}\lhd \IntervalPoset_{j}$ implies that there is an integer $0\leq i'\leq v_{i}$ such that $i=j-i'$. By Condition~\ref{condDTD2+} of a dual Tamari diagram, $v_{i}= v_{j-i'} \leq v_{j} - i'$. Likewise, $\IntervalPoset_{j}\lhd \IntervalPoset_{k}$ implies that there is an integer $0\leq j'\leq v_{j}$ such that $j = k-j'$. By the same condition~\ref{condDTD2+}, $v_{j} = v_{k-j'}\leq v_{k} - j'$.
By these two inequalities, one has $v_{k}\geq v_{i} + i' +j'$. Since $i'+j' = k-i$, one has $v_{k}\geq k-i$, which implies by the definition of $\chi$ that $\IntervalPoset_{i}\lhd \IntervalPoset_{k}$ in $\IntervalPoset$.

\item Suppose that $\IntervalPoset_{j}\lhd \IntervalPoset_{k}$ and that $\IntervalPoset_{k}\lhd \IntervalPoset_{i}$. Then $\IntervalPoset_{j}\lhd \IntervalPoset_{i}$ because $\IntervalPoset_{k}\lhd \IntervalPoset_{i}$ implies that all vertex between $\IntervalPoset_i$ and $\IntervalPoset_k$ is in relation with $\IntervalPoset_i$.

\end{enumerate}
This shows that $\IntervalPoset$ is transitive.
Notice that it is impossible to have the case $\IntervalPoset_{i}\lhd \IntervalPoset_{k}$ and $\IntervalPoset_{k}\lhd \IntervalPoset_{j}$ since $\IntervalPoset$ is the image of a Tamari interval diagram. Getting this case would contradict the fact that $u$ and $v$ are compatible. Similarly, the case $\IntervalPoset_{i}\lhd \IntervalPoset_{j}$ and $\IntervalPoset_{k}\lhd \IntervalPoset_{i}$ is impossible.

\item Let $i<j$ and $\IntervalPoset_i$, $\IntervalPoset_j$ be vertices of $\IntervalPoset$. Suppose that $\IntervalPoset_j\lhd \IntervalPoset_i$ and that $\IntervalPoset_i \lhd \IntervalPoset_j$. By the definition of $\chi$, $\IntervalPoset_j\lhd \IntervalPoset_i$ if and only if $u_i\geq j-i$. Likewise, $\IntervalPoset_i \lhd \IntervalPoset_j$ if and only if $v_j \geq j-i$. However, since $u$ and $v$ are compatible, this case is impossible. This shows that $\IntervalPoset$ is antisymmetric.

\item The definition of $\chi$ implies directly that $\IntervalPoset$ satisfies the interval-poset properties, 
namely that for all $\IntervalPoset_{i}$, $\IntervalPoset_{j}$ and $\IntervalPoset_{k}$ vertices of $\IntervalPoset$ with $i < j < k$, if $\IntervalPoset_{k}\lhd \IntervalPoset_{i}$, then $\IntervalPoset_{j}\lhd \IntervalPoset_{i}$, and if $\IntervalPoset_{i}\lhd \IntervalPoset_{k}$, then $\IntervalPoset_{j}\lhd \IntervalPoset_{k}$.
\end{enumerate}
\end{proof}
\smallbreak

Let $n \geq 0$ and $\chi'$ be the map sending an interval-poset $\IntervalPoset$ of size $n$ on a pair of words $(u,v) \in \mathbb{N}^n\times\mathbb{N}^n$, such that for all $i \in [n]$, 
\begin{align}
\label{réciproque} u_i := \# \{\IntervalPoset_j\in \IntervalPoset : \IntervalPoset_j\lhd \IntervalPoset_i \mbox{ and } i<j\}; \\
\label{réciproque2} v_j := \# \{\IntervalPoset_i\in \IntervalPoset : \IntervalPoset_i\lhd \IntervalPoset_j \mbox{ and } i<j\}.
\end{align}
\smallbreak

\begin{Lemma}\label{upivpi}
Let $n \geq 0$, $\IntervalPoset \in \SetIntervalPoset(n)$ and $(u,v) := \chi'(\IntervalPoset)$. If $u_i \geq j - i$ (resp. $v_j \geq j - i$), then $\IntervalPoset_j \lhd \IntervalPoset_i$ (resp. $\IntervalPoset_i \lhd \IntervalPoset_j$), with $0 \leq i \leq j \leq n$.
\end{Lemma}

\begin{proof}
According to~\eqref{réciproque}, the fact that $u_i \geq j - i$ means that there are at least $j - i$ vertices in decreasing relation to the vertex $\IntervalPoset_i$. By the point~\ref{pip1} of interval-poset properties, this implies in particular that $\IntervalPoset_{j}\lhd \IntervalPoset_i$.
Respectively, we show with the point~\ref{pip2} of interval-poset properties that $v_j \geq j - i$ implies that $\IntervalPoset_i \lhd \IntervalPoset_j$.
\end{proof}

\begin{Theorem}\label{bijDIT}
For any $n \geq 0$, the map $\chi : \dit \rightarrow \SetIntervalPoset(n)$ is bijective.
\end{Theorem}
\smallbreak

\begin{proof}
Let us show that $\chi'$ is the inverse map of $\chi$. Let $n \geq 0$, $\IntervalPoset \in \SetIntervalPoset(n)$ and $(u,v) := \chi'(\IntervalPoset)$.
\smallbreak

\begin{enumerate}[label=(\arabic*)]
\item Since $\IntervalPoset$ is an interval-poset, there are at most $n-i$ vertices of $\IntervalPoset$ in decreasing relation to $\IntervalPoset_i$ and at most $i-1$ vertices of $\IntervalPoset$ in increasing relation to $\IntervalPoset_i$ for all $i\in [n]$. Therefore, the word $u$ satisfies Condition~\ref{condDTD1+} of a Tamari diagram and the word $v$ satisfies Condition~\ref{condDTD1+} of a dual Tamari diagram.

\item Let $\IntervalPoset_i$ and $\IntervalPoset_{i+j}$ be vertices of $\IntervalPoset$ such that $i\in [n]$ and $j\in [0, u_i]$. By Lemma~\ref{upivpi}, the fact that $u_i \geq j$ means that $\IntervalPoset_{i+j}\lhd \IntervalPoset_i$. Thus, by transitivity of interval-posets, one has that for any $i+j \leq k \leq n$, if $\IntervalPoset_k\lhd \IntervalPoset_{i+j}$, then $\IntervalPoset_k\lhd \IntervalPoset_i$. Thus, $u_{i+j} +j \leq u_i$, which implies Condition~\ref{condDT2+} of a Tamari diagram.

\noindent
Symmetrically, Condition~\ref{condDTD2+} of a dual Tamari diagram is checked by considering $\IntervalPoset_i$ and $\IntervalPoset_{i-j}$ vertices of $\IntervalPoset$ such that $i\in [n]$ and $j\in [0, v_i]$.

\item For all $i,j$ such that $1\leq i<j\leq n$ and $u_i\geq j-i$, suppose that $v_j\geq j-i$. By Lemma~\ref{upivpi}, the relation $u_i\geq j-i$ implies that $\IntervalPoset_j\lhd \IntervalPoset_i$. Likewise, the relation $v_j\geq j-i$ means that $\IntervalPoset_i\lhd \IntervalPoset_j$. Both of these implications lead to a contradiction with the antisymmetric nature of interval-posets. Necessarily, we have $v_j< j-i$, which implies that $u$ and $v$ are compatible.
\end{enumerate}
\smallbreak

The pair $(u,v)$ is a Tamari interval diagram of size $n$. 
Finally, it is clear that $\chi(u,v) = \IntervalPoset$ by construction.
Therefore, the map $\chi$ is surjective.
\smallbreak

Let $(u,v)$ and $(u',v')$ be two Tamari interval diagrams of size $n$, such that $(u,v)\neq (u',v')$ and such that $\chi(u,v):= \IntervalPoset$ and $\chi(u',v') := \IntervalPoset'$. So there is at least one letter of $(u,v)$ and $(u',v')$ such that $u_i\neq u'_i$ or $v_i\neq v'_i$, for $i \in [n]$. Therefore, the number of vertices of $\IntervalPoset$ in relation to the vertex $\IntervalPoset_i$ associated with the component $u_i$ and $v_i$ by $\chi$ is different from the number of vertices of $\IntervalPoset'$ in relation to the vertex $\IntervalPoset'_i$ associated with the component $u'_i$ and $v'_i$ by $\chi$, we thus have $\IntervalPoset\neq \IntervalPoset'$. This shows that the map $\chi$ is injective.
\end{proof}
\smallbreak

The minimalist representation of the interval-posets defined in Section~\ref{Sec:préliminaires} allows us to describe a direct construction of the corresponding Tamari interval diagram.
Indeed, let us consider the minimalist representation of an interval-poset $\IntervalPoset$ of size $n$. For any relation $\IntervalPoset_j\lhd \IntervalPoset_i$ (resp.\ $\IntervalPoset_i\lhd \IntervalPoset_j$) drawn, with $1\leq i<j \leq n$, we set $u_i := j-i$ (resp.\ $v_j := j-i$) and all other elements not involved in any relation to $0$. This forms a pair of words $(u,v)$ which is the inverse image of $\IntervalPoset$ by $\chi$.
\smallbreak

An example is given by Figure~\ref{dit}, where a Tamari interval diagram and its interval-poset which is its image by $\chi$ are shown.
\smallbreak


\subsection{Cubic coordinates}\label{subsecCC}
We describe in this part the set of cubic coordinates, and we show that there is a bijection between this set and the set of Tamari interval diagrams.
\smallbreak

An $(n-1)$-tuple $c$ on $\mathbb{Z}$ is a \Def{cubic coordinate} if there is a Tamari interval diagram $(u,v)$ of size $n$ such that
\begin{equation}
c = (u_1-v_2, u_2-v_3,\dots, u_{n-1}-v_n).
\end{equation}
The size of a cubic coordinate is its number of components plus one. 
The set of cubic coordinates of size $n$ is denoted by $\cc$.
For instance, $(9,-1,2,1,-4,4,3,1,-2)$ is a cubic coordinate of size $10$ since there is the Tamari interval diagram $(9021043100, 0010040002)$ satisfying the conditions of the definition.
\smallbreak

Besides, for any $n \geq 1$, let $\phi$ be the map sending an $(n-1)$-tuple $c$ on $\mathbb{Z}$ to a pair $(u,v)$ of words on $\mathbb{N}$, both of length $n$, such that $u$ satisfies $u_n = 0$ and for any $i\in[n-1]$,
\begin{equation}
    u_i = \max (c_i,~0),
\end{equation}
and $v$ satisfies $v_1 = 0$ and for any $2\leq i\leq n$,
\begin{equation}
    v_i = |\min (c_{i-1},~0)|.
\end{equation} 
\smallbreak

\begin{Theorem}\label{bijDIT-CC}
For any $n\geq 0$, the map $\phi : \cc \rightarrow \dit$ is bijective.
\end{Theorem}
\smallbreak

\begin{proof}
Let $c$ and $c'$ be two cubic coordinates of size $n$ such that $c\neq c'$. Then there is a component $c_i$ such that $c_i\neq c'_i$, with $i\in[n-1]$. By the map $\phi$, one has then $u_i\neq u'_i$ or $v_{i+1}\neq v'_{i+1}$, namely $(u,v)\neq (u',v')$. Which shows that the map $\phi$ is injective.
\smallbreak

Let $(u,v)\in\dit$. Let $c := (u_1-v_2, u_2-v_3,\dots, u_{n-1}-v_n)$, the $(n-1)$-tuple whose components are given by the difference between $u_i$ and $v_{i+1}$ for any $i\in[n-1]$. Now if $u_i \neq 0$, then $v_{i+1}= 0$ for any $i\in[n-1]$. Therefore, $\phi(c) = (u,v)$, where $(u,v)$ is indeed a Tamari interval diagram by hypothesis. By the definition of a cubic coordinate, one can conclude that $c\in \cc$. This shows that the map $\phi$ is surjective.
\end{proof}
\smallbreak

Therefore, by the map $\phi$ it is possible to build a cubic coordinate from a Tamari interval diagram and reciprocally. Graphically, we have to shift the upper part of a Tamari interval diagram (corresponding to the dual Tamari diagram) to the left by one position and collect the height of the needles from left to right. Then, we put a positive sign for the needles of the lower part of the Tamari interval diagram (corresponding to the Tamari diagram) and a negative sign for the upper part, and we forget the last needle of zero height. To reconstruct a Tamari interval diagram from a cubic coordinate, we reconstruct the needles of the Tamari diagram and the dual Tamari diagram from the components of the cubic coordinate in the same way, and then we shift the dual Tamari diagram to the right by one position. 
\smallbreak

Using the map $\chi$ we can then directly give the cubic coordinate of an interval-poset $\IntervalPoset$. In the same way that we shift the dual Tamari diagram one position to the left, we shift all the increasing relations of the interval-poset to the left by one vertex. Then, for each vertex $\IntervalPoset_i$, we count the number of elements in increasing or decreasing relation of target $\IntervalPoset_i$, out of reflexive relation, for all $i\in[n-1]$. These numbers become the components of positive sign if it is a decreasing relation, negative otherwise, of the cubic coordinate. As the increasing relations have been shifted, the number associated with the vertex $\IntervalPoset_n$ is always zero. Therefore, this vertex is forgotten for the cubic coordinate. In the same way, to construct an interval-poset from a cubic coordinate with each component of a cubic coordinate, we rebuild the increasing and decreasing relations on $n-1$ vertices, we add the vertex $\IntervalPoset_n$, then we shift the increasing relations to the right.
\smallbreak

\begin{Lemma}\label{lemmeutilepartout}
Let $n \geq 0$ and $c \in\cc$ such that there is a component $c_i \ne 0$, for $i\in[n-1]$. Let $c'$ be the $(n-1)$-tuple such that $c'_i = 0$ and $c'_j = c_j$ for any $j\ne i$, with $j\in[n-1]$. Then $c'$ is a cubic coordinate.
\end{Lemma}
\smallbreak

\begin{proof}
Let $(u',v') := \phi(c')$ and $(u'_j, v'_{j+1})$ be the pair of letters corresponding to $c'_j$ by the map $\phi$, with $j \in [n-1]$. 
Since $c'_i = 0$, then $(u'_i, v'_{i+1}) = (0,0)$. By hypothesis, all other pairs of letters are the same as those of $(u, v) := \phi(c)$.
In order to show that $c'$ is a cubic coordinate, we have to show that $(u',v')$ is a Tamari interval diagram, namely that $(u',v')$ satisfies the conditions of a Tamari diagram, of a dual Tamari diagram, and of compatibility. 
Clearly, with  $(u'_i, v'_{i+1}) = (0,0)$, all these conditions are satisfied for $(u',v')$.
\end{proof}
\smallbreak

Depending on the case, either the definition of cubic coordinates or the definition of Tamari interval diagrams is used, as it is done for the proof of Lemma~\ref{lemmeutilepartout}. For example, the following results are stated for Tamari interval diagrams.
\smallbreak

Let $n \geq 0$. A Tamari interval diagram $(u,v)$ of size $n$ is  \Def{synchronized} if either $u_i \ne 0$ or $v_{i+1} \ne 0$ for any $i\in [n-1]$.
\smallbreak

Likewise, a cubic coordinate $c$ of size $n$ is synchronized if $c_i\ne 0$ for any $i\in [n-1]$. The set of synchronized cubic coordinates of size $n$ is denoted by $\ccs$.
\smallbreak

A Tamari interval $(\TreeS,\TreeT)$ is synchronized if and only if the binary trees $\TreeS$ and $\TreeT$ have the same canopy~\cite{FPR17, PRV17}. The definition of the canopy is recalled in Section~\ref{Sec:préliminaires}.
\smallbreak

\begin{Proposition}\label{ditsyn=tisyn}
Let $n \geq 0$ and $(u,v)\in\dit$. The Tamari interval diagram $(u,v)$ is synchronized if and only if $\rho(\chi(u,v))$ is a synchronized Tamari interval. 
\end{Proposition}
\smallbreak

\begin{proof}
If $(u,v)$ is not synchronized, then there is an index $i\in[n-1]$ such that $u_i = 0$ and $v_{i+1} = 0$. Let $\IntervalPoset := \chi(u,v)$ be the interval-poset associated to $(u,v)$, and $(\TreeS,\TreeT) := \rho(\chi(u,v))$. 
The two binary trees $\TreeS$ and $\TreeT$ are not synchronized if there is at least one letter of some index $j$ in the canopy of the tree $\TreeS$ that is different from the letter of the same index $j$ in the canopy of $\TreeT$. 
Let us show that $(u,v)$ is not synchronized if and only if the binary trees $\TreeS$ and $\TreeT$ are not synchronized.
\smallbreak

The letter $u_i$ is equal to $0$ if and only if there is no descending relation of target $\IntervalPoset_i$ in $\IntervalPoset$, namely, if and only if the node $i$ has no right child in the tree $\TreeS$ (see Section~\ref{sss:TamariIntervals-IntervalPosets}). To summarize, $u_i = 0$ if and only if 
the right subtree of the node $i$ is a leaf oriented to the right. 
Now, as recall in Section~\ref{Subsec:binarytrees}, a leaf linked to the node $i$ is oriented to the right if and only if the $i$-th letter in the canopy corresponding to $\TreeS$ is $1$.
\smallbreak

Symmetrically, $v_{i+1} = 0$ if and only if there is no increasing relation of target $\IntervalPoset_{i+1}$ in $\IntervalPoset$, namely, if and only if the node $i+1$ has no left child in the tree $\TreeT$. 
Then, $v_{i+1} = 0$ if and only if the left subtree of the node $i + 1$ is a leaf oriented to the left. As seen in Section~\ref{Sec:préliminaires}, a leaf linked to the node $i + 1$ is oriented to the left if and only if the $i$-th letter in the canopy corresponding to $\TreeT$ is $0$.
\smallbreak

To conclude, $u_i = 0$ and $v_{i+1} = 0$ if and only if the letter of index $i$ in the canopy of the tree $\TreeS$ is different from the letter of index $i$ in the canopy of the tree $\TreeT$. 
Therefore, $(u,v)$ is not synchronized if and only if the binary trees $\TreeS$ and $\TreeT$ are not synchronized.
\end{proof}
\smallbreak

An interval-poset $\IntervalPoset$ of size $n \geq 3$ is \Def{new} if
\begin{enumerate}
\item there is no decreasing relation of source $\IntervalPoset_n$,
\item there is no increasing relation of source $\IntervalPoset_1$,
\item there is no relation $\IntervalPoset_{i+1}\lhd \IntervalPoset_{j+1}$ and $\IntervalPoset_j\lhd \IntervalPoset_i$ with $i<j$.
\end{enumerate}
The definition of a new interval-poset is given in~\cite{Rog20}.
\smallbreak

For any $n \geq 3$, a Tamari interval diagram $(u,v)$ of size $n$ is \Def{new} if the following conditions are satisfied
\begin{enumerate}[label=(\roman*)]
 \item \label{ditnv1} $0\leq u_i \leq n-i-1$ for all $i\in[n-1]$,
 \item \label{ditnv2} $0\leq v_j\leq j-2$ for all $j\in [2,n]$,
 \item \label{ditnv3} $u_{k}<l-k-1$ or $v_{l}<l-k-1$ for all $k,l \in [n]$ such that $k+1 < l$.
\end{enumerate}
\smallbreak

\begin{Proposition}\label{ditnv=intnv}
Let $n\geq 3$ and $(u,v)\in\dit$. The Tamari interval diagram $(u,v)$ is new if and only if $\chi(u,v)$ is a new interval-poset.
\end{Proposition}
\smallbreak

\begin{proof}
Let us show that $\IntervalPoset := \chi(u,v)$ is not new if and only if $(u,v)$ is not new. Theorem~\ref{bijDIT} leads to three cases.
\begin{itemize}
\item Let us consider the negation of~\ref{ditnv1} of a new Tamari interval diagram by assuming that $u_i = n-i$. By Lemma~\ref{upivpi}, this implies that $\IntervalPoset_n \lhd \IntervalPoset_i$ with $i\in[n-1]$. Reciprocally, if $\IntervalPoset_n \lhd \IntervalPoset_i$ with $i\in[n-1]$, then by the point~\ref{pip1} of interval-poset properties, all vertices between $\IntervalPoset_i$ and $\IntervalPoset_n$ are in decreasing relation to $\IntervalPoset_i$. Since $u_i := \# \{\IntervalPoset_j\in \IntervalPoset : \IntervalPoset_j\lhd \IntervalPoset_i \mbox{ and } i<j\}$, it implies that $u_i = n-i$.

\item Likewise, by Lemma~\ref{upivpi}, if $v_j = j-1$, then $\IntervalPoset_1\lhd \IntervalPoset_j$ with $j\in [2,n]$. By the point~\ref{pip2} of interval-poset properties, we get the converse property. 

\item According to Lemma~\ref{upivpi}, if $u_{i} \geq j - i$, then $\IntervalPoset_j\lhd \IntervalPoset_i$, and if $v_{j+1} \geq j - i$, then $\IntervalPoset_{i+1}\lhd \IntervalPoset_{j+1}$ with $i<j$. We obtain the two converse properties with respectively the point~\ref{pip1} and the point~\ref{pip2} of interval-poset properties. Specifically, by setting $l := j+1$ and $k := i$, we find the formulation of the negation of~\ref{ditnv3} of a new Tamari interval diagram, with $k+1<l$.
\end{itemize}
\end{proof}
\smallbreak

In~\cite{Rog20} it is shown that a Tamari interval is new if and only if the associated interval-poset is new. With Proposition~\ref{ditnv=intnv} we get the following result.
\smallbreak

\begin{Proposition}
Let $n \geq 3$ and $(u,v)\in\dit$. The Tamari interval diagram $(u,v)$ is new if and only if $\rho(\chi(u,v))$ is a new Tamari interval. 
\end{Proposition}
\smallbreak

\begin{Proposition}\label{synnn}
Let $n \geq 3$ and $(u,v)\in\dit$. If $(u,v)$ is synchronized, then $(u,v)$ is not new.
\end{Proposition}
\smallbreak

\begin{proof}
Assume by contradiction that $(u,v)$ is synchronized and new. Since $(u,v)$ is new, one has $u_{i}<n-i$ for $i\in[n-1]$, and $v_{j}<j-1$ for $j\in[2,n]$. In particular, $u_{n-1} = 0$ and $v_{2} = 0$. This implies, since $(u,v)$ is synchronized, that $u_{1} \ne 0$ and $v_{n} \ne 0$. 
Furthermore, since $(u,v)$ is new, Condition~\ref{ditnv3} of a Tamari interval diagram is satisfied. Specifically, for any $k\in[n-2]$, either $u_{k}<1$ or $v_{k+2}<1$. Let us denote by $(\ast)$ this condition.
Assuming that $u_{1} \ne 0$, since $(u,v)$ is synchronized, one has either $u_{2} \ne 0$ or $v_{3} \ne 0$. By $(\ast)$, the second choice is impossible, thus $u_{2} \ne 0$. By the same reasoning, for every $k\in[n-2]$, $u_{k} \ne 0$. However, also by assumption one has $v_{n} \ne 0$. Therefore, $u_{n-2} \ne 0$ and $v_{n} \ne 0$ which is a contradiction with~$(\ast)$.
\end{proof}
\smallbreak

\subsection{Order structure}

Firstly, we endow the set of cubic coordinates with an order relation. Then we show that there is an isomorphism between this poset and the poset of Tamari intervals. The two bijections constructed in the first two parts of Section~\ref{sec:coordonnées-cubiques} allow us to establish this poset isomorphism.
\smallbreak

Let $n \geq 0$ and $c, c' \in\cc$. We set that $c\Leq c'$ if and only if $c_i \leq c'_i$ for all $i \in[n - 1]$. Endowed with $\Leq$, the set $\cc$ is a poset called the \Def{cubic coordinate poset}.
\smallbreak

Recall that the map $\phi$ is defined at the beginning of Section~\ref{subsecCC} and the map $\chi$ is defined at the beginning of Section~\ref{subsecIP}. Let $(\TreeS,\TreeT), (\TreeS',\TreeT')\in \SetIntervalBinaryTrees$ and let $\psi := \phi^{-1} \circ \chi^{-1} \circ \rho^{-1}$ be the map from the Tamari interval poset to the cubic coordinate poset $\cc$.
\smallbreak

For the next results in all this section, let us denote by $c := \psi(\TreeS,\TreeT)$, $c' := \psi(\TreeS',\TreeT')$ and $(u,v) := \phi(c)$, $(u',v') := \phi(c')$, and $\IntervalPoset := \chi(u,v)$, $\IntervalPoset' := \chi(u',v')$.
\smallbreak

\begin{Lemma}\label{sens-direct-isomorphism}
If $(\TreeS',\TreeT')$ covers $(\TreeS,\TreeT)$, then there is a unique different component $c_i$ between $c$ and $c'$ such that $c_i < c'_i$ and there is no cubic coordinate $c''$ different from $c$ and $c'$ such that $c\Leq c''\Leq c'$.
\end{Lemma}
\smallbreak

\begin{proof}
By Lemma~\ref{rotation/ip} we know that $(\TreeS',\TreeT')$ covers $(\TreeS,\TreeT)$ if and only if $\IntervalPoset$ and $\IntervalPoset'$ satisfy either $(\star)$ or $(\diamond)$. 
Let us assume that $\IntervalPoset$ and $\IntervalPoset'$ satisfy either $(\star)$ or $(\diamond)$ for the vertex $\IntervalPoset_i$. By using~\eqref{réciproque} and \eqref{réciproque2}, two cases are possible. 

\begin{itemize}
\item Suppose that $\IntervalPoset$ and $\IntervalPoset'$ satisfy $(\star)$, then since only decreasing relations are added in $\IntervalPoset'$ relative to $\IntervalPoset$, only $u'$ is modified in $(u',v')$ relative to $(u,v)$. Furthermore, since $\IntervalPoset'$ is obtained by adding decreasing relations of target $\IntervalPoset_i$ in $\IntervalPoset$, only the letter $u'_i$ in $u'$ is increased relative to $u$. Moreover, since the number of descending relations added in $\IntervalPoset$ is minimal, there cannot be any Tamari interval diagram between $(u,v)$ and $(u',v')$, and thus no cubic coordinate between $c$ and $c'$. In the end, the image by $\phi^{-1}$ of $(u',v')$ is the cubic coordinate $c'$ with $c'_i = u'_i$ and $c'_j=c_j$ for any $j\ne i$. 

\item Suppose that $\IntervalPoset$ and $\IntervalPoset'$ satisfy $(\diamond)$, 
the arguments are roughly the same, with the difference that this time, only increasing relations are removed in $\IntervalPoset'$ relative to $\IntervalPoset$. We obtain that only the component $c'_{i-1} = - v'_i$ of $c'$ has increased relative to $c$.
\end{itemize}

In both cases, the implication is true.
\end{proof}
\smallbreak

Note that if there is a unique different component $c_i$ between $c$ and $c'$ such that $c_i < c'_i$ and there is no cubic coordinate $c''$ different from $c$ and $c'$ such that $c\Leq c''\Leq c'$, then in particular $c'$ covers $c$. Thus, Lemma~\ref{sens-direct-isomorphism} has the consequence that if $(\TreeS',\TreeT')$ covers $(\TreeS,\TreeT)$, then $c'$ covers~$c$.
\smallbreak

\begin{Lemma}\label{lem:baspetitethautgrand}
Let $n \geq 0$ and  $c, c' \in \cc$. If $c \Leq c'$, then there is a cubic coordinate $c''$ such that $u'' = u$ and $v'' = v'$, where $(u'',v'') := \phi(c'')$.
\end{Lemma}

\begin{proof}
The composition of bijections $\phi^{-1} \circ \chi^{-1}$ associates a pair of words $(u, v)$ to a pair of comparable binary trees $(\TreeS, \TreeT)$ such that $u$ encodes the binary tree $\TreeS$ and $v$ encodes the binary tree $\TreeT$. 
By this composition, $u$ (resp.\ $v$) is obtained by counting in $\TreeS$ (resp.\ $\TreeT$) the number of left (resp.\ right) descendants of each node for the infix order. 
Additionally, we know that if $(\TreeS,\TreeT) \Leq_\TamariIntervalOrder (\TreeS',\TreeT')$, then the interval $(\TreeS, \TreeT')$ is a Tamari interval because we always have $\TreeS \Leq_{\TamariOrder} \TreeS' \Leq_{\TamariOrder} \TreeT'$. The construction of $\phi^{-1} \circ \chi^{-1}$ and the fact that $(\TreeS, \TreeT')$ is a Tamari interval imply that the pair $(u, v')$ is always a Tamari interval diagram. Therefore, $c''$ is a cubic coordinate.
\end{proof}
\smallbreak

For any $c, c' \in \cc$, let
\begin{equation}
	\DiffIndexes^{-}\Par{c, c'} := \Bra{d ~:~ c_d \ne c'_d \mbox{ and  } c'_d \leq 0},
\end{equation}
and 
\begin{equation}
	\DiffIndexes^{+}\Par{c, c'} := \Bra{d ~:~c_d \ne c'_d \mbox{ and  } c_d \geq 0}.
\end{equation}
\smallbreak

Now consider the case where $c$ and $c'$ share either their Tamari diagrams or their associated dual Tamari diagrams, then we have the two following lemmas.
\smallbreak

\begin{Lemma}\label{lem:mêmebasethautdifferent}
Let $n \geq 0$ and $c, c' \in \cc$. If $c \Leq c'$ such that $u = u'$ and $\# \DiffIndexes^{-}\Par{c, c'} = r$, then there is a chain
\begin{equation}
\Par{c = c^{(0)}, c^{(1)}, \dots, c^{(r-1)}, c^{(r)}= c'},
\end{equation}
such that $\# \DiffIndexes^{-}\Par{c^{(i-1)}, c^{(i)}} = 1$ for all $i \in [r]$.
\end{Lemma}
\smallbreak

\begin{proof}
Let 
\begin{equation}
\DiffIndexes^{-}\Par{c, c'} = \{d_1, d_2, \dots, d_r\}
\end{equation} 
with $d_{k-1} < d_{k}$ for all $k \in [2,r]$.
For any $k \in [r]$, let $c^{(k)}$ be a tuple obtained by replacing in $c$ all the components $c_{d_i}$ by the components $c'_{d_i}$ for $i \in [k]$. 
The tuple $c^{(k)}$ is a cubic coordinate. Indeed, by denoting $\phi(c^{(k)})$ by $(u^{(k)}, v^{(k)})$, one has that $u^{(k)} = u = u'$, so the compatibility with $v^{(k)}$ is always satisfied. Therefore, the only thing to check is that $v^{(k)}$ is a dual Tamari diagram. Condition~\ref{condDTD1+} is naturally satisfied. Since $c \leq c'$, one has $v_i \geq v'_i$ for all $i \in [n]$. Therefore, Condition~\ref{condDTD2+} is satisfied because for $i \in [d_k]$ and $j \in [i+1,n]$, $v^{(k)}_i = v'_i$ and $v^{(k)}_j = v_j$, and so $v_j^{(k)} - v_i^{(k)} = v_j - v'_i \geq v_j - v_i \geq j - i$. The word $v^{(k)}$ is then a dual Tamari diagram. Consider the chain 
\begin{equation}
\Par{c = c^{(0)}, c^{(1)}, \dots, c^{(r-1)}, c^{(r)}= c'}.
\end{equation}
For all $i \in [r]$, since we change only one component between $c^{(i-1)}$ and $c^{(i)}$, one has $\# \DiffIndexes^{-}\Par{c^{(i-1)}, c^{(i)}} = 1$.
\end{proof}
\smallbreak

\begin{Lemma}\label{lem:mêmehautetbasdifferent}
Let $n \geq 0$ and $c, c' \in \cc$. If $c \Leq c'$ such that $v = v'$ and $\DiffIndexes^{+}\Par{c, c'} = s$, then there is a chain
\begin{equation}
\Par{c = c^{(0)}, c^{(1)}, \dots, c^{(s-1)}, c^{(s)}= c'},
\end{equation}
such that $\# \DiffIndexes^{+}\Par{c^{(i-1)}, c^{(i)}} = 1$ for all $i \in [s]$.
\end{Lemma}
\smallbreak

\begin{proof}
The proof is similar to the demonstration of Lemma~\ref{lem:mêmebasethautdifferent}. Let 
\begin{equation}
\DiffIndexes^{+}\Par{c, c'} = \{d_1, d_2, \dots, d_s\}
\end{equation} 
with $d_{k-1} < d_{k}$ for all $k \in [2,s]$.
For any $k \in [s]$, let $c^{(k)}$ be a tuple obtained by replacing in $c$ all the components $c_{d_i}$ by the components $c'_{d_i}$ for $i \in [k]$.
As we did in the proof of Lemma~\ref{lem:mêmebasethautdifferent}, we can check that, for any $k \in [s]$, the tuple $c^{(k)}$ is a cubic coordinate.
Then, by consider the chain 
\begin{equation}
\Par{c = c^{(0)}, c^{(1)}, \dots, c^{(s-1)}, c^{(s)}= c'},
\end{equation}
one has that $\# \DiffIndexes^{+}\Par{c^{(i-1)}, c^{(i)}} = 1$ for all $i \in [s]$.

\end{proof}
\smallbreak

\begin{Theorem}\label{morphPoset}
For any $n \geq 0$, the map $\psi$ is a poset isomorphism.
\end{Theorem} 
\smallbreak

\begin{proof}
The map $\psi$ is an isomorphism of posets if $\psi$ and its inverse preserves the partial order. As these relations are transitive, Lemma~\ref{sens-direct-isomorphism} gives the direct implication. Suppose that $c \Leq c'$.
According to Lemma~\ref{lem:baspetitethautgrand}, Lemma~\ref{lem:mêmebasethautdifferent} and Lemma~\ref{lem:mêmehautetbasdifferent} there is always a chain between $c$ and $c'$ such that the components are independently increasing one by one. So we can see what happens when we change only one component $c_i$ by $c'_i$ at any step between $c$ and $c'$.
\smallbreak

Obviously, if $c_i = c'_i$, then $u_i = u'_i$ and $v_{i+1} = v'_{i+1}$ and no changes are made between the corresponding binary tree pairs. Suppose that $c_i < c'_i$, then three cases are possible.

\begin{itemize}
\item Suppose that $c'_i$ is positive and $c_{i}$ is positive or null.
The image by $\phi$ of $c$ and $c'$ differ for the letter $u_{i}$, namely $c'_{i}= u'_{i}$ and $c_{i}= u_{i}$, and $v_{i+1} = v'_{i+1} = 0$.
The difference of a letter $u_{i}$ between $(u,v)$ and $(u',v')$ is directly translated by the map $\chi$: the interval-poset $\IntervalPoset'$ has more decreasing relations of target $\IntervalPoset_i$ than the vertex $\IntervalPoset_i$ in $\IntervalPoset$. By the map $\rho$, it means that to go from the tree $\TreeS$ to the tree $\TreeS'$ at least one right rotation of the edge $(i,j)$ is made, where $j$ is the father of the node $i$ in~$\TreeS$.

\item Symmetrically, assume that $c'_i$ is negative or null, then $c'_i = - v'_{i+1}$, $c_i = - v_{i+1}$ and $u_i = u'_i = 0$. By the map $\chi$, the interval-poset $\IntervalPoset'$ has less decreasing relations of target $\IntervalPoset_{i+1}$ than the vertex $\IntervalPoset_{i+1}$ in $\IntervalPoset$. This implies by $\rho$ that to pass from the tree $\TreeT$ to the tree $\TreeT'$ at least one right rotation of the edge $(k,i+1)$ is made, where $k$ is the right child of the node $i+1$ in $\TreeT$.

\item Finally, with Lemma~\ref{lem:baspetitethautgrand}, the case where $c_i$ is negative and $c'_i$ is positive falls into the conjunction of the two previous cases.
\end{itemize}

Therefore, $c \Leq c'$ implies that $(\TreeS,\TreeT) \Leq_\TamariIntervalOrder (\TreeS',\TreeT')$. 
Hence, the map $\psi$ is an isomorphism of posets.
\end{proof}
\smallbreak

Let us denote by $\lessdot$ the covering relation of the poset $\cc$.
\smallbreak

\begin{Proposition}\label{prop:cubicoordinatecovering}
Let $n \geq 0$ and $c, c' \in \cc$ such that $c \Covered c'$. Then, there is a unique different component between $c$ and $c'$.
\end{Proposition}
\smallbreak

\begin{proof}
It is a consequence of Theorem~\ref{morphPoset} and Lemma~\ref{sens-direct-isomorphism}.
\end{proof}
\smallbreak

The following diagram provides a summary of the applications used in Section~\ref{sec:coordonnées-cubiques}. Recall that $\psi = \phi^{-1} \circ \chi^{-1} \circ \rho^{-1}$, therefore this diagram of poset isomorphisms is commutative.
\begin{equation} \label{equ:}
    \begin{tikzpicture}[Centering,xscale=2.5,yscale=2,font=\small]
        \node(TamIntDiag)at(0,0){$\dit$};
        \node(IntPoset)at(1,0){$\SetIntervalPoset(n)$};
        \node(CubicCoord)at(0,-1){$\cc$};
        \node(TamariInt)at(1,-1){$\SetIntervalBinaryTrees$};
        \draw[Map](TamIntDiag)--(IntPoset)node[midway,above]
            {$\chi$};
        \draw[Map](CubicCoord)--(TamIntDiag)node[midway,right]
            {$\phi$};
        \draw[Map](TamariInt)--(CubicCoord)node[midway,above]
            {$\psi$};
        \draw[Map](IntPoset)--(TamariInt)node[midway,right]
            {$\rho$};
    \end{tikzpicture}
\end{equation}
\smallbreak

A consequence of the poset isomorphism $\psi$ is that the order dimension~\cite{MP90,Tro02} of the poset of Tamari intervals is at most $n-1$.

\section{Geometric properties}\label{sec:réalisation-cubique}
In this section, we give a very natural geometrical realization for the lattices of cubic coordinates. After defining the cells of this realization, we give some properties related to them. Finally, we show that the lattice of the cubic coordinates is EL-shellable.
\smallbreak

\subsection{Cubic realizations}\label{CubicRealCC}

Theorem~\ref{morphPoset} provides a simpler translation of the order relation between two Tamari intervals. We provide the geometrical realization induced by this order relation, which is natural for cubic coordinates. In a combinatorial way we study the cells formed by this realization.
\smallbreak{}

For any $n \geq 0$, the \Def{cubic realization} of $\cc$ is the geometric object
$\CubicReal\Par{\cc}$ defined in the space $\R^{n-1}$ and obtained by placing for each $c \in \cc$ a vertex of coordinates $\Par{c_1, \dots, c_{n-1}}$, and by forming
for each $c, c' \in \cc$ such that $c \Covered c'$ an edge between $c$ and $c'$. Every edge of $\CubicReal\Par{\cc}$ is parallel to some vector in the canonical basis of $\mathbb{R}^{n-1}$.
\smallbreak

\begin{figure}[ht]
    \centering
    \scalebox{1.2}{
    \begin{tikzpicture}[Centering,xscale=1,yscale=1,rotate=-135]
        \draw[Grid](-1,-1)grid(4,4);
        \node[NodeGraph](000)at(0,0){};
        \node[NodeGraph](011)at(1,1){};
        \node[NodeGraph](002)at(0,2){};
        \node[NodeGraph](003)at(0,3){};
        \node[NodeGraph](022)at(2,2){};
        \node[NodeGraph](010)at(1,0){};
        \node[NodeGraph](012)at(1,2){};
        \node[NodeGraph](013)at(1,3){};
        \node[NodeGraph](032)at(3,2){};
        \node[NodeGraph](020)at(2,0){};
        \node[NodeGraph](021)at(2,1){};
        \node[NodeGraph](033)at(3,3){};
        \node[NodeGraph](031)at(3,1){};
        \node[NodeLabeldGraph,above of=000]{$(\bar{1},\bar{2})$};
        \node[NodeLabeldGraph,above of=011]{$(0,\bar{1})$};
        \node[NodeLabeldGraph,above of=002]{$(\bar{1},0)$};
        \node[NodeLabeldGraph,right of=003]{$(\bar{1},1)$};
        \node[NodeLabeldGraph,below of=022]{$(1,0)$};
        \node[NodeLabeldGraph,above of=010]{$(0,\bar{2})$};
        \node[NodeLabeldGraph,right of=012]{$(0,0)$};
        \node[NodeLabeldGraph,below of=013]{$(0,1)$};
        \node[NodeLabeldGraph,below of=032]{$(2,0)$};
        \node[NodeLabeldGraph,above of=020]{$(1,\bar{2})$};
        \node[NodeLabeldGraph,left of=021]{$(1,\bar{1})$};
        \node[NodeLabeldGraph,below of=033]{$(2,1)$};
        \node[NodeLabeldGraph,below of=031]{$(2,\bar{1})$};
        \draw[EdgeGraph](000)--(002);
        \draw[EdgeGraph](000)--(010);
        \draw[EdgeGraph](011)--(012);
        \draw[EdgeGraph](002)--(003);
        \draw[EdgeGraph](002)--(012);
        \draw[EdgeGraph](012)--(022);
        \draw[EdgeGraph](021)--(022);
        \draw[EdgeGraph](003)--(013);
        \draw[EdgeGraph](022)--(032);
        \draw[EdgeGraph](010)--(011);
        \draw[EdgeGraph](010)--(020);
        \draw[EdgeGraph](012)--(013);
        \draw[EdgeGraph](031)--(021);
        \draw[EdgeGraph](011)--(021);
        \draw[EdgeGraph](032)--(033);
        \draw[EdgeGraph](020)--(021);
        \draw[EdgeGraph](013)--(033);
        \draw[EdgeGraph](031)--(032);
    \end{tikzpicture}}
    \caption{$\CubicReal(\ccc(3))$.}
\label{RealCube}
\end{figure}
\smallbreak

Figure~\ref{RealCube} shows the cubic realization of $\ccc(3)$, where the elements are the vertices and the edges are the covering relations. Figure~\ref{DessinCC4} shows the cubic realization of $\ccc(4)$. In these drawings the negative sign components are denoted with a bar.
\smallbreak 

In algebraic topology, to define the tensor products of $A_{\infty}$-algebras, one can use a cell complex called the \Def{diagonal of the associahedron}. This complex has notably been studied by Loday~\cite{Lod11}, by Saneblidze and Umble~\cite{SU04}, and by Markl and Shnider~\cite{MS06}. More recently, there is a description of this object in~\cite{MTTV21}.
The realization of this complex seems to be identical to the cubic realization, up to continuous deformation.
\smallbreak

\subsection{Covering map}\label{sss:MinimalIncreasingMap}
Let $n \geq 0$. We define the set of
\begin{itemize}
    \item \Def{input-wings} as the set $\InputWings(\cc)$ containing any $c \in
    \cc$ which covers exactly $n-1$ elements,

    \item \Def{output-wings} as the set $\OutputWings(\cc)$ containing any $c \in
    \cc$ which is covered by exactly $n-1$ elements.
\end{itemize}
\smallbreak

Let $n \geq 0$ and $c\in\cc$.  
For $i \in [n-1]$, the \Def{covering map} $\uparrow_{i}$ sends $c$ to its covering differing only at index $i$, when such covering exists. We denote by $\uparrow c_{i}$ the letter which differs in $\uparrow_{i}(c)$.
\smallbreak

In particular, for $n\geq 0$, a cubic coordinate $c$ of size $n$ is an output-wing if for any $i\in[n-1]$, $\uparrow_i(c)$ is well-defined.
\smallbreak

Let $n \geq 0$ and $c\in\cc$, and $(u,v) := \phi(c)$.
If $\uparrow c_{i}$ is positive, then the letter $u_i$ increases and becomes equal to $\uparrow c_{i}$ and $v_{i+1}$ is equal to $0$. Then, we define $\uparrow u_{i} := \uparrow c_{i}$.
If $\uparrow c_{i}$ is negative or null, then $v_{i+1}$ decreases and becomes equal to $|\uparrow c_{i}|$ and $u_i$ is equal to $0$. Then, we set $\downarrow v_{i+1} := -\uparrow c_{i}$.
\smallbreak

\begin{Lemma} \label{lemmequifaitgagnerdutps}
Let $n\geq 0$ and $c\in\cc$, and $i\in[n-1]$ such that $\uparrow_i(c)$ is well-defined. Then, 
\begin{enumerate}[label=(\roman*)]
\item if $c_i < 0$, then $\uparrow c_{i}\leq 0$,
\item if $c_i \geq 0$, then $\uparrow c_{i} > 0$.
\end{enumerate}
\end{Lemma}
\smallbreak

\begin{proof}
Let us show the first implication, the second being obvious because the covering map always strictly increases a component.
Let $c_i < 0$, and let $c'$ be the $(n-1)$-tuple such that $c'_i = 0$ and $c'_j = c_j$ for any $j\ne i$, with $j \in[n-1]$.
By Lemma~\ref{lemmeutilepartout}, $c'$ is a cubic coordinate.
As $c \leq c'$ and they differ only at the $i$-th component, by the definition of $\uparrow_i(c)$, we have $c \leq \uparrow_i(c) \leq c'$, thus $\uparrow c_i \leq c'_i = 0$.
\end{proof}
\smallbreak

Let $c \in \cc$. For all $i \in [n]$, let 
\begin{equation}\label{coverRtoL}
\Uparrow_{i} (c) := \uparrow_{i}(\uparrow_{i+1}\dots (\uparrow_{n-1}(\uparrow_{n}(c)))),
\end{equation}
with the convention that $\uparrow_{n}(c) := c$.
For instance, for $c \in \ccc(5)$, $\Uparrow_{2} (c) = \uparrow_{2}(\uparrow_{3}(\uparrow_{4}(\uparrow_{5}(c))))$.

\begin{Lemma}\label{existeMax}
Let $n \geq 0$ and $c \in \OutputWings(\cc)$. For all $i\in[n]$, $\Uparrow_{i}(c)$ is a cubic coordinate.
\end{Lemma}
\smallbreak

\begin{proof}
For $i = n$, one has by convention that $\Uparrow_{n} (c)$ is a cubic coordinate. Let us suppose that for $i\in[n-1]$, $\Uparrow_{i+1} (c)$ is a cubic coordinate, and let us show that $\Uparrow_{i} (c)$ is also a cubic coordinate.
Depending on the sign of $\Uparrow_{i+1} (c)_i$, two cases are possible.
\smallbreak

Suppose that $\Uparrow_{i+1} (c)_i < 0$.
In this case, consider $c'$ the $(n-1)$-tuple obtained from $\Uparrow_{i+1} (c)$ by replacing the component $\Uparrow_{i+1} (c)_i$ by $0$. By Lemma~\ref{lemmeutilepartout}, $c'$ is a cubic coordinate.
Since $\Uparrow_{i+1} (c)_i < 0$ one has $\Uparrow_{i+1} (c) \Leq c'$.
If $c'$ covers $\Uparrow_{i+1} (c)$, then $c' = \Uparrow_{i} (c)$. Otherwise, it is always possible to find another cubic coordinate $c''$ between $\Uparrow_{i+1} (c)$ and $c'$ such that $c'' = \Uparrow_{i} (c)$.
In both cases, $\Uparrow_{i} (c)$ is a cubic coordinate.
\smallbreak

Suppose that $\Uparrow_{i+1} (c)_i \geq 0$. Let us set $(u,v) := \phi(c)$, and $(x,y) := \phi(\Uparrow_{i+1} (c))$. Since $u_i$ is not changed yet in $x$, one has $x_{i} = u_{i}$. Due to Condition~\ref{condDT2+} of a Tamari diagram and the compatibility condition, there are two configurations, involving indices, which can make contradiction with the fact that $(x,y)$ is still a Tamari interval diagram when $x_{i}$ becomes $\uparrow x_i$.
\begin{enumerate}[label=({\arabic*})]
\item \label{existeMax1} If there is an index $j$ such that $1\leq i < j \leq n$ and $y_j \geq j - i$ in $y$, then, since $y_j < v_j$, one has $v_j \geq j - i$ in $v$. By the compatibility condition, that implies $u_i < j - i$ in $u$. Moreover, since $c$ is assumed to be an output-wing, $u_i < j - i -1$ in $u$, so that $u_i$ can be increased. This inequality remains true in $x$.

\item \label{existeMax2} If there is an index $h$ such that $1\leq i-h \leq u_h$, by Condition~\ref{condDT2+} of a Tamari diagram, $u_i \leq u_h - i + h$ in $u$. This remains true in $x$ because components with index smaller than $i$ remain unchanged between $c$ and $\Uparrow_{i+1} (c)$. Furthermore, since $c$ is an output-wing, then $u_i < u_h - i + h$. This inequality remains true for $\Uparrow_{i+1} (c)$.

\end{enumerate}
\smallbreak

With these two configurations, let us build a cubic coordinate $c'$ different from $\Uparrow_{i+1} (c)$ only for $\Uparrow_{i+1} (c)_i$, depending on which choices are available to increase $u_i$. Let us set $(u',v') := \phi(c')$.
\smallbreak

\begin{enumerate}[label=({\alph*})]
\item \label{casa} Suppose there is a $j$ satisfying~\ref{existeMax1}, and there is no $h$ satisfying~\ref{existeMax2} in $\Uparrow_{i+1} (c)$. In this case, by choosing the minimal index $j$ such that~\ref{existeMax1} holds, we set $u'_i := j - i - 1$ in $c'$. Thus, $u'_i$ is also minimized, and since $u'_i < j-i$, the compatibility condition is satisfied in $c'$. Furthermore, since $\Uparrow_{i+1} (c)$ is assumed to be a cubic coordinate, all conditions in a Tamari diagram and a dual Tamari diagram are satisfied for $c'$. Therefore, our candidate $c'$ is a cubic coordinate. Note that in the construction of $c'$, other possible not minimal $j$ satisfying~\ref{existeMax1} will not cause any problem.

\item \label{casb} Suppose there is an $h$ satisfying~\ref{existeMax2}, and there is no $j$ satisfying~\ref{existeMax1} in $\Uparrow_{i+1} (c)$. Then, by choosing the minimal index $h$ such that~\ref{existeMax2} holds, we set $u'_i := u'_h - i + h$. Therefore, Condition~\ref{condDT2+} of a Tamari diagram is satisfied for $u'$. Also, by Condition~\ref{condDT1+} of a Tamari diagram, $u'_h \leq n - h$ which implies $u'_i \leq n -i$. Finally, the compatibility condition is also satisfied because it was assumed that there was no $j$ satisfying~\ref{existeMax1}. The tuple $c'$ is thus a cubic coordinate. As for the previous case, other possible not minimal $h$ satisfying~\ref{existeMax2} will not cause any problem.

\item \label{casc} Suppose there is a $j$ and an $h$ satisfying~\ref{existeMax1} and~\ref{existeMax2} in $\Uparrow_{i+1} (c)$. In this case, we set $u'_i := \min \{ u'_h - i + h,~ j - i - 1\}$. By the two previous cases, the tuple $c'$ is a cubic coordinate.

\item \label{casd} Otherwise, we set $u'_i := n - i$. The tuple $c'$ is a cubic coordinate.
\end{enumerate}
\smallbreak


In any case, for $u'_i$ fixed in $c'$, either $c'$ covers $\Uparrow_{i+1} (c)$, and so $c' = \Uparrow_{i} (c)$, or there is a cubic coordinate $c''$ between $\Uparrow_{i+1} (c)$ and $c'$ such that $c'' = \Uparrow_{i} (c)$. In both cases, $\Uparrow_{i} (c)$ is a cubic coordinate, and differs by only one component from $c'$.
\end{proof}
\smallbreak

Let $n \geq 0$ and $c \in \OutputWings(\cc)$. The cubic coordinate $\Uparrow_{1} (c)$ is the \Def{corresponding input-wing} of $c$ (the name comes from a corollary of Theorem~\ref{nombredesommetdanscellule}).
For instance $c = (0,-1,1,-1,-5,0,1,-1,-3)$ is an output-wing, and its corresponding input-wing is $\Uparrow_{1} (c) = (1,0,2,0,-4,3,2,0,-2)$.
By Lemma~\ref{existeMax} such an element does exist. Note that performing the covering map on $c$ in a different order than the one prescribed by~\eqref{coverRtoL} does not always result in the corresponding input-wing. This observation can already be made on the two pentagons of Figure~\ref{RealCube}.
\smallbreak

\subsection{Cells and synchronized cubic coordinates}
In Figure~\ref{RealCube} and Figure~\ref{DessinCC4}, we notice that a "cellular" organization appears. Thanks to the cubic coordinates, a combinatorial definition of these cells is provided. The aim is to have a better understanding of the realization of the cubic coordinate posets as a geometrical object.
\smallbreak

For any $n \geq 0$, let $c, c' \in \cc$ such that $c \Leq c'$.
A \Def{cell} is the set of points
\begin{equation}
    \Angle{c, c'} := \Bra{x \in \R^{n-1} : c_i \leq x_i \leq c'_i \mbox{ for all } i \in [n-1]}.
\end{equation}
By the definition, a cell is an orthotope, that is, a parallelotope whose edges are all mutually orthogonal or parallel. The \Def{dimension} $\dim\Angle{c, c'}$ of a
cell $\Angle{c, c'}$ is its dimension as an orthotope and it satisfies $\dim \Angle{c, c'} =
\# \DiffIndices(c, c')$, where $\DiffIndices(c, c') := \{ d~:~c_d \neq c'_d \}$. 
\smallbreak

From now on, we denote by $c^{\Outp}$ any output-wing and by $c^{\Inp}$ its corresponding input-wing. Any particular cell $\langle c^{\Outp},c^{\Inp} \rangle$ formed by an output-wing and by its corresponding input-wing is called a \Def{cell-wing}.
\smallbreak

A consequence of Lemma~\ref{lemmequifaitgagnerdutps} is that for any cell-wing $\langle c^{\Outp},c^{\Inp} \rangle$ of dimension $n - 1$, for all $i\in [n-1]$, 
\begin{enumerate}[label=(\roman*)]
\item if $c^{\Outp}_{i}<0$, then $c^{\Inp}_{i}\leq 0$,
\item if $c^{\Outp}_{i}\geq 0$, then $c^{\Inp}_{i}>0$.
\end{enumerate}
\smallbreak

\begin{Theorem}\label{nombredesommetdanscellule}
Let $n \geq 1$ and $\langle c^{\Outp},c^{\Inp} \rangle$ be a cell-wing of dimension $n-1$, and $c$ be a $(n-1)$-tuple such that for all $i\in [n-1]$, the component $c_i$ is equal either to $c^{\Outp}_i$ or to $c^{\Inp}_i$. Then $c$ is a cubic coordinate.
\end{Theorem}
\smallbreak

\begin{proof}
If all the components of $c$ are equal to those of $c^{\Outp}$ (resp.\ to those of $c^{\Inp}$), then $c$ is a cubic coordinate. Suppose this is not the case, meaning that $c$ has components of $c^{\Outp}$ and $c^{\Inp}$.
\smallbreak

Let us denote $(u^{\Outp}_i,v^{\Outp}_{i+1})$ (resp.\ $(u^{\Inp}_i,v^{\Inp}_{i+1})$) the pair of letters corresponding to $c^{\Outp}_i$ (resp.\ $c^{\Inp}_i$) and $(u_i,v_{i+1})$ the one corresponding to $c_i$ for any $i\in [n-1]$. 
By hypothesis on $c^{\Outp}$ and $c^{\Inp}$ the letter $u_i$ which is equal to $u^{\Outp}_i$ or $u^{\Inp}_i$ satisfies $0\leq u_i\leq n-i$ for any $i\in[n]$. Similarly, the letter $v_i$ which is equal to $v^{\Outp}_i$ or $v^{\Inp}_i$ satisfies $0\leq v_i \leq i-1$ for any $i\in[n]$.
In order to show that $c$ is a cubic coordinate, let us prove that $u$ satisfies Condition~\ref{condDT2+} of a Tamari diagram, $v$ satisfies Condition~\ref{condDTD2+} of a dual Tamari diagram and $(u,v)$ satisfies the compatibility condition.
\smallbreak

\begin{enumerate}[label=({\roman*})]
\item Let us show that for any choice of letters $u_i$ and $u_{i+j}$ with $i \in [n]$ and $j \in [0,u_i]$ one has $u_{i+j} \leq u_i - j$.
\begin{itemize}
\item If $u_i$ and $u_{i+j}$ are equal respectively to $u^{\Outp}_i$ and to $u^{\Outp}_{i+j}$ (resp.\ to $u^{\Inp}_i$ and to $u^{\Inp}_{i+j}$), then Condition~\ref{condDT2+} of a Tamari diagram is satisfies because $c^{\Outp}$ (resp.\ $c^{\Inp}$) is a cubic coordinate. 

\item Suppose that $u_i = u^{\Inp}_i$ and $u_{i+j} = u^{\Outp}_{i+j}$. By the definition of $c^{\Inp}$ one has $u^{\Outp}_{i+j} < u^{\Inp}_{i+j}$. However $u^{\Inp}_{i+j} \leq u^{\Inp}_i -j$ because $c^{\Inp}$ is a cubic coordinate. Therefore, Condition~\ref{condDT2+} of a Tamari diagram is satisfied.

\item Suppose that $u_i = u^{\Outp}_i$ and $u_{i+j} = u^{\Inp}_{i+j}$. Let $c' := \Uparrow_{i+j} (c^{\Outp})$. 
According to Lemma~\ref{existeMax} $c'$ is a cubic coordinate such that $c'_i = u^{\Outp}_i$ and $c'_{i+j} = u^{\Inp}_{i+j}$. Since Condition~\ref{condDT2+} of a Tamari diagram is satisfied for $c'$, it must also be satisfied for $c$.
\end{itemize}

\item Condition~\ref{condDTD2+} of a dual Tamari diagram is satisfied with similar arguments given for the previous case, applied to the dual Tamari diagram $v$.

\item Rather than showing the compatibility condition as it is stated, let us show the contrapositive. That is, for every $1 \leq i < j \leq n$ such that $v_j \geq j - i$, let us show that $u_i < j-i$.
\begin{itemize}
\item Clearly, if $u_i$ and $v_j$ are equal to $u^{\Outp}_i$ and $v^{\Outp}_j$ (resp.\ to $u^{\Inp}_i$ and $v^{\Inp}_j$), then the compatibility condition is satisfied.

\item Suppose that $u_i = u^{\Outp}_i$ and $v_j = v^{\Inp}_j$. If $v^{\Inp}_j \geq j-i$, then for $c^{\Outp}$ one has $v^{\Outp}_j \geq j-i$ because $v^{\Inp}_j < v^{\Outp}_j$. Since $c^{\Outp}$ is a cubic coordinate, this implies that $u^{\Outp}_i < j-i$.

\item Suppose that $u_i = u^{\Inp}_i$ and $v_j = v^{\Outp}_j$. If $v^{\Outp}_j \geq j-i$, then for all $k \in [i, j-1]$, $u^{\Outp}_k < j-k$ because $c^{\Outp}$ is a cubic coordinate and then satisfies the compatibility condition. Moreover, since $c^{\Outp} \in \OutputWings(\cc)$ each component can be minimally increased independently of the others, thus  $u^{\Outp}_k < j-k-1$ for all $k \in [i, j-1]$. For the same reason $u_{i+h} < u_i - h$ for all $h \in [0, u_i]$.
These two reasons imply that if one builds the cubic coordinate $c' = \Uparrow_{i} (c^{\Outp})$, then by the definition of the covering map one has $c'_i = u'_i < j -i$, because at worst, the covering map sends $u^{\Outp}_i$ to $j - i - 1$ (we have already seen this in the proof of Lemma~\ref{existeMax}). 
However, by the definition of $c^{\Inp}$ one has $u^{\Inp}_i = u'_i$, that is $u^{\Inp}_i < j - i$. Therefore, the compatibility condition between $u^{\Inp}$ and $v^{\Outp}_j$ is satisfied for $c$.
\end{itemize}
\end{enumerate}
\smallbreak

Thus, for all choices of letters of $u$ and $v$ one has that $c$ is a cubic coordinate.
\end{proof}
\smallbreak

One of the direct consequences of Theorem~\ref{nombredesommetdanscellule} is that for every cell-wing $\langle c^{\Outp},c^{\Inp} \rangle$, at least $2^{n-1}$ cubic coordinates belong to this cell. 
\smallbreak

This theorem also implies that a corresponding input-wing covers $n-1$ cubic coordinates, and so is in particular an input-wing. 
\smallbreak

Moreover, due to the fact the Tamari interval lattice is self-dual, the number of output-wings is equal to the number of input-wings. Therefore, by Theorem~\ref{morphPoset}, an input-wing is always a corresponding input-wing of some output-wing. 
\smallbreak

Let $n\geq 0$, and $\epsilon \in \{-1,1\}^{n-1}$, and $c\in\cc$.
The \Def{$\epsilon$-region} of $c$ is the set
\begin{equation}\label{eqreg}
\mathcal{R}_{\epsilon}(c) := \{(x_1, \dots, x_{n-1})\in \mathbb{R}^{n-1} ~:~ x_i < c_i \mbox{ if } \epsilon_i = -1,~ x_i > c_i \mbox{ otherwise}\}.
\end{equation}
The cubic coordinate $c$ is \Def{external} if there is $\epsilon\in \{-1,1\}^{n-1}$ such that $\cc\cap\mathcal{R}_{\epsilon}(c) = \emptyset$. The $\epsilon$-region $\mathcal{R}_{\epsilon}(c)$ is then \Def{empty}. Otherwise, $c$ is \Def{internal}.
\smallbreak

\begin{Proposition}\label{intnn}
Let $n \geq 0$ and $c \in \cc$. If $c$ is internal, then $\phi(c)$ is a new Tamari interval diagram.
\end{Proposition}
\smallbreak

\begin{proof}
Instead, let us show that if $\phi(c)$ is not new, then $c$ is external. Let us denote $(u_{i},v_{i+1})$ the pair of letters corresponding to $c_{i}$ by the map $\phi$ for $i\in[n-1]$. 
\smallbreak

Tamari interval diagram $\phi(c)$ is not new if there is 
\begin{enumerate}[label={(\arabic*)}]
\item\label{ccnn1} either $i\in [n-1]$ such that $u_{i} = n-i$,
\item\label{ccnn2} or $j\in [2,n]$ such that $v_{j} = j - 1$,
\item\label{ccnn3} or $k,l\in[n]$ such that $u_{k} = l-k-1$ and $v_{l} = l-k-1$ with $k+1 <l$.
\end{enumerate}
Suppose there is some $i$ satisfying~\ref{ccnn1}, then there cannot be a cubic coordinate $c'$ such that $c'_{i} > c_{i}$ because, by the definition of a Tamari diagram, $c'_{i}\leq n-i$.
Similarly, if we assume that there is $j$ satisfying~\ref{ccnn2}, then there cannot be a cubic coordinate $c'$ such that $c'_{j-1}<c_{j-1}$ because by the definition of a dual Tamari diagram, $c'_{j-1}\geq 1-j$.
If~\ref{ccnn3} is satisfied, then there cannot be a cubic coordinate $c'$ such that $c'_{k}>c_{k}$ and $c'_{l-1}<c_{l-1}$. Indeed, if the letters $u_{k}$ and $v_{l}$ are increased in $c$, then the compatibility condition is contradicted, so the result cannot be a cubic coordinate.
Since in each case at least one $\epsilon$-region is empty, $c$ is external.
\end{proof}
\smallbreak

\begin{Proposition}\label{synext}
Let $n \geq 0$ and $c\in\ccs$. Then $c$ is external.
\end{Proposition}
\smallbreak

\begin{proof}
By Proposition~\ref{synnn} we know that if $c$ is synchronized, then $\phi(c)$ is not new. Now, we just saw from Proposition~\ref{intnn} that if $\phi(c)$ is not new, then $c$ is external.
\end{proof}
\smallbreak

\begin{figure}[h!]
\centering
    \scalebox{0.7}{
\begin{tikzpicture}[Centering,xscale=.7,yscale=.7,
    x={(0,-1cm)}, y={(-2cm,-2cm)}, z={(3cm,-3cm)}]
    \foreach \x in {-1,...,3} {
        \foreach \y in {-2,...,2} {
            \foreach \z in {-3,...,1} {
                \draw[LineGrid](0,\y,\z)--(\x,\y,\z);
                \draw[LineGrid](\x,0,\z)--(\x,\y,\z);
                \draw[LineGrid](\x,\y,0)--(\x,\y,\z);
            }
        }
    }
\draw[EdgeGraph](0, 0, 0)--(1, 0, 0);
\draw[EdgeGraph](0, 0, 0)--(0, 1, 0);
\draw[EdgeGraph](0, 0, 0)--(0, 0, 1);
\draw[EdgeGraph](0, 0, -1)--(1, 0, -1);
\draw[EdgeGraph](0, 0, -1)--(0, 1, -1);
\draw[EdgeGraph](0, 0, -1)--(0, 0, 0);
\draw[EdgeGraph](0, 0, -2)--(1, 0, -2);
\draw[EdgeGraph](0, 0, -2)--(0, 1, -2);
\draw[EdgeGraph](0, 0, -2)--(0, 0, -1);
\draw[EdgeGraph](0, 0, -3)--(1, 0, -3);
\draw[EdgeGraph](0, 0, -3)--(0, 1, -3);
\draw[EdgeGraph](0, 0, -3)--(0, 0, -2);
\draw[EdgeGraph](0, -1, 0)--(1, -1, 0);
\draw[EdgeGraph](0, -1, 0)--(0, 0, 0);
\draw[EdgeGraph](0, -1, 0)--(0, -1, 1);
\draw[EdgeGraph](0, -1, -2)--(1, -1, -2);
\draw[EdgeGraph](0, -1, -2)--(0, 0, -2);
\draw[EdgeGraph](0, -1, -2)--(0, -1, 0);
\draw[EdgeGraph](0, -1, -3)--(1, -1, -3);
\draw[EdgeGraph](0, -1, -3)--(0, 0, -3);
\draw[EdgeGraph](0, -1, -3)--(0, -1, -2);
\draw[EdgeGraph](0, -2, 0)--(1, -2, 0);
\draw[EdgeGraph](0, -2, 0)--(0, -1, 0);
\draw[EdgeGraph](0, -2, 0)--(0, -2, 1);
\draw[EdgeGraph](0, -2, -3)--(1, -2, -3);
\draw[EdgeGraph](0, -2, -3)--(0, -1, -3);
\draw[EdgeGraph](0, -2, -3)--(0, -2, 0);
\draw[EdgeGraph](-1, 0, 0)--(0, 0, 0);
\draw[EdgeGraph](-1, 0, 0)--(-1, 1, 0);
\draw[EdgeGraph](-1, 0, 0)--(-1, 0, 1);
\draw[EdgeGraph](-1, 0, -1)--(0, 0, -1);
\draw[EdgeGraph](-1, 0, -1)--(-1, 1, -1);
\draw[EdgeGraph](-1, 0, -1)--(-1, 0, 0);
\draw[EdgeGraph](-1, 0, -3)--(0, 0, -3);
\draw[EdgeGraph](-1, 0, -3)--(-1, 1, -3);
\draw[EdgeGraph](-1, 0, -3)--(-1, 0, -1);
\draw[EdgeGraph](-1, -2, 0)--(0, -2, 0);
\draw[EdgeGraph](-1, -2, 0)--(-1, 0, 0);
\draw[EdgeGraph](-1, -2, 0)--(-1, -2, 1);
\draw[EdgeGraph](-1, -2, -3)--(0, -2, -3);
\draw[EdgeGraph](-1, -2, -3)--(-1, 0, -3);
\draw[EdgeGraph](-1, -2, -3)--(-1, -2, 0);
\draw[EdgeGraph](0, 0, 1)--(1, 0, 1);
\draw[EdgeGraph](0, 0, 1)--(0, 2, 1);
\draw[EdgeGraph](0, -1, 1)--(1, -1, 1);
\draw[EdgeGraph](0, -1, 1)--(0, 0, 1);
\draw[EdgeGraph](0, -2, 1)--(1, -2, 1);
\draw[EdgeGraph](0, -2, 1)--(0, -1, 1);
\draw[EdgeGraph](-1, 0, 1)--(0, 0, 1);
\draw[EdgeGraph](-1, 0, 1)--(-1, 2, 1);
\draw[EdgeGraph](-1, -2, 1)--(0, -2, 1);
\draw[EdgeGraph](-1, -2, 1)--(-1, 0, 1);
\draw[EdgeGraph](0, 1, 0)--(2, 1, 0);
\draw[EdgeGraph](0, 1, 0)--(0, 2, 0);
\draw[EdgeGraph](0, 1, -1)--(2, 1, -1);
\draw[EdgeGraph](0, 1, -1)--(0, 2, -1);
\draw[EdgeGraph](0, 1, -1)--(0, 1, 0);
\draw[EdgeGraph](0, 1, -2)--(2, 1, -2);
\draw[EdgeGraph](0, 1, -2)--(0, 1, -1);
\draw[EdgeGraph](0, 1, -3)--(2, 1, -3);
\draw[EdgeGraph](0, 1, -3)--(0, 1, -2);
\draw[EdgeGraph](-1, 1, 0)--(0, 1, 0);
\draw[EdgeGraph](-1, 1, 0)--(-1, 2, 0);
\draw[EdgeGraph](-1, 1, -1)--(0, 1, -1);
\draw[EdgeGraph](-1, 1, -1)--(-1, 2, -1);
\draw[EdgeGraph](-1, 1, -1)--(-1, 1, 0);
\draw[EdgeGraph](-1, 1, -3)--(0, 1, -3);
\draw[EdgeGraph](-1, 1, -3)--(-1, 1, -1);
\draw[EdgeGraph](0, 2, 0)--(3, 2, 0);
\draw[EdgeGraph](0, 2, 0)--(0, 2, 1);
\draw[EdgeGraph](0, 2, -1)--(3, 2, -1);
\draw[EdgeGraph](0, 2, -1)--(0, 2, 0);
\draw[EdgeGraph](-1, 2, 0)--(0, 2, 0);
\draw[EdgeGraph](-1, 2, 0)--(-1, 2, 1);
\draw[EdgeGraph](-1, 2, -1)--(0, 2, -1);
\draw[EdgeGraph](-1, 2, -1)--(-1, 2, 0);
\draw[EdgeGraph](0, 2, 1)--(3, 2, 1);
\draw[EdgeGraph](-1, 2, 1)--(0, 2, 1);
\draw[EdgeGraph](1, 0, 0)--(2, 0, 0);
\draw[EdgeGraph](1, 0, 0)--(1, 0, 1);
\draw[EdgeGraph](1, 0, -1)--(2, 0, -1);
\draw[EdgeGraph](1, 0, -1)--(1, 0, 0);
\draw[EdgeGraph](1, 0, -2)--(2, 0, -2);
\draw[EdgeGraph](1, 0, -2)--(1, 0, -1);
\draw[EdgeGraph](1, 0, -3)--(2, 0, -3);
\draw[EdgeGraph](1, 0, -3)--(1, 0, -2);
\draw[EdgeGraph](1, -1, 0)--(2, -1, 0);
\draw[EdgeGraph](1, -1, 0)--(1, 0, 0);
\draw[EdgeGraph](1, -1, 0)--(1, -1, 1);
\draw[EdgeGraph](1, -1, -2)--(2, -1, -2);
\draw[EdgeGraph](1, -1, -2)--(1, 0, -2);
\draw[EdgeGraph](1, -1, -2)--(1, -1, 0);
\draw[EdgeGraph](1, -1, -3)--(2, -1, -3);
\draw[EdgeGraph](1, -1, -3)--(1, 0, -3);
\draw[EdgeGraph](1, -1, -3)--(1, -1, -2);
\draw[EdgeGraph](1, -2, 0)--(1, -1, 0);
\draw[EdgeGraph](1, -2, 0)--(1, -2, 1);
\draw[EdgeGraph](1, -2, -3)--(1, -1, -3);
\draw[EdgeGraph](1, -2, -3)--(1, -2, 0);
\draw[EdgeGraph](1, 0, 1)--(3, 0, 1);
\draw[EdgeGraph](1, -1, 1)--(3, -1, 1);
\draw[EdgeGraph](1, -1, 1)--(1, 0, 1);
\draw[EdgeGraph](1, -2, 1)--(1, -1, 1);
\draw[EdgeGraph](2, 0, 0)--(3, 0, 0);
\draw[EdgeGraph](2, 0, 0)--(2, 1, 0);
\draw[EdgeGraph](2, 0, -1)--(3, 0, -1);
\draw[EdgeGraph](2, 0, -1)--(2, 1, -1);
\draw[EdgeGraph](2, 0, -1)--(2, 0, 0);
\draw[EdgeGraph](2, 0, -2)--(3, 0, -2);
\draw[EdgeGraph](2, 0, -2)--(2, 1, -2);
\draw[EdgeGraph](2, 0, -2)--(2, 0, -1);
\draw[EdgeGraph](2, 0, -3)--(2, 1, -3);
\draw[EdgeGraph](2, 0, -3)--(2, 0, -2);
\draw[EdgeGraph](2, -1, 0)--(3, -1, 0);
\draw[EdgeGraph](2, -1, 0)--(2, 0, 0);
\draw[EdgeGraph](2, -1, -2)--(3, -1, -2);
\draw[EdgeGraph](2, -1, -2)--(2, 0, -2);
\draw[EdgeGraph](2, -1, -2)--(2, -1, 0);
\draw[EdgeGraph](2, -1, -3)--(2, 0, -3);
\draw[EdgeGraph](2, -1, -3)--(2, -1, -2);
\draw[EdgeGraph](2, 1, 0)--(3, 1, 0);
\draw[EdgeGraph](2, 1, -1)--(3, 1, -1);
\draw[EdgeGraph](2, 1, -1)--(2, 1, 0);
\draw[EdgeGraph](2, 1, -2)--(3, 1, -2);
\draw[EdgeGraph](2, 1, -2)--(2, 1, -1);
\draw[EdgeGraph](2, 1, -3)--(2, 1, -2);
\draw[EdgeGraph](3, 0, 0)--(3, 1, 0);
\draw[EdgeGraph](3, 0, 0)--(3, 0, 1);
\draw[EdgeGraph](3, 0, -1)--(3, 1, -1);
\draw[EdgeGraph](3, 0, -1)--(3, 0, 0);
\draw[EdgeGraph](3, 0, -2)--(3, 1, -2);
\draw[EdgeGraph](3, 0, -2)--(3, 0, -1);
\draw[EdgeGraph](3, -1, 0)--(3, 0, 0);
\draw[EdgeGraph](3, -1, 0)--(3, -1, 1);
\draw[EdgeGraph](3, -1, -2)--(3, 0, -2);
\draw[EdgeGraph](3, -1, -2)--(3, -1, 0);
\draw[EdgeGraph](3, 0, 1)--(3, 2, 1);
\draw[EdgeGraph](3, -1, 1)--(3, 0, 1);
\draw[EdgeGraph](3, 1, 0)--(3, 2, 0);
\draw[EdgeGraph](3, 1, -1)--(3, 2, -1);
\draw[EdgeGraph](3, 1, -1)--(3, 1, 0);
\draw[EdgeGraph](3, 1, -2)--(3, 1, -1);
\draw[EdgeGraph](3, 2, 0)--(3, 2, 1);
\draw[EdgeGraph](3, 2, -1)--(3, 2, 0);
\node[NodeGraph](000)at(0,0,0){};
\node[NodeGraph]at(0,0,-1){};
\node[NodeGraph]at(0,0,-2){};
\node[NodeGraph]at(0,0,-3){};
\node[NodeGraph]at(0,-1,0){};
\node[NodeGraph]at(0,-1,-2){};
\node[NodeGraph]at(0,-1,-3){};
\node[NodeGraph]at(0,-2,0){};
\node[NodeGraph]at(0,-2,-3){};
\node[NodeGraph]at(-1,0,0){};
\node[NodeGraph]at(-1,0,-1){};
\node[NodeGraph](-10-3)at(-1,0,-3){};
\node[NodeGraph](-1-20)at(-1,-2,0){};
\node[NodeGraph](-1-2-3)at(-1,-2,-3){};
\node[NodeGraph]at(0,0,1){};
\node[NodeGraph]at(0,-1,1){};
\node[NodeGraph]at(0,-2,1){};
\node[NodeGraph]at(-1,0,1){};
\node[NodeGraph]at(-1,-2,1){};
\node[NodeGraph]at(0,1,0){};
\node[NodeGraph]at(0,1,-1){};
\node[NodeGraph]at(0,1,-2){};
\node[NodeGraph]at(0,1,-3){};
\node[NodeGraph]at(-1,1,0){};
\node[NodeGraph]at(-1,1,-1){};
\node[NodeGraph]at(-1,1,-3){};
\node[NodeGraph]at(0,2,0){};
\node[NodeGraph]at(0,2,-1){};
\node[NodeGraph]at(-1,2,0){};
\node[NodeGraph]at(-1,2,-1){};
\node[NodeGraph](021)at(0,2,1){};
\node[NodeGraph]at(-1,2,1){};
\node[NodeGraph]at(1,0,0){};
\node[NodeGraph]at(1,0,-1){};
\node[NodeGraph]at(1,0,-2){};
\node[NodeGraph]at(1,0,-3){};
\node[NodeGraph]at(1,-1,0){};
\node[NodeGraph]at(1,-1,-2){};
\node[NodeGraph]at(1,-1,-3){};
\node[NodeGraph]at(1,-2,0){};
\node[NodeGraph]at(1,-2,-3){};
\node[NodeGraph]at(1,0,1){};
\node[NodeGraph]at(1,-1,1){};
\node[NodeGraph]at(1,-2,1){};
\node[NodeGraph]at(2,0,0){};
\node[NodeGraph]at(2,0,-1){};
\node[NodeGraph]at(2,0,-2){};
\node[NodeGraph]at(2,0,-3){};
\node[NodeGraph]at(2,-1,0){};
\node[NodeGraph]at(2,-1,-2){};
\node[NodeGraph]at(2,-1,-3){};
\node[NodeGraph]at(2,1,0){};
\node[NodeGraph]at(2,1,-1){};
\node[NodeGraph]at(2,1,-2){};
\node[NodeGraph]at(2,1,-3){};
\node[NodeGraph]at(3,0,0){};
\node[NodeGraph]at(3,0,-1){};
\node[NodeGraph]at(3,0,-2){};
\node[NodeGraph]at(3,-1,0){};
\node[NodeGraph]at(3,-1,-2){};
\node[NodeGraph](301)at(3,0,1){};
\node[NodeGraph]at(3,-1,1){};
\node[NodeGraph]at(3,1,0){};
\node[NodeGraph]at(3,1,-1){};
\node[NodeGraph]at(3,1,-2){};
\node[NodeGraph](320)at(3,2,0){};
\node[NodeGraph]at(3,2,-1){};
\node[NodeGraph](321)at(3,2,1){};]
\node[BigNodeLabeldGraph,left of=000]{$(0,0,0)$};
\node[BigNodeLabeldGraph,below of=321]{$(3,2,1)$};
\node[BigNodeLabeldGraph,right of=021]{$(0,2,1)$};
\node[BigNodeLabeldGraph,right of=301]{$(3,0,1)$};
\node[BigNodeLabeldGraph,left of=320]{$(3,2,0)$};
\node[BigNodeLabeldGraph,above of=-1-2-3]{$(\bar{1},\bar{2},\bar{3})$};
\node[BigNodeLabeldGraph,left of=-10-3]{$(\bar{1},0,\bar{3})$};
\node[BigNodeLabeldGraph,right of=-1-20]{$(\bar{1},\bar{2},0)$};
\end{tikzpicture}}
\caption{$\CubicReal(\ccc(4))$.}
\label{DessinCC4}
\end{figure}
\smallbreak


We know that each cell-wing contains at least $2^{n-1}$ cubic coordinates on the edges. Now, let us show that it is possible to associate bijectively each cell-wing to a synchronized cubic coordinate. 
\smallbreak

Let $n \geq 1$ and $\langle c^{\Outp},c^{\Inp} \rangle$ be a cell-wing of dimension $n-1$ and $\gamma$ be the map defined by
\begin{equation}
\gamma (c^{\Outp}_{i}, c^{\Inp}_{i}) := 
\begin{cases}
    c^{\Outp}_{i} & \mbox{ if } c^{\Outp}_{i}<0,\\
    c^{\Inp}_{i} & \mbox{ if } c^{\Outp}_{i}\geq 0,
\end{cases}
\end{equation}
for all $i\in [n-1]$. Note that the components returned by the map $\gamma$ are never zero.
Let denote by $(u^{\Outp}_i,v^{\Outp}_{i+1})$ (resp.\ $(u^{\Inp}_i,v^{\Inp}_{i+1})$) the pair of letters corresponding to $c^{\Outp}_i$ (resp.\ $c^{\Inp}_i$) by the map $\phi$, for any $i\in[n-1]$. Thus, the map $\gamma$ becomes
\begin{equation}
\gamma (c^{\Outp}_{i}, c^{\Inp}_{i}) := 
\begin{cases}
    -v^{\Outp}_{i+1} & \mbox{ if } c^{\Outp}_{i}<0,\\
    u^{\Inp}_{i} & \mbox{ if } c^{\Outp}_{i}\geq 0.
\end{cases}
\end{equation}
\smallbreak

Let $\Gamma$ be the map defined by
\begin{align}
\Gamma \langle c^{\Outp},c^{\Inp} \rangle := (\gamma(c^{\Outp}_{1}, c^{\Inp}_{1}), \gamma(c^{\Outp}_{2}, c^{\Inp}_{2}),\dots,\gamma(c^{\Outp}_{n-1}, c^{\Inp}_{n-1})).
\end{align}
\smallbreak

For instance, the cell-wing $\langle (0,-1,1,-1,-5,0,1,-1,-3),(1,0,2,0,-4,3,2,0,-2) \rangle$ is sent by $\Gamma$ to $(1,-1,2,-1,-5,3,2,-1,-3)$.
\smallbreak

\begin{Theorem}\label{bigGam}
For any $n \geq 1$, the map $\Gamma$ is a bijection from the set of cell-wings of dimension $n-1$ to $\ccs$.
\end{Theorem}
\smallbreak

\begin{proof}
The components of $\Gamma \langle c^{\Outp},c^{\Inp} \rangle$ belong to either $c^{\Outp}$ or $c^{\Inp}$. In both cases, it is a non-zero component. According to Theorem~\ref{nombredesommetdanscellule}, $\Gamma \langle c^{\Outp},c^{\Inp} \rangle$ is therefore a cubic coordinate of size $n$. Moreover, this cubic coordinate is synchronized because none of its components is null.
\smallbreak

Let $\langle c^{\Outp},c^{\Inp} \rangle$ and $\langle e^{\Outp},e^{\Inp} \rangle$ be two cell-wings of dimension $n-1$ such that $\Gamma \langle c^{\Outp},c^{\Inp} \rangle = \Gamma \langle e^{\Outp},e^{\Inp} \rangle$.
Let us denote $(u^{\Outp}_i,v^{\Outp}_{i+1})$ (resp.\ $(u^{\Inp}_i,v^{\Inp}_{i+1})$) the pair of letters corresponding to $c^{\Outp}_i$ (resp.\ $c^{\Inp}_i$) and $(x^{\Outp}_i,y^{\Outp}_{i+1})$ (resp.\ $(x^{\Inp}_i,y^{\Inp}_{i+1})$) the pair of letters corresponding to $e^{\Outp}_i$ (resp.\ $e^{\Inp}_i$) by the map $\phi$, for all $i\in[n-1]$.
\smallbreak

To suppose that $\Gamma \langle c^{\Outp},c^{\Inp} \rangle = \Gamma \langle e^{\Outp},e^{\Inp} \rangle$ is equivalent to suppose that for all $i\in[n-1]$, $\gamma (c^{\Outp}_i, c^{\Inp}_i) = \gamma (e^{\Outp}_i, e^{\Inp}_i)$. 
The map $\Gamma$ is injective if, for every $i\in[n-1]$, $c^{\Outp}_i = e^{\Outp}_i$ and $c^{\Inp}_i = e^{\Inp}_i$.
Suppose that there is some index $i$ such that $c^{\Outp}_i \neq e^{\Outp}_i$ or $c^{\Inp}_i \neq e^{\Inp}_i$, and we take the smallest such index.
Then, two cases have to be considered: either $\gamma (c^{\Outp}_i, c^{\Inp}_i) = u^{\Inp}_{i}$ or $\gamma (c^{\Outp}_i, c^{\Inp}_i) = -v^{\Outp}_{i+1}$. 
\smallbreak

\begin{enumerate}[label={(\arabic*)}]
\item \label{injGam1} Suppose that $\gamma (c^{\Outp}_i, c^{\Inp}_i) = u^{\Inp}_{i}$.

\begin{itemize}
\item In this case, $\gamma (e^{\Outp}_i, e^{\Inp}_i) = x^{\Inp}_{i}$ and $u^{\Inp}_{i} = x^{\Inp}_{i}$. Moreover, since $u^{\Inp}_i \ne 0$ (resp.\ $x^{\Inp}_i \ne 0$), then necessarily $v^{\Inp}_{i+1} = 0$ (resp.\ $y^{\Inp}_{i+1} = 0$). Therefore, $c^{\Inp}_i = e^{\Inp}_i$.

\item On the other hand , the fact that $u^{\Inp}_i > 0$ (resp.\ $x^{\Inp}_i > 0$) implies by Lemma~\ref{lemmeutilepartout} that $0\leq u^{\Outp}_i < u^{\Inp}_i$ and $v^{\Outp}_{i+1} = 0$ (resp.\ $0\leq x^{\Outp}_i < x^{\Inp}_i$ and $y^{\Outp}_{i+1} = 0$). Thus, one has $v^{\Outp}_{i+1} = y^{\Outp}_{i+1}$. 
Therefore, the only way for the hypothesis to be true is that $u^{\Outp}_i \neq x^{\Outp}_i$. 

Without loss of generality, suppose that $u^{\Outp}_i < x^{\Outp}_i$. 
By the definition of the covering map, one has $x^{\Outp}_i < x^{\Inp}_i$. This implies, in addition to the hypothesis that $x^{\Inp}_i = u^{\Inp}_i$, that $u^{\Outp}_i < x^{\Outp}_i < u^{\Inp}_i$.

Let $c := \Uparrow_{i+1} (c^{\Outp})$ and $e := \Uparrow_{i+1} (e^{\Outp})$, both cubic coordinates by Lemma~\ref{existeMax}. 
By construction, $c_j = c^{\Outp}_j$ (resp. $e_j = e^{\Outp}_j$)  for all $j \in [i]$ and $c_k = c^{\Inp}_k$ (resp. $e_k = e^{\Inp}_k$) for all $k \in [i+1, n-1]$. 

By minimality of $i$, we have that $c_j = e_j$ for all $j \in [i]$. Moreover, by the hypothesis that $\Gamma \langle c^{\Outp},c^{\Inp} \rangle = \Gamma \langle e^{\Outp},e^{\Inp} \rangle$, we have that $u^{\Inp}_k = x^{\Inp}_k$ for $k \in [i+1, n-1]$. 
Indeed, if $u^{\Inp}_k > 0$ (resp. $x^{\Inp}_k > 0$) then necessarily $\gamma (c^{\Outp}_k, c^{\Inp}_k) = u^{\Inp}_{k}$ (resp. $\gamma (e^{\Outp}_k, e^{\Inp}_k) = x^{\Inp}_{k}$) and so $u^{\Inp}_k = x^{\Inp}_k$. Otherwise, $u^{\Inp}_k = x^{\Inp}_k = 0$.
Note that because we know nothing about $v^{\Inp}_k$ and $y^{\Inp}_k$ for $k \in [i+2, n]$, we cannot say that $\uparrow_i (c)$ and $\uparrow_i (e)$ are equal. 

Now, let $c'$ be a tuple such that $c'_i = x^{\Outp}_i$ and $c'_j = c_j$ for all $j \ne i$ and let $(u',v')$ the pair of words corresponding to $c'$ by the map $\phi$. Let us show that $c'$ is a cubic coordinate.

By construction, since the word $v'$ is the dual Tamari diagram of $c$, $v'$ is a dual Tamari diagram. Likewise, since the word $u'$ is the Tamari diagram of $\uparrow_i (e)$, $u'$ is a Tamari diagram. 

Moreover, we know that between $c$, $c'$ and $\uparrow_i(c)$, only one positive letter changes, with $c_i = u^{\Outp}_i$, $c'_i = x^{\Outp}_i$ and  $\uparrow c_i = u^{\Inp}_i$, and we have established that $u^{\Outp}_i < x^{\Outp}_i < u^{\Inp}_i$. Since the letter $u^{\Inp}_i$ satisfies the compatibility condition with the letters of $v^{\Inp}$ in $\uparrow_i (c)$, then all letter lower in position $i$ satisfies this condition as well. Therefore, $u'$ and $v'$ are compatible and $c'$ is a cubic coordinate distinct from $c$ and $\uparrow_i (c)$ such that $c \Leq c' \Leq \uparrow_i (c)$.

However, if $c'$ is a cubic coordinate, then by the definition of the covering map $\uparrow c_i := u^{\Inp}_i = x^{\Outp}_i$, and so $\uparrow_i (c) := \Uparrow_{i} (c^{\Outp}) = c'$. This is not possible with the assumption that $u^{\Inp}_i = x^{\Inp}_i$, and so that $\gamma (c^{\Outp}_i, c^{\Inp}_i) = \gamma (e^{\Outp}_i, e^{\Inp}_i)$.
\end{itemize}

\item \label{injGam2} Suppose that $\gamma (c^{\Outp}_i, c^{\Inp}_i) = -v^{\Outp}_{i+1}$. In this case $\gamma (e^{\Outp}_i, e^{\Inp}_i) = -y^{\Outp}_{i+1}$ and $v^{\Outp}_{i+1} = y^{\Outp}_{i+1}$.
By rephrasing the arguments of the case~\ref{injGam1} for the dual, we show that $c^{\Outp}_i = e^{\Outp}_i$ and $c^{\Inp}_i = e^{\Inp}_i$.
\end{enumerate}

This shows that the map $\Gamma$ is injective.
\smallbreak

Now let us show that the cardinal of the set of cell-wings of dimension $n-1$ is equal to the cardinal of $\ccs$.
Recall that the set of cells of size $n$ is exactly $\OutputWings(\cc)$. Furthermore, by the poset isomorphism $\psi$ we know that these elements are the Tamari intervals having $n-1$ elements covering in the Tamari interval lattices. 
In~\cite{Cha18} Chapoton shows that the set of these Tamari intervals has the same cardinal as the set of synchronized Tamari intervals (see Theorem~2.1 and Theorem~2.3 from~\cite{Cha18}). Finally, Proposition~\ref{ditsyn=tisyn} allows us to conclude that the cardinal of $\ccs$ and the cardinal of the set of cell-wings of dimension $n-1$ are equal.
Thus, the map $\Gamma$ is bijective.
\end{proof}
\smallbreak

Let us also defined the map $\bar{\gamma}$ by
\begin{equation}
\bar{\gamma} (c^{\Outp}_{i}, c^{\Inp}_{i}) := 
\begin{cases}
    c^{\Inp}_{i} & \mbox{ if } c^{\Outp}_{i}<0,\\
    c^{\Outp}_{i} & \mbox{ if } c^{\Outp}_{i}\geq 0,
\end{cases}
\end{equation}
for all $i\in[n-1]$. Then $\bar{\Gamma}$ is defined by
\begin{align}
\bar{\Gamma} \langle c^{\Outp},c^{\Inp} \rangle := (\bar{\gamma}(c^{\Outp}_{1}, c^{\Inp}_{1}), \bar{\gamma}(c^{\Outp}_{2}, c^{\Inp}_{2}),\dots,\bar{\gamma}(c^{\Outp}_{n-1}, c^{\Inp}_{n-1})).
\end{align}
\smallbreak

By Theorem~\ref{nombredesommetdanscellule}, $\bar{\Gamma} \langle c^{\Outp},c^{\Inp} \rangle$ is a cubic coordinate belonging to $\langle c^{\Outp},c^{\Inp} \rangle$, called~\Def{opposite cubic coordinate}. For the synchronized cubic coordinate $c$ associated with $\langle c^{\Outp},c^{\Inp} \rangle$ by $\Gamma$, denote $c^{op}$ the opposite cubic coordinate. All the components of $c^{op}$ are different from those of $c$, and these differences are the greatest possible. For any synchronized cubic coordinate $c$, such a cubic coordinate $c^{op}$ always exists and is unique.
\smallbreak

Note that the map $\Gamma$ only returns the positive components of $c^{\Inp}$ and the negative components of $c^{\Outp}$. Conversely, the map $\bar{\Gamma}$ returns the positive components of $c^{\Outp}$ and the negative components of $c^{\Inp}$. We already know that the latter combination is always possible for any comparable cubic coordinates according to Lemma~\ref{lem:baspetitethautgrand}. On the other hand, this is not the case for the first mentioned combination.
\smallbreak

\subsection{Volume of $\CubicReal(\ccc)$}

Now let us take a closer look at the geometry of the cubic realization. We already know that there are at least $2^{n-1}$ cubic coordinates forming an outline of each cell-wing. The following notions will allow us to say more.
\smallbreak

A point $x$ of $\R^{n-1}$ is \Def{inside} a cell $\Angle{c, c'}$ if, for any $i \in [n-1]$, $c_i \ne c'_i$ implies $c_i < x_i < c'_i$.  A cell $\Angle{c, c'}$ is \Def{pure} if there is no cubic coordinate inside $\Angle{c, c'}$. The \Def{volume} $\Volume \Angle{c, c'}$ of $\Angle{c, c'}$ is its volume as an orthotope and it satisfies
\begin{equation}
    \Volume \Angle{c, c'} = \prod_{i \in \DiffIndices(c, c')} (c'_i - c_i).
\end{equation}
\smallbreak

\begin{Lemma}\label{cellulevide}
Let $n \geq 1$ and $\langle c^{\Outp},c^{\Inp} \rangle$ be a cell-wing of dimension $n-1$. The cell $\langle c^{\Outp},c^{\Inp} \rangle$ is pure. 
\end{Lemma}
\smallbreak

\begin{proof} 
Suppose there is a cubic coordinate $c$ such that $c^{\Outp}_{i} < c_{i} < c^{\Inp}_{i}$ for all $i\in[n-1]$. By Lemma~\ref{lemmequifaitgagnerdutps} we know that if $c^{\Outp}_{i}<0$, then $c^{\Inp}_{i}\leq 0$ and if $c^{\Outp}_{i}\geq 0$, then $c^{\Inp}_{i}>0$. However, since $c^{\Outp}_{i} < c_i < c^{\Inp}_{i}$, then $c_{i}$ is different from $0$. In the end, if such a cubic coordinate $c$ exists, it would be synchronized.
But then, there would be a cubic coordinate both synchronized and internal by hypothesis. This is impossible according to Proposition~\ref{synext}.
\end{proof}
\smallbreak

We showed with Theorem~\ref{nombredesommetdanscellule} that each cell-wing contains at least $2^{n-1}$ cubic coordinates. By Lemma~\ref{cellulevide}, we know that each cell-wing $\langle c^{\Outp},c^{\Inp} \rangle$ is pure, and then has only cubic coordinates on its border.
\smallbreak 

Let $n \geq 1$ and $\langle c^{\Outp},c^{\Inp} \rangle$ be a cell-wing of dimension $n-1$. Since between $c^{\Outp}$ and $c^{\Inp}$ all components are different, one has $\DiffIndices(c^{\Outp},c^{\Inp}) = n-1$, and so the volume of $\langle c^{\Outp},c^{\Inp} \rangle$ satisfies
\begin{equation}
\Volume \langle c^{\Outp},c^{\Inp} \rangle = \prod_{i=1}^{n-1} (c^{\Inp}_i - c^{\Outp}_i).
\end{equation}
\smallbreak

Let us denote by $c^0$ the cubic coordinate such that $c^0_i = 0$ for any $i\in[n-1]$. To compute $\Volume \langle c^{\Outp},c^{\Inp} \rangle$ from the synchronized cubic coordinate $c$ associated by $\Gamma$, we must first compute the volume of the cell formed by $c^0$ and $c$. 
\smallbreak

By Lemma~\ref{lemmequifaitgagnerdutps}, any cell-wing is included in an $\epsilon$-region of the $c^0$ cubic coordinate. This means that no cell-wing can be cut by a line passing by
the origin $c^0$ and a cubic coordinate of the form $\Par{0, \dots, 0, 1, 0, \dots, 0}$ or $\Par{0, \dots, 0, -1, 0, \dots, 0}$.
\smallbreak

According to Lemma~\ref{lemmeutilepartout}, for any cubic coordinate, replacing any component by $0$ gives a cubic coordinate. In other words, for any cubic coordinate $c$, there are $n-1$ cubic coordinates related to $c$ which are its projections on the lines passing by $c^0$ and a cubic coordinate of the form $\Par{0, \dots, 0, 1, 0, \dots, 0}$ or $\Par{0, \dots, 0, -1, 0, \dots, 0}$. 
Therefore, even if $c^{0}$ and $c$ are not comparable, we consider the cell, denoted by $\Angle{c}$, between $c^{0}$ and $c$, such that the volume of this cell satisfies
\begin{equation}
    \Volume \Angle{c} = \prod_{i \in \DiffIndices(c, c^0)} |c_i|.
\end{equation}
Note that the dimension of a cell is less than or equal to $n-1$. Moreover, $\Angle{c}$ can be no-pure, and may even contain other cells of the same dimension.
\smallbreak

By the map $\Gamma$, the components of the synchronized cubic coordinate $c$ of the cell-wing $\langle c^{\Outp},c^{\Inp} \rangle$ are the greatest in absolute value between $c^{\Outp}$ and $c^{\Inp}$. Therefore, in the cell-wing $\langle c^{\Outp},c^{\Inp} \rangle$, $c$ is the furthest cubic coordinate from $c^0$. In particular, $\Angle{c}$ contains the cell-wing $\langle c^{\Outp},c^{\Inp} \rangle$ and the dimension of $\Angle{c}$ is~$n-1$.
\smallbreak

Let $n\geq 0$ and $c\in\ccs$. 
Since by the definition, all components of $c$ are different from $0$, one has $\DiffIndices(c, c^0) = n-1$. Therefore,
\begin{equation}
    \Volume \Angle{c} = \prod^{n-1}_{i=1} |c_{i}|.
\end{equation}
\smallbreak

Let us endow the set $\ccs$ with the partial order $\Leq_{\mathrm{s}}$ such that for $c,c'\in\ccs$ one has $c' \Leq_{\mathrm{s}} c$ if $c'_i$ and $c_i$ have the same sign and $|c'_i| \leq |c_i|$ for any $i\in[n-1]$.
\smallbreak

\begin{Lemma}\label{lemmeunioncell}
For any $n \geq 1$, let $\langle c^{\Outp},c^{\Inp} \rangle$ be a cell-wing of dimension $n-1$, and $c := \Gamma \langle c^{\Outp},c^{\Inp} \rangle$. For any $x\in \mathbb{R}^{n-1}$ such that $x\in \Angle{c}$, if $x\notin \langle c^{\Outp},c^{\Inp} \rangle$, then there is $c'\in \ccs$ different from $c$ such that $c' \Leq_{\mathrm{s}} c$ and $x \in \Angle{c'}$.
\end{Lemma}
\smallbreak

\begin{proof}
Let $c^{op}$ be the opposite cubic coordinate of $c$. Since $x\notin \langle c^{\Outp},c^{\Inp} \rangle$ and $x\in \Angle{c}$, then necessarily $c^{op} \ne c^0$. 
For the same reasons, there is an index $i$ such that $|x_i| < |c^{op}_i|$ where $c^{op}_i \ne 0$.
Let us build from such index $i$ the $(n-1)$-tuple $\nabla c$ such that $\nabla c_i = c^{op}_i$ and $\nabla c_j = c_j$ for all $j\ne i$. According to Theorem~\ref{nombredesommetdanscellule}, $\nabla c$ is a cubic coordinate and belongs to the cell-wing $\langle c^{\Outp},c^{\Inp} \rangle$. Also, $\nabla c$ is a synchronized cubic coordinate which satisfies $\nabla c \Leq_{\mathrm{s}} c$ and which is different from $c$. We can then associate to $\nabla c$ a cell, which is strictly included in $\Angle{c}$. Then $x \in \Angle{\nabla c}$. 
\end{proof}
\smallbreak

Since by Lemma~\ref{cellulevide} all cell-wings are pure, Lemma~\ref{lemmeunioncell} implies that $\Angle{c} \subseteq \coprod_{c'\Leq_{\mathrm{s}} c}\Gamma^{-1}(c')$, and since the reciprocal inclusion is obvious, one has the following result.
\smallbreak

\begin{Lemma}\label{prop:unioncell}
Let $n \geq 0$ and $c\in\ccs$. Then
\begin{equation}
\Angle{c} = \coprod_{c'\Leq_{\mathrm{s}} c}\Gamma^{-1}(c').
\end{equation}
\end{Lemma}
\smallbreak

Let $n \geq 0$ and $c\in\ccs$. The \Def{synchronized volume} of $c$ is defined by
\begin{equation}\label{equ:volume sych}
\SynVolume(c) := \Volume \Angle{c} - \sum_{\substack{c'\Leq_{\mathrm{s}} c \\ c'\ne c}} \SynVolume(c').
\end{equation}
Note that~\eqref{equ:volume sych} is a Möbius inversion~\cite{Sta12}.
\smallbreak

\begin{Proposition}\label{VolumeTotal}
Let $n \geq 1$ and $\langle c^{\Outp},c^{\Inp} \rangle$ be a cell-wing of dimension $n-1$. By setting $c := \Gamma\langle c^{\Outp},c^{\Inp} \rangle$, we have
\begin{equation}\label{equ:volume}
\Volume \langle c^{\Outp},c^{\Inp} \rangle = \SynVolume(c).
\end{equation}
\end{Proposition}
\smallbreak

\begin{proof}
This is a consequence of Lemma~\ref{prop:unioncell} and of~\eqref{equ:volume sych}.
\end{proof}
\smallbreak

With Proposition~\ref{VolumeTotal} we are able to compute, for any $n \geq 0$, the volume of $\CubicReal(\cc)$ depending on synchronized cubic coordinates,
\begin{equation}
\Volume (\CubicReal(\cc)) = \sum_{c \in \ccs} \SynVolume(c).
\end{equation}


\subsection{EL-shellability}

In~\cite{BW96} and~\cite{BW97}, Björner and Wachs
generalized the method of labellings of
the cover relations of graded posets to the case of non-graded posets. In particular, they
showed the EL-shellability of the Tamari poset~\cite{BW97}.
\smallbreak

Let $\PosetP$ be a bounded poset and $\Lambda$ be a poset, and $\lambda : \Covered_{\PosetP} \to \Lambda$ be a map. For any saturated chain $\Par{x^{(1)}, \dots, x^{(k)}}$ of
$\PosetP$, we set
\begin{equation}
    \lambda\Par{x^{(1)}, \dots, x^{(k)}}
    :=
    \Par{\lambda\Par{x^{(1)}, x^{(2)}}, \dots,
    \lambda\Par{x^{(k - 1)}, x^{(k)}}}.
\end{equation}
We say that a saturated chain of $\PosetP$ is \Def{$\lambda$-increasing}
(resp.\ \Def{$\lambda$-decreasing}) if its image by $\lambda$ is an
increasing (resp.\ decreasing) word for the order relation
$\Leq_{\Lambda}$. We say also that a saturated chain $\Par{x^{(1)},
\dots, x^{(k)}}$ of $\PosetP$ is \Def{$\lambda$-smaller} than a
saturated chain $\Par{y^{(1)}, \dots, y^{(k)}}$ of $\PosetP$ if $\lambda\Par{x^{(1)}, \dots, x^{(k)}}$ is smaller than $\lambda\Par{y^{(1)}, \dots, y^{(k)}}$ for the
lexicographic order induced by $\Leq_{\Lambda}$. The map $\lambda$ is called \Def{EL-labeling} (edge lexicographic labeling) of $\PosetP$ if for any $x, y \in \PosetP$ satisfying $x \Leq_\PosetP y$, there is exactly one $\lambda$-increasing saturated chain from $x$ to~$y$, and this chain is $\lambda$-minimal among all saturated chains from $x$ to $y$. Any bounded poset that admits an EL-labeling is \Def{EL-shellable}~\cite{BW96,BW97}.
\smallbreak

The EL-shellability of a poset $\PosetP$ implies several topological and
order theoretical properties of the associated order complex
$\Delta(\PosetP)$ built from $\PosetP$. Recall that the faces of this
simplicial complex are all the chains of $\PosetP$. Moreover, if $\PosetP$ has at most one
$\lambda$-decreasing chain between any pair of elements, then the Möbius function of $\PosetP$ takes values in $\{-1, 0, 1\}$. In this case, the simplicial complex associated with each open interval of $\PosetP$ is either contractible or has the homotopy type of
a sphere~\cite{BW97}.
\smallbreak

For the sequel, we set $\Lambda$ as the poset $\Z^3$ wherein elements are ordered lexicographically.
Let $(c,c') \in \lessdot$ such that, for $i\in [n-1]$, $c_i < c'_i$, and let $\lambda : \lessdot \rightarrow \mathbb{Z}^3$ be the map defined by
\begin{equation}\label{equ: lambda}
 \lambda(c,c') := (\varepsilon, i , c_i),
\end{equation}
where $\varepsilon :=
\begin{cases}
-1 &\mbox{if } c_i <0,\\
1 &\mbox{else.}
\end{cases}$
\smallbreak

Note that by Proposition~\ref{prop:cubicoordinatecovering}, the index $i$ such that $c_i < c'_i$ is unique.
\smallbreak

\begin{Theorem}\label{ccshell}
For any $n \geq 0$, the map $\lambda$ is an EL-labeling of $\cc$. Moreover, there is at most one $\lambda$-decreasing chain between any pair of elements of~$\cc$.
\end{Theorem}
\smallbreak

\begin{proof}
Let $c,c'\in\cc$ such that $c\Leq c'$. 
By Lemma~\ref{lem:baspetitethautgrand}, there is a cubic coordinate $c''$ such that $u'' = u$ and $v'' = v'$ with $(u'', v'') := \phi(c'')$. 
Let 
\begin{equation}
\DiffIndexes^{-}\Par{c, c''} = \{d_1, d_2, \dots, d_r\}
\end{equation} 
with $d_{k-1} < d_{k}$ for all $k \in [2,r]$, and 
\begin{equation}
\DiffIndexes^{+}\Par{c'', c'} = \{d'_{1}, d'_{2}, \dots, d'_s\},
\end{equation}
with $d'_{k-1} < d'_{k}$ for all $k \in [2,s]$. 
\smallbreak

By Lemma~\ref{lem:mêmebasethautdifferent}, there is a chain between $c$ and $c''$
\begin{equation}\label{equ:chain}
\Par{c, c^{(1)}, \dots, c^{(r-1)}, c^{(r)}= c''},
\end{equation}
where, for $k \in [r]$, $c^{(k)}$ be a cubic coordinate obtained by replacing in $c$ all the components $c_{d_i}$ by the components $c''_{d_i}$ for $i \in [k]$.
\smallbreak

By Lemma~\ref{lem:mêmehautetbasdifferent}, there is a chain between $c''$ and $c'$
\begin{equation}\label{equ:chain2}
\Par{c'', c'^{(1)}, \dots, c'^{(s-1)}, c'^{(s)}= c'},
\end{equation}
where, for $k \in [s]$, $c'^{(k)}$ be a cubic coordinate obtained by replacing in $c''$ all the components $c''_{d_i}$ by the components $c'_{d_i}$ for $i \in [k]$. 
\smallbreak

Let us consider the chain obtained by concatenating the two chains~\eqref{equ:chain} and \eqref{equ:chain2}.
Since in this chain only one component differs between two consecutive cubic coordinates, a saturated chain $\mu$ can be constructed by considering all the cubic coordinates between them.
For both chains~\eqref{equ:chain} and \eqref{equ:chain2}, the components are independently increasing one by one from the left to the right. By construction, it implies that $\mu$ is $\lambda$-increasing for the lexicographic order induced by~\eqref{equ: lambda}.
\smallbreak

Moreover, any other choice of saturated chain between $c$ and $c'$ implies choosing, at a certain step $k$, a greater label for the lexicographical order than the label $(\varepsilon, k, c_k)$ of $\mu$, and then having to choose the label $(\varepsilon, k, c_k)$ afterwards.
Thus, in addition to being $\lambda$-increasing, the saturated chain $\mu$ is unique and is $\lambda$-minimal among all saturated chains from $c$ to $c'$.
\smallbreak

If a saturated chain $\lambda$-decreasing exists between $c$ and $c'$, it is built by first changing the different and negative components between $c$ and $c''$ from right to left, and then changing the different and positive components between $c''$ and $c'$ from right to left.
For the same reason that any saturated $\lambda$-increasing chain is unique for any interval, if it exists, the $\lambda$-decreasing chain is also unique.
\end{proof}
\smallbreak

For instance, in Figure~\ref{RealCube}, the $\lambda$-increasing saturated chain between $(-1,-2)$ and $(2,1)$ is the chain 
\begin{equation}
\Par{(-1,-2), (0,-2), (0,-1), (0,0), (1,0), (2,0), (2,1)},
\end{equation}
and 
\begin{equation}
\lambda\Par{(-1,-2), \dots, (2,1)} = \Par{(-1, 1, -1), (-1, 2,-2), (-1, 2, -1), (1, 1, 0), (1, 1, 1), (1, 2, 0)}.
\end{equation}
\vspace{1cm}

\section*{Acknowledgements}

The author would like to thank the anonymous reviewer for all his good advices, which contributed to the improvement of this article. The author would also like to thank Frédéric Chapoton, Samuele Giraudo, and Baptiste Rognerud for the numerous discussions and their suggestions.
\bigbreak

My manuscript has no associated data.

\bibliographystyle{alpha}
\bibliography{Bibliography}

\begin{thebibliography}{BMFPR12}

\bibitem[BB09]{BB09}
O.~Bernardi and N.~Bonichon.
\newblock Intervals in {C}atalan lattices and realizers of triangulations.
\newblock {\em J. Combin. Theory Ser. A}, 116(1):55--75, 2009.

\bibitem[BMFPR12]{BMFPR12}
M.~Bousquet-M\'elou, \'E. Fusy, and L.-F. Pr\'{e}ville-Ratelle.
\newblock The number of intervals in the $m$-{T}amari lattices.
\newblock {\em Electronic J. Combin.}, 18(2), 2012.

\bibitem[BPR12]{BPR12}
F.~Bergeron and L.-F. Pr\'{e}ville-Ratelle.
\newblock Higher trivariate diagonal harmonics via generalized {T}amari posets.
\newblock {\em J. Combin.}, (3):317--341, 2012.

\bibitem[BW96]{BW96}
A.~Bj\"{o}rner and M.~L. Wachs.
\newblock Shellable nonpure complexes and posets. {I}.
\newblock {\em Trans. Amer. Math. Soc.}, 348(4):1299--1327, 1996.

\bibitem[BW97]{BW97}
A.~Bj\"{o}rner and M.~L. Wachs.
\newblock Shellable nonpure complexes and posets. {II}.
\newblock {\em Trans. Amer. Math. Soc.}, 349(10):3945--3975, 1997.

\bibitem[Cha06]{Cha06}
F.~Chapoton.
\newblock Sur le nombre d'intervalles dans les treillis de {T}amari.
\newblock {\em S\'{e}m. Lothar. Combin.}, 55:Art. B55f, 18, 2006.

\bibitem[Cha18]{Cha18}
F.~Chapoton.
\newblock Une note sur les intervalles de {T}amari.
\newblock {\em Ann. Math. Blaise Pascal}, 25(2):299--314, 2018.

\bibitem[Com19]{Com19}
C.~Combe.
\newblock Cubic realizations of {T}amari interval lattices.
\newblock {\em S\'{e}m. Lothar. Combin.}, 82B.23:12 pp., 2019.

\bibitem[CP15]{CP15}
G.~Châtel and V.~Pons.
\newblock {Counting smaller elements in the Tamari and $m$-Tamari lattices}.
\newblock {\em {J. Combin. Theory Ser. A}}, 134: 58–97, 2015.

\bibitem[Fan18]{Fan18}
W.~Fang.
\newblock Planar triangulations, bridgeless planar maps and {T}amari intervals.
\newblock {\em European J. Combin.}, 70:75--91, 2018.

\bibitem[FPR17]{FPR17}
W.~Fang and L.-F. Pr\'{e}ville-Ratelle.
\newblock The enumeration of generalized {T}amari intervals.
\newblock {\em European J. Combin.}, 61:69--84, 2017.

\bibitem[Gir12]{Gir12}
S.~Giraudo.
\newblock Algebraic and combinatorial structures on pairs of twin binary trees.
\newblock {\em J. Algebra}, 360:115--157, 2012.

\bibitem[HT72]{HT72}
S.~Huang and D.~Tamari.
\newblock Problems of associativity: A simple proof for the lattice property of
  systems ordered by a semi-associative law.
\newblock {\em J. Combinatorial Theory Ser. A}, 13:7-13, 1972.

\bibitem[Lod11]{Lod11}
J.-L. Loday.
\newblock The diagonal of the {S}tasheff polytope.
\newblock In {\em Higher structures in geometry and physics}, volume 287 of
  {\em Progr. Math.}, pages 269--292. Birkh\"{a}user/Springer, New York, 2011.

\bibitem[MP90]{MP90}
E.~C. Milner and M.~Pouzet.
\newblock A note on the dimension of a poset.
\newblock {\em Order}, 7(1):101--102, 1990.

\bibitem[MS06]{MS06}
M.~Markl and S.~Shnider.
\newblock Associahedra, cellular {$W$}-construction and products of
  {$A_\infty$}-algebras.
\newblock {\em Trans. Amer. Math. Soc.}, 358(6):2353--2372, 2006.

\bibitem[MTTV21]{MTTV21}
N.~Masuda, H.~Thomas, A.~Tonks, and B.~Vallette.
\newblock The diagonal of the associahedra.
\newblock {\em J. \'{E}c. polytech. Math.}, 8:121--146, 2021.

\bibitem[Pal86]{Pal86}
J.~M. Pallo.
\newblock {Enumerating, ranking and unranking binary trees}.
\newblock {\em {Comput. J. 29}}, no. 2, 171–175, 1986.

\bibitem[PRV17]{PRV17}
L.-F. Pr\'eville-Ratelle and X.~Viennot.
\newblock The enumeration of generalized {T}amari intervals.
\newblock {\em Trans. Amer. Math. Soc.}, 369(7):5219--5239, 2017.

\bibitem[Rog20]{Rog20}
B.~Rognerud.
\newblock Exceptional and modern intervals of the {T}amari lattice.
\newblock {\em S\'{e}m. Lothar. Combin.}, 79:Art. B79d, 23, 2018-2020.

\bibitem[Slo]{Slo}
N.~J.~A. Sloane.
\newblock {The On-Line Encyclopedia of Integer Sequences}.
\newblock \url{https://oeis.org/}.

\bibitem[Sta12]{Sta12}
R.~P. Stanley.
\newblock {\em Enumerative combinatorics. {V}olume 1}, volume~49 of {\em
  Cambridge Studies in Advanced Mathematics}.
\newblock Cambridge University Press, Cambridge, second edition, 2012.

\bibitem[SU04]{SU04}
S.~Saneblidze and R.~Umble.
\newblock Diagonals on the permutahedra, multiplihedra and associahedra.
\newblock {\em Homology Homotopy Appl.}, 6(1):363--411, 2004.

\bibitem[Tam62]{Tam62}
D.~Tamari.
\newblock The algebra of bracketings and their enumeration.
\newblock {\em Nieuw Arch. Wisk. (3)}, 10:131--146, 1962.

\bibitem[Tro02]{Tro02}
W.~T. Trotter.
\newblock {\em Combinatorics and partially ordered sets: {D}imension theory}.
\newblock Johns Hopkins Series in the Mathematical Sciences. The Johns Hopkins
  University Press, 2002.

\end{thebibliography}

\end{document}